\documentclass[12pt]{amsart}
\usepackage{amsfonts,amsmath,amsthm,amssymb}
\usepackage{MnSymbol}
\usepackage{tikz}
\usetikzlibrary{arrows}
\usetikzlibrary{decorations.markings}
\usepackage[square]{natbib}
\setcitestyle{numbers}

\usepackage[all]{xy}
\usepackage{dsfont}
\usepackage{hyperref}

\usepackage[margin=1.1in]{geometry}

\definecolor{darkred}{rgb}{1,0,0}

\newtheorem{theorem}{Theorem}[section]
\newtheorem{thm}[theorem]{Theorem}

\newtheorem{lem}[theorem]{Lemma}

\newtheorem{prop}[theorem]{Proposition}

\newtheorem{cor}[theorem]{Corollary}
\newtheorem{conj}[theorem]{Conjecture}

\theoremstyle{remark}

\theoremstyle{definition}

\newtheorem{rmk}[theorem]{Remark}
\newtheorem{example}[theorem]{Example}
\newtheorem{defn}[theorem]{Definition}

\newtheorem{observation}[theorem]{Observation}
\newtheorem*{claim*}{Claim}
\newtheorem*{defn*}{Definition}

\newcommand{\bbc}{\mathbb{C}}

\newcommand{\bbp}{\mathbb{P}}

\newcommand{\bbr}{\mathbb{R}}

\newcommand{\bbz}{\mathbb{Z}}

\newcommand{\mca}{\mathcal{A}}
\newcommand{\mcb}{\mathcal{B}}

\newcommand{\mcf}{\mathcal{F}}

\newcommand{\mci}{\mathcal{I}}

\newcommand{\mcm}{\mathcal{M}}
\newcommand{\mcn}{\mathcal{N}}
\newcommand{\mco}{\mathcal{O}}

\newcommand{\mcr}{\mathcal{R}}
\newcommand{\mcs}{\mathcal{S}}

\newcommand{\mcw}{\mathcal{W}}

\newcommand{\mfi}{\mathfrak{I}}

\newcommand{\tX}{\tilde{X}}

\renewcommand{\aa}{\alpha}

\newcommand{\lam}{\lambda}

\newcommand{\eps}{\epsilon}

\newcommand{\SL}{\textnormal{SL}}

\newcommand{\eval}{\textnormal{eval}}
\newcommand{\Sym}{\textnormal{Sym}}
\newcommand{\sign}{\textnormal{sign}}

\newcommand{\Web}{\textnormal{Web}}

\newcommand{\la}{\langle}
\newcommand{\ra}{\rangle}

\newcommand{\Gr}{{\rm Gr}}

\def\tGr{\widetilde \Gr}
\def\tPi{\widetilde \Pi}

\newcommand{\sgn}{\textnormal{sign}}

\DeclareMathOperator{\Hom}{Hom}

\DeclareMathOperator{\Ind}{Ind}

\DeclareMathOperator{\Imm}{Imm}

\def\wt{{\rm wt}}
\newcommand{\injects}{\hookrightarrow}
\newcommand{\surjects}{\twoheadrightarrow}
\newcommand{\spann}{\textnormal{span}}

\title{From dimers to webs}
\author{Chris Fraser}
\address{Department of Mathematical Sciences, Indiana Unversity--Purdue University Indianapolis, 402 N. Blackford St., Indianapolis, 46202.}
\email{chfraser@iupui.edu}
\author{Thomas Lam}
\address{Department of Mathematics, University of Michigan, 530 Church St., Ann Arbor MI 48109.}
\email{tfylam@umich.edu}
\author{Ian Le}
\address{Perimeter Institute for Theoretical Physics, 31 Caroline Street North, Waterloo, Canada, N2L 2Y5.}
\email{ile@perimeterinstitute.ca}

\numberwithin{equation}{section}

\begin{document}
\begin{abstract}
We formulate a higher-rank version of the \emph{boundary measurement map} for weighted planar bipartite \emph{networks}  in the disk. It sends a network to a linear combination of $\SL_r$-\emph{webs}, and is built upon the $r$-fold dimer model on the network. When $r$ equals~1, our map is a reformulation of Postnikov's boundary measurement used to coordinatize positroid strata.  When $r$ equals 2 or 3, it is a reformulation of the $\SL_2$- and  $\SL_3$-\emph{web immanants} defined by the second author. The basic result is that the higher rank map factors through Postnikov's map. As an application, we deduce generators and relations for the space of $\SL_r$-webs, reproving a result of Cautis-Kamnitzer-Morrison. We establish compatibility between our map and restriction to positroid strata, and thus between webs and total positivity. 
\end{abstract}

\maketitle
%\tableofcontents
\vspace{-.3in}

%Todo: introd, duality pictures, restriction of bases, grassmann necklaces and positroid subspaces

\section{Introduction}
\label{secn:intro}
The Grassmannian $\Gr(k,n)$ of $k$-planes in $\bbc^n$ is an algebraic variety which is well-loved in combinatorics. This paper links two combinatorial tools -- briefly, \emph{dimer configurations} and \emph{webs} -- which have been used to study $\Gr(k,n)$  and its (homogeneous) coordinate ring. %$\bbc[\Gr(k,n)]$. 
Both approaches have a similar flavor -- each involves certain planar diagrams in a disk, and each relies heavily on local diagrammatic moves/relations amongst such diagrams. We show that this resemblance is not coincidental, and that these approaches are dual in a sense that we make precise. Moreover, statements on each side can be translated to the other to give meaningful consequences.

The first approach starts with a choice of \emph{network}~$N$, meaning a planar bipartite graph in the disk whose edges are weighted by nonzero complex numbers. Such a network comes with two parameters: $n$ (the number of boundary vertices) and $k$ (the \emph{excedance}, cf.~\eqref{eq:excedance}). The key operation is Postnikov's \emph{boundary measurement map} 
\begin{equation}\label{eq:introbmmap}
N \mapsto (\Delta_{I}(N))_{I \in \binom{[n]}k} \in \Gr(k,n)
\end{equation}
sending a network~$N$ to its $\binom n k$ \emph{boundary measurements}. Each boundary measurement $\Delta_I(N)$ is a complex number obtained by summing over \emph{dimer configurations of $N$} whose \emph{boundary data} is~$I$. %The numbers $k$ and $n$ are determined from~$N$ in a simple way. 
The striking feature is that for any network, the $\Delta_I(N)$ satisfy the well-known \emph{Pl\"ucker relations}, so the image of the boundary measurement is a point in the Grassmannian. 

The second approach considers a distinguished class of functions on $\Gr(r,n)$, indexed by planar diagrams known as $\SL_r$-\emph{webs}. Let us consider the space $\mcw(r,n) = \Hom_{\SL_r}((\bbc^r)^{\otimes n}, \bbc)$, i.e. the vector space of $\SL_r$-invariant multilinear function of $n$ vectors. The homogeneous coordinate ring $\bbc[\Gr(r,n)]$ can be viewed as an algebra of $\SL_r$-invariants, and $\mcw(r,n)$ is a certain subspace of $\bbc[\Gr(r,n)]$. We will call elements of $\mcw(r,n)$ \emph{tensor invariants}.

An \emph{$\SL_r$-web diagram} is a planar bipartite graph in the disk, with its edges labeled by positive integers so that the labels around each internal vertex sum to $r$. (See the left hand side of (\ref{eq:SL4WebEvaluationctd}) for an example when $r=4$. Edges with label $1$ are suppressed.) We denote by $\mcf \mcs (r)$ the vector space of formal sums of $\SL_r$-web diagrams (it is related to the \emph{free spider category} \cite{SkewHowe}). Each web diagram determines an element\footnote{in fact, the diagrams defined in this introduction only determine a tensor invariant up to a sign. The sign is determined by a procedure called a \emph{tagging} of the web.} of $\mcw(r,n)$. Intuitively, web diagrams determine tensor invariants as follows: an edge in a web diagram with label $a$ corresponds to a copy of the exterior power $\bigwedge^a(\bbc^r)$, and each interior vertex $v$ corresponds to an $\SL_r$-invariant map between the exterior powers adjacent to~$v$. The web diagram indicates how to compose these building blocks to create more complicated tensor invariants. The map $\mcf \mcs (r) \to \mcw(r,n)$ is surjective; i.e. $\SL_r$-invariants coming from webs span the space of tensor invariants. 

Webs can be used to study more general tensor invariant spaces, and other homogeneous pieces of $\bbc[\Gr(k,n)]$, but in this introduction, we will focus on a particular situation where our results are easiest to state. Thus, we restrict attention to networks $N$ whose number of boundary vertices $n$ is a multiple of its excedance $k$, i.e. $n = kr$. For these~$N$, we construct an $r$-fold boundary measurement map   
\begin{equation}\label{eq:introrbmmap}
N \mapsto \Web_r(N) \in \mcw(r,n),
\end{equation} 
sending a network to a particular formal sum of webs. Its key feature (our Theorem~\ref{thm:welldefined}) is that it factors through the boundary measurement map \eqref{eq:introbmmap}. That is -- if $N$ and $N'$ are two networks with the same boundary measurements (and hence which give the same point in the affine cone over the Grassmannian  $\tGr(k,n)$), then $\Web_r(N)  = \Web_r(N')$ as tensor invariants,  even though these will typically look different as elements of $\mcf \mcs(r)$. When $r$ is $2$ or $3$, this reduces to the construction of Temperley-Lieb immanants and web immanants  via the double-dimer and triple-dimer models studied by the second author \cite{LamDimers}.  We also show (Theorem~\ref{thm:symmetricgroupduality}) that \eqref{eq:introrbmmap} induces a natural isomorphism $\mcw(k,n) \cong \mcw(r,n)^*$. Moreover, both spaces $\mcw(k,n)$ and $\mcw(r,n)$ are irreducible $S_n$-modules, where $S_n$ acts on the tensor invariant spaces by permuting the vectors. As $S_n$-modules, we have that $\mcw(k,n) \cong \mcw(r,n)^* \otimes \epsilon$ where $\epsilon$ is the sign representation. Thus, we get a canonically defined $S_n$-equivariant pairing between $\SL_k$- and $\SL_r$-invariant spaces. For small values of $k$ and $r$, we draw the resulting ``web duality pictures'' in an Appendix. 

We will lay out the contents of this paper while mentioning the applications of our main construction. Section~\ref{secn:networks} reviews the boundary measurement map for networks via dimer configurations \cite{LamDimers}, as well as the local moves on networks which preserve the boundary measurements. We also review in this section how boundary measurements, networks, and local moves, can be used to study the \emph{positroid stratification} of the Grassmannian. Section~\ref{secn:tensors} gives the basics of tensor invariants and web diagrams. In Section~\ref{secn:main} we make our main definition, i.e. the $r$-fold boundary measurement map \eqref{eq:introrbmmap}, and introduce the closely related \emph{immanant map}. Section~\ref{secn:proofofmain} gives a self-contained proof that the $r$-fold boundary measurement factors through Postnikov's boundary measurement map. The key tool is a lemma about sign-coherence for webs that we conjecture is related to total positivity.

Section~\ref{secn:skein} discusses one of our main applications. One of the deepest results on the boundary measurement map is that if networks $N$ and $N'$ have the same boundary measurements, then $N'$ is connected to $N$ by a sequence of local moves \cite{Postnikov}. Likewise, a guiding problem in the history of web combinatorics was to find a complete set of diagrammatic moves describing the kernel of the map $\mcf \mcs (r) \to \mcw(r,n)$. This problem was solved for $r=3$ by Kuperberg \cite{Kuperberg}, studied further  by Kim \cite{KimThesis} and Morrison \cite{MorrisonThesis}, and settled for all $r$ by Cautis, Kamnitzer, and Morrison in \cite{SkewHowe}. We show that the completeness of the relations in \cite{SkewHowe} follows from Theorem~\ref{thm:welldefined} and the connectedness result for networks. We remark that the results in \cite{SkewHowe} hold in greater generality than ours -- those authors work in a quantum deformation of our setting. We are hopeful that an understanding of this will lead to a quantum deformation of the dimer model.

%In small dimensions ($r = 2,3$), there is an elegant choice of a \emph{basis} consisting of \emph{non-elliptic} webs. It is closely related to Luzstig's dual canonical basis, and to the set of \emph{cluster monomials} in the \emph{cluster algebra structure} on the Grassmannian \cite{Tensors}. Things are more challenging when $r>3$, as webs are merely a distinguished spanning set. Even in these cases, we expect that webs, the dual canonical basis, and cluster monomials should all be close cousins. 

In Section~\ref{secn:positroids}, we make a connection between webs and \emph{positroid varieties}. Here, we extend the results of  Sections~\ref{secn:main} and~\ref{secn:proofofmain} to the context of positroids. The \emph{totally nonnegative Grassmannian} $\Gr(k,n)_{\geq 0}$ is the set of real points in $\Gr(k,n)$ with nonnegative Plucker coordinates \cite{Postnikov}.  The matroid $\mcm$ of a totally nonnegative point $x \in \Gr(k,n)_{\geq 0}$ is known as a \emph{positroid}. The Zariski-closures of the positroid strata are known as \emph{positroid varieties}~\cite{KLS}. The homogeneous coordinate ring $\bbc[\Pi]$  of a positroid variety~$\Pi = \Pi_\mcm$ is a quotient of $\bbc[\Gr(k,n)]$. We show that the homogeneous pieces % $\bbc[\Pi]_\lam$ 
of $\bbc[\Pi]$ are dual to certain naturally defined subspaces of tensor invariants (cf.~\eqref{eq:dualspace}), and suggest a way to compute the dimensions of these pieces using webs. We believe that our duality reflects deep relations between positroids and webs.

We denote by $[n]$ the set of positive integers $\{1,\dots,n\}$, and by $\binom {S} k$ the collection of $k$-element subsets of a set $S$. 

\section*{Acknowledgements}
This project, especially Section~\ref{secn:duality}, grew from a project suggestion made by Thomas Lam and David Speyer for the AMS Snowbird '14 MRC on cluster algebras. Our results from that workshop are in an Appendix, which is jointly written with Darlayne Addabbo, Eric Bucher, Sam Clearman, Laura Escobar, Ningning Ma, Suho Oh, and Hannah Vogel, who were our group members during this workshop. Some of the work took place at the Perimeter Institute. Research at Perimeter Institute is supported by the Government of Canada through Industry Canada and by the Province of Ontario through the Ministry of Research \& Innovation.  T.L. acknowledges support from the Simons Foundation under award number 341949 and from the NSF under agreement No. DMS-1464693.

\section{Networks, boundary measurements, and positroid strata}\label{secn:networks}
\subsection{The dimer model in the disk}
We denote by $\Gr(k,n)$ the Grassmannian of $k$-dimensional subspaces in a fixed $n$-dimensional complex vector space. It has a projective embedding $\Gr(k,n) \injects \bbp^{\binom n k -1}$, and we denote by $\tGr(k,n) \subset \bbc^{\binom n k}$ the affine cone over $\Gr(k,n)$ with respect to this embedding. Points in the affine cone are collections of $\binom n k$ coordinates satisfying the well-known \emph{Pl\"ucker relations}. The homogeneous coordinate ring  $\bbc[\Gr(k,n)]$ (or equivalently, the coordinate ring of the affine cone $\tGr(k,n)$) is generated by \emph{Pl\"ucker coordinates} $(\Delta_I)_{I \in \binom {[n]} k}$. If we represent a point in $x \in \tGr(k,n)$ by a $k \times n$ matrix of maximal rank, then $\Delta_I(x)$ is the $k \times k$ minor of this matrix with columns~$I$.

By a \emph{planar bipartite graph in the disk} we mean a graph $G$ embedded in a closed disk, with its vertices colored in two colors (black and white) such that edges join vertices of opposite color. Furthermore, label the vertices on the boundary of the disk $1,\dots,n$ in counterclockwise order, and require that that $i$th \emph{boundary vertex} is incident to at most one \emph{boundary edge}~$b_i$. 

Furthermore, for the sake of simplicity, throughout this paper we will require that each of the boundary vertices of $G$ is black.

%(if boundary vertex $i$ is adjacent to no edges, treat $e_i$ as an edge of of weight zero, or multiplicity zero, in the recipes that follow).  

By a \emph{network}~$N$ we will mean a planar bipartite graph in the disk whose edges have been weighted by nonzero complex numbers. In examples, an edge drawn without an edge weight is implicitly assumed to have edge weight  equal to~1. 
%We always use the symbol $G$ to represent the underlying bipartite graph of a network $N$. 

A \emph{dimer configuration} on~$N$ (or \emph{almost perfect matching} of $N$), is a collection~$\pi $ of edges of $N$ such that each interior vertex
is used exactly once in $\pi$ (and each boundary vertex is used one or zero times). The \emph{boundary subset} $\partial(\pi) \subset [n]$ is the set of boundary vertices that are used in $\pi$. The cardinality $k = |\partial(\pi)|$ depends only on the underlying bipartite graph $G$, not on the choice of~$\pi$ or the edge weights on $N$. Explicitly: 
\begin{equation}\label{eq:excedance}
k = |\partial(\pi)| = \text{no. of interior white vertices in $G$ minus no. of interior black vertices}.
\end{equation} 
We call the number $k$ in \eqref{eq:excedance} the \emph{excedance} of~$N$.

\begin{rmk}
The requirement that boundary vertices of a planar bipartite graph $G$ are always black simplifies our statements, but can be removed. Any graph $G'$, possibly with white vertices on the boundary, can be turned into a graph $G$ whose boundary vertices are all black, by adding one edge at each white boundary vertex (if $i$ is a white boundary vertex, we drag it into the interior of the disk, and connect it to a newly created black boundary vertex by a newly created edge). 
Dimer configurations $\pi$ on $G$ with boundary $\partial(\pi) = I$ are in bijection with dimer configurations $\pi'$ of $G'$ such that $I$ is the union of the black vertices used in $\pi'$ with the white vertices not used in $\pi'$, cf.~\cite{LamDimers} for further details.  
%More generally, our results have versions in the setting of planar bicolored graphs (i.e., the setting in which graphs are no longer required to be bipartite), using perfect orientations and flows (cf.~\cite[Section 5]{LamNotesII}).  
\end{rmk}

For a network $N$ and a dimer cover $\pi$ on $N$, the \emph{weight} $\wt(\pi)$ is the product of the weights of the edges used in $\pi$. 

The \emph{boundary measurement} $\Delta_I(N)$ is a weight generating
function for dimer configurations with boundary~$I$:
\begin{equation}\label{eq:boundarymeasurmentdefn}
\Delta_I(N) = \sum_{\pi \colon \, \partial(\pi) = I}\text{wt}(\pi).
\end{equation}

\begin{prop}[Kuo \cite{Kuo}, Postnikov-Speyer-Williams \cite{PSW}, Lam \cite{LamNotesII}]\label{prop:localmoves} Let $N$ be a network of excedance $k$, with $n$ boundary vertices, and with at least one almost perfect matching. Then the boundary measurements $(\Delta_I(N))_{I \in \binom {[n]} k} \in \bbc^{\binom n k}$ determine a point $\tX(N)$ in the affine cone $\tGr(k,n)$, and thus a point $X(N) \in \Gr(k,n)$. 
\end{prop}

That is, the boundary measurements satisfy the Pl\"ucker relations. If $N$ has no almost perfect matchings, then all of its Pl\"ucker coordinates are zero, so $X(N)$ is not a well-defined point in $\Gr(k,n)$. We make the standing assumption that all networks $N$ considered in this paper have an almost perfect matching.

\subsection{Local moves}

It is easy to verify that the following local moves can be applied to a network $N$ to yield a new network $N'$ satisfying $X(N') = X(N)$. Thus the Pl\"ucker coordinates for $\tX(N)$ and $\tX(N')$ differ by a common scalar $\aa$ and we write $\tX(N) = \aa \tX(N')$.
\begin{enumerate}
\item [(G)] Gauge equivalence: If edges $e_1,e_2,\ldots,e_d$ are the edges incident to an interior vertex $v$, we can rescale all of their edge weights by the same constant $\aa \in \bbc^*$. The resulting network $N'$ satisfies $\tX(N) = \aa \tX(N')$.
\item[(M1)]
Spider move, square move, or urban renewal: assuming the leaf edges of the spider have been gauge fixed to 1, the transformation is
\begin{equation}\label{eq:spiderparameters}
a'=\frac{a}{ac+bd} \qquad b'=\frac{b}{ac+bd} \qquad c'=\frac{c}{ac+bd} \qquad d'=\frac{d}{ac+bd}
\end{equation}
\begin{center}
\begin{tikzpicture}[scale=0.6]
%\node at (0,2.2) {$1$};
%\node at (2.2,0) {$2$};
%\node at (0,-2.2) {$3$};
%\node at (-2.2,0) {$4$};
%\draw (0,0) circle (2cm);
\draw (-2,0) -- (0,1)--(2,0)--(0,-1)--  (-2,0);
\draw (0,1) -- (0,2);
\draw (0,-1) -- (0,-2);
\node at (-1.6,0.9) {$a$};
\node at (-1.6,-0.9) {$d$};
\node at (1.6,0.9) {$b$};
\node at (1.6,-0.9) {$c$};

\filldraw[black] (0,1) circle (0.1cm);
\filldraw[black] (0,-1) circle (0.1cm);
\filldraw[white] (-2,0) circle (0.1cm);
\draw (-2,0) circle (0.1cm);
\filldraw[white] (0,2) circle (0.1cm);
\draw (0,2) circle (0.1cm);
\filldraw[white] (2,0) circle (0.1cm);
\draw (2,0) circle (0.1cm);
\filldraw[white] (0,-2) circle (0.1cm);
\draw (0,-2) circle (0.1cm);
\end{tikzpicture}
%\node at (0,2.2) {$1$};
%\node at (2.2,0) {$2$};
%\node at (0,-2.2) {$3$};
%\node at (-2.2,0) {$4$};
%\draw (0,0) circle (2cm);
\hspace{30pt}
\begin{tikzpicture}[scale=0.7]
\draw (0,-2) -- (1,0)-- (0,2)-- (-1,0)-- (0,-2);
\draw (1,0) -- (2,0);
\draw (-1,0) -- (-2,0);
\node at (0.9,-1.4) {$a'$};
\node at (-0.9,-1.4) {$b'$};
\node at (0.9,1.4) {$d'$};
\node at (-0.9,1.4) {$c'$};

\filldraw[black] (1,0) circle (0.1cm);
\filldraw[black] (-1,0) circle (0.1cm);
\filldraw[white] (-2,0) circle (0.1cm);
\draw (-2,0) circle (0.1cm);
\filldraw[white] (0,2) circle (0.1cm);
\draw (0,2) circle (0.1cm);
\filldraw[white] (2,0) circle (0.1cm);
\draw (2,0) circle (0.1cm);
\filldraw[white] (0,-2) circle (0.1cm);
\draw (0,-2) circle (0.1cm);
\end{tikzpicture}
\end{center}
and satisfies $\tX(N) = (ac+bd) \tX(N')$.
\item[(M2)]
two-valent vertex removal.  If $v$ has degree two, we can gauge fix both incident edges $(v,u)$ and $(v,u')$ to have weight 1, then contract both edges (that is, we remove both edges, and identify $u$ with $u'$).  Note that if $v$ is a two-valent vertex adjacent to boundary vertex $i$, with edges $(v,i)$ and $(v,u)$, 
%then removing $v$ produces an edge $(b,u)$, and the color of $b$ flips.
then this move can only be applied when $u$ has degree at most two. 

\item[(R1)]
Multiple edges with same endpoints is the same as one edge with sum of weights.
\item[(R2)]
Leaf removal:  Suppose $v$ is a leaf, and $(v,u)$ the unique edge incident to it.  Then we can remove both $v$ and $u$, and all edges incident to $u$. If $(u,v)$ is a boundary edge $b_i$, then the leaf cannot be removed. 
\item[(R3)]
Dipoles (two degree one vertices joined by an edge) can be removed.
\end{enumerate}

On the other hand, the following is one of the deepest results on the combinatorics of networks: 
\begin{theorem}\label{thm:connectedness}[Postnikov \cite{Postnikov}]
If $N$ and $N'$ satisfy $X(N) = X(N')$, then they are connected to each other by a finite sequence of these moves. If furthermore both of these networks have the minimal number of faces in their move-equivalence class, then they are connected by a finite sequence using (M1) and (M2) only. 
\end{theorem}

\subsection{Positroids}\label{ssec:positroids}

Considering boundary measurement maps for various bipartite graphs~$G$ leads to a special stratification of the Grassmannian by \emph{positroid varieties}.  The positroid varieties are distinguished in many senses -- they are exactly the varieties that can be obtained by projecting Richardson varieties from the flag variety \cite{KLS}, they are exactly the compatibly split Frobenius subvarieties of $\Gr(k,n)$ with respect to the standard splitting~\cite{KLS}, the (open versions of) positroid varieties are exactly the symplectic leaves with respect to a Poisson structure on $\Gr(k,n)$ \cite{Goya}. We only review what we will need here and refer the reader to \cite{KLS,LamNotesII}.  
 
The \emph{totally nonnegative Grassmannian} $\Gr(k,n)_{\geq 0}$ consists of points in the Grassmannian $\Gr(k,n)$ that can be given by $k \times n$ matrices with real entries, all of whose Pl\"ucker coordinates are nonnegative.  

For any point $x \in \Gr(k,n)$, the \emph{matroid} of $x$ is the realizable matroid $\mcm(x)$ formed by the subsets $I \in \binom {[n]} k$ such that $\Delta_I(x) \neq 0$. The \emph{matroid variety} associated to a matroid $\mcm$ is the closure in $\Gr(k,n)$ of $\{x \in \Gr(k,n) \colon \, \mcm(x) = \mcm\}.$

A matroid is a \emph{positroid} if it is the matroid of a totally nonnegative point $x \in \Gr(k,n)_{\geq 0}$. Its corresponding matroid variety $\Pi = \Pi_\mcm$ is called a \emph{positroid variety}. The \emph{open positroid variety} $\overset{\circ}{\Pi}_\mcm$ is defined to be the subset of $\Pi$ not belonging to a lower-dimensional positroid variety; it is a Zariski open subset of $\Pi$. The intersection $\overset{\circ}{\Pi} \cap \Gr(k,n)_{\geq 0}$ is called a \emph{positroid cell} and is homeomorphic to~$\bbr_{>0}^{\dim \Pi}$.  

Unlike general matroid varieties which can have arbitrarily bad singularities, positroid varieties are normal and Cohen-Macaulay. Likewise, while the problem of indexing matroid varieties (that is, recognizing representable matroids) is essentially hopeless, positroids are indexed by a well-behaved family of combinatorial structures.  Here is one way of indexing positroid strata: 

\begin{theorem}[Postnikov, Knutson-Lam-Speyer \cite{KLS}]\label{thm:postnikov}
Let $G$ be a planar bipartite graph in the disk. %Then the image of the boundary measurement map for $G$
Then as $N$ varies over all possible edge-weightings of $G$ by nonzero complex numbers, the boundary measurements $X(N)$ sweep out a Zariski dense subset of a single positroid variety $\Pi = \Pi(G)$. Furthermore, every positroid variety $\Pi$ arises in this way from some planar bipartite graph $G$. If we restrict attention to networks $N$ with positive real edge weights, then the boundary measurements sweep out the entire positroid cell $\Gr(k,n)_{\geq 0} \cap \overset{\circ}{\Pi}$. In particular, if a given boundary measurement $\Delta_I$ is not identically zero on $G$, then $\Delta_I(N)$ is positive for each choice of network $N$ with $\bbr_{>0}$ edge weights. 
%If $G$ is a graph for $\Pi$ with the minimal number of faces, then the boundary measurement map is an open embedding $(\bbc^*)^{\text{no. of edges}} \to \overset{\circ}{\Pi}$. 
\end{theorem}

We denote by $\mcm(G)$ the positroid such that $\Pi_\mcm = \Pi(G)$.  We say that $G$ (or $N$) \emph{represents the top cell} if its positroid $\mcm(G)$ is the uniform matroid, or equivalently, if its positroid variety $\Pi(G)$ is equal to the Grassmannian $\Gr(k,n)$.

\section{Tensor invariants and webs}\label{secn:tensors}
Let $U$ be an $r$-dimensional vector space. We denote by $\bigwedge^a(U)$ the $a$th exterior power of $U$. Throughout this paper we assume we have chosen a basis $E_1\dots,E_r$ for $U$, thus giving an isomorphism $\bigwedge^r(U) \cong \bbc$ under which the volume form $E_1 \wedge \cdots \wedge E_r \mapsto 1$. 
%(Our results do not rely on this choice of basis, but our remarks and examples do).  

Let $V_1,\dots,V_n$ be a sequence of irreducible representations of $\SL(U)$. A \emph{tensor invariant} is an element of the space
$$\Hom_{\SL(U)}\left(\bigotimes_{i=1}^n V_i,\bbc\right).$$ 
In this paper, we will be interested in tensor invariants of fundamental representations. Let $\lam = (\lam_1,\dots,\lam_n)$ be a sequence of integers satisfying
\begin{equation}\label{eq:rcomposition}
0 \leq \lam_i \leq r \text{ for } i=1, \dots,n, \text{ and } \lam_1 + \lam_2 + \cdots + \lam_n = kr \text{ for some $k$}.  
\end{equation}
We are interested in the space 
\begin{equation}\label{eq:tensorinvariant}
\mcw_\lam(U) = \Hom_{\SL(U)}\left(\bigotimes_{i=1}^n \bigwedge^{\lam_i}(U),\bbc\right).
\end{equation}

The dimension of the space \eqref{eq:tensorinvariant} is the coefficient of the Schur polynomial $s_{((k^r))}(x_1,\dots,x_r)$ in the product of $r$ elementary symmetric polynomials $e_{a_1}(x_1,\dots,x_r) \cdots e_{a_r}(x_1,\dots,x_r)$. By iterating the dual Pieri Rule, this is the number of $r \times k$ tableaux, of content $\lam$, whose entries are strictly increasing along rows and weakly increasing down columns. 

The simplest example, which the reader is encouraged to have in mind throughout, is when $n= kr$ and $\lam = (1,\dots,1)$. In this case we prefer to denote $\mcw_{(1,\dots,1)}(U)$ by $\mcw(r,n)$. The space $\mcw(r,n)$ is the vector space of $\SL_r$-invariant multilinear functions on $n$ vectors.  Its dimension is the number of standard Young tableaux with $r$ rows and $k$ columns. %Most of our results about $\mcw_\lam(U)$ could be deduced from their versions for $\mcw(r,n)$. 

\medskip

Webs are a particular way of encoding tensor invariants by using planar diagrams. We will define these tensor invariants in two steps, by introducing \emph{untagged webs} followed by \emph{tagged webs}. 

\begin{defn} An \emph{untagged web} $W$ is a planar graph in the disk, with each directed edge $e = (u,v)$ labeled by a multiplicity $m(e) = m(u,v) \in [r-1]$, so that reversing directions replaces $a$ by $r-a$, that is $m(u,v)+m(v,u) = r$ for all directed edges $(u,v)$. Furthermore, at each interior vertex, there exists a way to  direct the edges so that the sum of the incoming multiplicities equals the sum of the outgoing multiplicities. The degree $\lam(W)$ is the sequence of multiplicities $(m(b_1),\dots,m(b_n))$ obtained by directing all boundary edges away from the boundary. 
\end{defn}

We will see that an untagged web $W$ determines a tensor invariant in $\mcw_\lam(U)$, up to a sign. The sign is determined by choosing a \emph{tagging} of the web, which we now define.

\begin{defn}[Tagged web]\label{df:taggedweb}
A \emph{tagged $\SL_r$-web} $\hat{W}$ %(sometimes referred to as a\emph{tagged $\SL_r$-web} below) 
is a planar graph in the disk subject to the following. First, each edge or half-edge of $\hat{W}$ is directed and labeled by a multiplicity in one of the three following ways
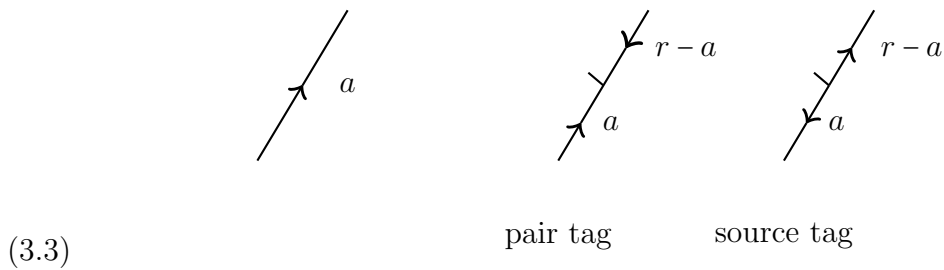
\begin{equation}\label{eq:edgetypes}
\begin{tikzpicture}[scale = 1]
\draw [thick,    decoration={markings,mark=at position 1 with {\arrow[scale=1.7]{>}}},
    postaction={decorate},
    shorten >=0.4pt] (-1,-1)--(-.4,0);
\draw [thick] (-.4,0)--(.2,1);
\node at (.2,0) {$a$};

\begin{scope}[xshift = -5cm]
\node at (8.7,-.5) {$a$};
\node at (9.7,.5) {$r-a$};
\draw [thick, decoration={markings,mark=at position 1 with {\arrow[scale=1.7]{>}}},
    postaction={decorate},
    shorten >=0.4pt] (8,-1)--(8.3,-.5); 
\draw [thick] (8.3,-.5)--(8.9,.5); 
\draw [thick,    decoration={markings,mark=at position 1 with {\arrow[scale=2]{>}}},
    postaction={decorate},
    shorten >=0.4pt] (9.2,1)--(8.9,.5);
\draw [thick] (8.6,0)--(8.4,.17);
\node at (8,-2) {pair tag};
\end{scope}

\begin{scope}[xshift = -2cm]
\node at (8.7,-.5) {$a$};
\node at (9.7,.5) {$r-a$};
\draw [thick]  (8,-1)--(8.3,-.5); 
\draw [thick, decoration={markings,mark=at position 1 with {\arrow[scale=1.7]{>}}},
    postaction={decorate},
    shorten >=0.4pt] (8.6,0)--(8.3,-.5); 
\draw [thick, decoration={markings,mark=at position 1 with {\arrow[scale=1.7]{>}}},
    postaction={decorate},
    shorten >=0.4pt] (8.6,0)--(8.9,.5); 
\draw [thick] (9.2,1)--(8.9,.5);
\draw [thick] (8.6,0)--(8.4,.17);
\node at (8,-2) {source tag};
\end{scope}
\end{tikzpicture}.
\end{equation}

The ``tiny edges'' decorating the second and third edge types are called \emph{tags} and they provide us with a preferred choice of side for the given edge. Second, we require that 
each interior vertex of $\hat{W}$ is modeled on one of the following two pictures:

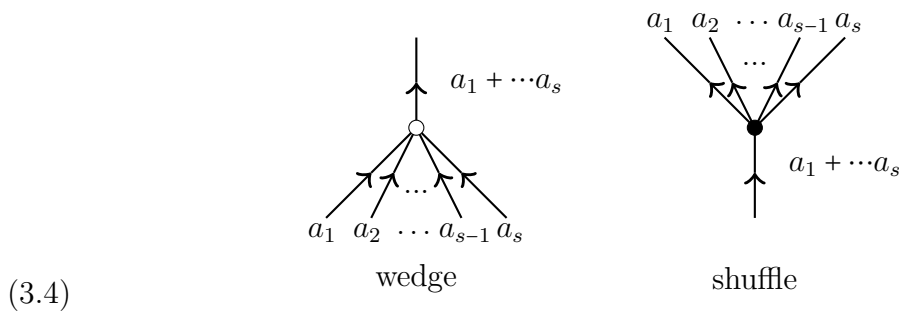
\begin{equation}
\label{eq:fundamentalmorphisms}
\begin{tikzpicture}[scale = 1]
\draw [thick,decoration={markings,mark=at position 1 with {\arrow[scale=1.7]{>}}},
    postaction={decorate},
    shorten >=0.4pt] (0,0)--(0,.6);
\draw [thick] (0,.6)--(0,1.2);
\draw [thick,decoration={markings,mark=at position 1 with {\arrow[scale=1.7]{>}}},
    postaction={decorate},
    shorten >=0.4pt] (-1.2,-1.2)--(-.6,-.6);
\draw [thick] (-.6,-.6)--(0,0);
\draw [thick,decoration={markings,mark=at position 1 with {\arrow[scale=1.7]{>}}},
    postaction={decorate},
    shorten >=0.4pt] (-.6,-1.2)--(-.3,-.6);
\draw [thick] (-.3,-.6)--(0,0);
\draw [thick,decoration={markings,mark=at position 1 with {\arrow[scale=1.7]{>}}},
    postaction={decorate},
    shorten >=0.4pt] (.6,-1.2)--(.3,-.6);
\draw [thick] (.3,-.6)--(0,0);
\draw [thick,decoration={markings,mark=at position 1 with {\arrow[scale=1.7]{>}}},
    postaction={decorate},
    shorten >=0.4pt] (1.2,-1.2)--(.6,-.6);
\draw [thick] (.6,-.6)--(0,0);

\node at (-1.25,-1.4) {$a_1$};
\node at (-.65,-1.4) {$a_2$};
\node at (0,-1.4) {$\dots$};
\node at (0,-.85) {$...$};
\node at (.65,-1.4) {$a_{s-1}$};
\node at (1.25,-1.4) {$a_s$};
\node at (1.2,.6) {$a_1+\cdots a_s$};
\draw [fill= white] (0,0) circle [radius = .1];
\node at (0,-2) {wedge};

\begin{scope}[xshift = 4.5cm]
\draw [thick,decoration={markings,mark=at position 1 with {\arrow[scale=1.7]{>}}},
    postaction={decorate},
    shorten >=0.4pt] (0,-1.2)--(0,-.6);
\draw [thick] (0,-.6)--(0,0);
\draw [thick] (-1.2,1.2)--(-.6,.6);
\draw [thick,decoration={markings,mark=at position 1 with {\arrow[scale=1.7]{>}}},
    postaction={decorate},
    shorten >=0.4pt] (0,0)--(-.6,.6);
\draw [thick] (-.6,1.2)--(-.3,.6);
\draw [thick,decoration={markings,mark=at position 1 with {\arrow[scale=1.7]{>}}},
    postaction={decorate},
    shorten >=0.4pt] (0,0)--(-.3,.6);
\draw [thick] (.6,1.2)--(.3,.6);
\draw [thick,decoration={markings,mark=at position 1 with {\arrow[scale=1.7]{>}}},
    postaction={decorate},
    shorten >=0.4pt] (0,0)--(.3,.6);
\draw [thick] (1.2,1.2)--(.6,.6);
\draw [thick,decoration={markings,mark=at position 1 with {\arrow[scale=1.7]{>}}},
    postaction={decorate},
    shorten >=0.4pt] (0,0)--(.6,.6);

\node at (-1.25,1.4) {$a_1$};
\node at (-.65,1.4) {$a_2$};
\node at (0,1.4) {$\dots$};
\node at (0,.85) {$...$};
\node at (.65,1.4) {$a_{s-1}$};
\node at (1.25,1.4) {$a_s$};
\node at (1.2,-.5) {$a_1+\cdots a_s$};
\draw [fill= black] (0,0) circle [radius = .1];
\node at (0,-2) {shuffle};
\end{scope}
\end{tikzpicture}.
\end{equation}

%In the webs that we deal with later sections, only the first two pictures will appear around any vertex. This is a special property of webs that come from $r$-weblike subgraphs of bipartite graphs which are dealt with in Section~\ref{secn:main}.

Thirdly and finally, we require that each boundary edge in $\hat{W}$ is a source. The \emph{degree} $\lam(\hat{W}) = (\lam_1,\dots,\lam_n)$ of $\hat{W}$ is the sequence of edge multiplicities of the boundary edges. 
\end{defn}

We adopt the convention throughout that $\hat{W}$ stands for a tagged web as in Definition~\ref{df:taggedweb}, and will reserve the symbol $W$ for untagged webs, or for the weblike subgraphs that we discuss later. We will refer to the data of tagging by the letter $\mco$, writing $\hat{W} = \hat{W}(W,\mco)$.

The definition of a tagged web could in principle allow for ``wedge \emph{and} shuffle'' vertices, in which there are multiple edges flowing in both the inward and outward direction in~\eqref{eq:fundamentalmorphisms}. However, these will not come up in practice in our paper, so we do not work with this definition.

\medskip

It will be important for us to note that by tagging sufficiently many edges in $\hat{W}$, one can always ensure that there are no oriented cycles in $\hat{W}$. We will always assume that $\hat{W}$ has been tagged in this way, because this simplifies the association of a tensor invariant to each web. 

\medskip

Denote by $\mcf \mcs_\lam(r)$ the space of finite formal $\bbc$-linear combinations of tagged $\SL_r$-web diagrams of degree $\lam$.  Now we explain how an $\SL_r$-web $\hat{W}$ of degree $\lam$ determines a tensor invariant $\hat{W} \in \mcw_\lam(U)$.  This induces a linear map $\mcf \mcs_\lam(r) \to \mcw_\lam(U)$ that is known to be surjective. %(see Section \ref{secn:skein}).

An edge with multiplicity $m(e)$ in $\hat{W}$ encodes a copy of $\bigwedge^{m(e)}(U)$. Each of the interior vertices \eqref{eq:fundamentalmorphisms} encodes an $\SL(U)$-invariant map between the indicated tensor powers of fundamental representations. The first vertex encodes the exterior product map 
\begin{equation}
\bigwedge^{a_1}(U) \otimes \cdots \otimes  \bigwedge^{a_s}(U) \to \bigwedge^{a_1+\cdots +a_s}(U) \label{eq:wedgeeq}
\end{equation}
given by $x_1 \otimes x_2 \otimes \cdots \otimes x_s \mapsto x_1 \wedge x_2 \wedge \cdots \wedge x_s$.  The second vertex encodes the map 
\begin{equation}
\bigwedge^{a_1+\cdots +a_s}(U) \to \bigwedge^{a_s}(U) \otimes \cdots \otimes  \bigwedge^{a_1}(U) \label{eq:forkeq}
\end{equation}
given by sending the wedge product $x_1 \wedge  \cdots  \wedge x_b \in \bigwedge^{b}(U)$ to the signed sum of shuffles
\begin{equation}\label{eq:shufflesigns}
\sum \pm (x_{i_1} \wedge x_{i_2} \wedge \cdots \wedge x_{i_{a_s}}) \otimes \cdots \otimes  (x_{i_{b-a_1+1}} \wedge x_{i_{b-a_1+2}} \wedge \cdots \wedge x_{i_{b}}) \in \bigwedge^{a_s}(U) \otimes \cdots \otimes  \bigwedge^{a_1}(U)
\end{equation}
where the summation is over permutations $(i_1,i_2,\ldots,i_b)$ of $(1,2,\ldots,b)$ such that indices are increasing in each block: $i_1 < i_2 < \cdots < i_{a_s}$, $\ldots$,  $i_{b-a_1+1}< \cdots < i_b$. 
The sign $\pm$ is the sign of the permutation $(i_1,i_2,\ldots,i_b)$, multiplied by the ``global sign'' of the permutation $(b-a_s+1,\dots,b,\dots,1,\dots,a_1)$ 
Note that in both cases, the cyclic order of the edges in \eqref{eq:fundamentalmorphisms} is crucial for specifying signs. 

The tagged edges in \eqref{eq:edgetypes} should be thought of as degenerate cases of the maps \eqref{eq:wedgeeq} and \eqref{eq:forkeq}, where the tag stands encodes a copy of $\bigwedge^r(U)$, which we canonically identify with $\bbc$ using the volume form. Thus the pair tag $\bigwedge^a \otimes \bigwedge^{r-a} \to \bigwedge^{r} \cong \bbc$ produces a number obtained by pairing the two incoming tensors, and source tag ``creates'' two tensors using the shuffle $\bigwedge^{r} \to \bigwedge^{a} \otimes \bigwedge^{r-a}$. Note that the side of the edge that the tag occurs on matters, because, for example, the maps $\bigwedge^a \otimes \bigwedge^{r-a} \to \bigwedge^{r} \cong \bbc$ and $\bigwedge^{r-a} \otimes \bigwedge^{a} \to \bigwedge^{r} \cong \bbc$ differ by a sign.

To evaluate $\hat{W}$ on a simple tensor $x_1 \otimes \cdots \otimes x_n \in \bigotimes_{i=1}^n \bigwedge^{\lam_i}(U)$, we imagine placing each tensor $x_i$ at boundary vertex~$i$. We repeatedly compose the four basic morphisms -- wedge, shuffle, pair, and source -- as indicated by the arrows in $\hat{W}$ to obtain the value~$\hat{W}(x_1 \otimes \cdots \otimes x_n)$. 

\begin{rmk}[Taggings and perfect orientations]\label{rmk:perfectorientations}
In practice, the following recipe will work to specify a tagging for the webs we consider in this paper. Suppose that we have an untagged web $W$ whose underlying graph $G$ is a bipartite graph.  Orient all the edges in~$W$ from white to black, and further suppose that the weights of the directed edges around each vertex sum to $r$. These conditions characterize the $r$-weblike subgraphs that appear in this paper. In this situation, we can tag the web by choosing a \emph{perfect orientation} of $G$. 

A \emph{perfect orientation} of a planar bipartite graph $G$ in the disk is a choice of direction on each edge of $G$ such that every white vertex has outdegree $1$ and every interior black vertex has indegree $1$. To get a perfect orientation of $G$, just choose an almost perfect matching of $G$: given an almost perfect matching $\pi$, we obtain a perfect orientation by directing each edge $e$ from white to black if $e \in \pi$ (resp. black to white  if $e \notin \pi$).  The $r$-weblike subgraphs defined below always have almost perfect matchings (in fact, they are a union of $r$ almost perfect matchings), so they can be perfectly oriented. Furthermore, if $G$ has excedance $k$, then any perfect orientation for $G$ will have exactly $k$ boundary sinks (and exactly $n-k$ boundary sources). By adding a pair tag at each of the $k$ boundary sinks, we get a tagged web. The only reason we don't use this construction in general is that there is no guarantee that the resulting web does not contain oriented cycles. To obtain such a web, we may need to add more tags along interior edges, introducing sources and pairings.
\end{rmk}

\begin{example}\label{eg:SL4WebEvaluation} Consider the $\SL_4$-web $\hat{W}$ in \eqref{eq:SL4WebEvaluation}. The edge mutiplicities are red, and multiplicities equal to one  are omitted. The underlying bipartite graph for $\hat{W}$ has excedance~$2$.  %and there are exactly $2$ boundary sinks, at vertices $6$ and $8$. 
In the second figure, we show how to evaluate $\hat{W}$ on the tensor product of basis vectors
%\begin{equation}\label{eq:exampleofevaluation}
$E = E_1 \otimes E_2 \otimes E_3 \otimes E_4 \otimes E_3 \otimes E_2 \otimes E_1 \otimes E_4 \in U^{\otimes 8}$. We abbreviate $E_{123} = E_1 \wedge E_2 \wedge E_3$ etc.
%\end{equation} 
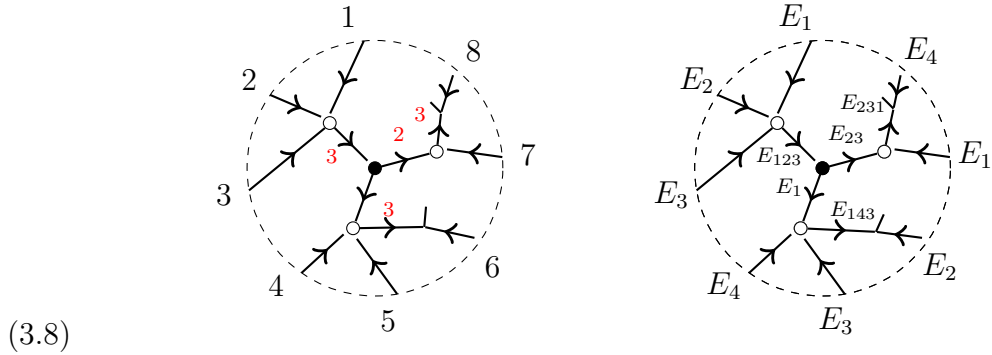
\begin{equation}\label{eq:SL4WebEvaluation}
\begin{tikzpicture}[scale=.85]
\node at (190:2.4cm) {3};
\node at (145:2.4cm) {2};
\node at (100:2.4cm) {1}; 
\node at (50:2.4cm) {8};
\node at (5:2.4cm) {7};
\node at (-40:2.4cm) {6};
\node at (-85:2.4cm) {5};
\node at (-130:2.4cm) {4};

\draw[dashed] (0,0) circle (2cm);
\draw[thick,decoration={markings,mark=at position .6 with {\arrow[scale=1.7]{>}}},
    postaction={decorate},
    shorten >=0.4pt] (0,0)--(15:.9cm);
\node [red] at (55:.65cm) {\tiny $2$};
\draw [thick,decoration={markings,mark=at position .7 with {\arrow[scale=1.7]{>}}},
    postaction={decorate},
    shorten >=0.4pt] (15:1cm)--(40:1.35cm);
\node [red] at (50:1.1cm) {\tiny $3$};
\draw [thick,decoration={markings,mark=at position .6 with {\arrow[scale=1.7]{>}}},
    postaction={decorate},
    shorten >=0.4pt] (50:1.9cm)--(40:1.35cm) ;
\draw [thick] (40:1.35cm)--(50:1.35cm);		
\draw [thick,decoration={markings,mark=at position .6 with {\arrow[scale=1.7]{>}}},
    postaction={decorate},
    shorten >=0.4pt] (5:2cm)--(15:1.15cm);
\draw [thick,decoration={markings,mark=at position .6 with {\arrow[scale=1.7]{>}}},
    postaction={decorate},
    shorten >=0.4pt] (135:1cm)--(135:.15cm);
\node [red] at (165:.7cm) {\tiny $3$};
\draw [thick,decoration={markings,mark=at position .6 with {\arrow[scale=1.7]{>}}},
    postaction={decorate},
    shorten >=0.4pt] (95:2cm)--(130:1.1cm);
\draw [thick,decoration={markings,mark=at position .6 with {\arrow[scale=1.7]{>}}},
    postaction={decorate},
    shorten >=0.4pt] (145:2cm)--(136:1.15cm);
\draw [thick,decoration={markings,mark=at position .6 with {\arrow[scale=1.7]{>}}},
    postaction={decorate},
    shorten >=0.4pt] (190:2cm)--(140:1cm);
\draw [thick,decoration={markings,mark=at position .6 with {\arrow[scale=1.7]{>}}},
    postaction={decorate},
    shorten >=0.4pt](0,0)--(250:.9cm);
\draw [thick,decoration={markings,mark=at position .6 with {\arrow[scale=1.7]{>}}},
    postaction={decorate},
    shorten >=0.4pt] (235:2cm)-- (245:1.13cm);
\draw [thick,decoration={markings,mark=at position .6 with {\arrow[scale=1.7]{>}}},
    postaction={decorate},
    shorten >=0.4pt] (280:2cm)--(255:1.13cm);
\draw [thick,decoration={markings,mark=at position .6 with {\arrow[scale=1.7]{>}}},
    postaction={decorate},
    shorten >=0.4pt] (250:1cm)--(310:1.2cm) node[pos = .5, red, above] {\tiny 3};
\draw [thick,decoration={markings,mark=at position .6 with {\arrow[scale=1.7]{>}}},
    postaction={decorate},
    shorten >=0.4pt] (325:1.9cm)--(310:1.2cm);
\draw [thick] (310:1.2cm)--(323:1cm);

\filldraw[black] (0,0) circle (0.1cm);
\filldraw[white] (15:1cm) circle (.1cm);
\draw (15:1cm) circle (.1cm);
\filldraw[white] (135:1cm) circle (.1cm);
\draw (135:1cm) circle (.1cm);
\filldraw[white] (250:1cm) circle (.1cm);
\draw (250:1cm) circle (.1cm);

\begin{scope}[xshift = 7cm]
\node at (190:2.4cm) {$E_3$};
\node at (145:2.4cm) {$E_2$};
\node at (100:2.4cm) {$E_1$};
\node at (50:2.4cm) {$E_4$};
\node at (5:2.4cm) {$E_1$};
\node at (-40:2.4cm) {$E_2$};
\node at (-85:2.4cm) {$E_3$};
\node at (-130:2.4cm) {$E_4$};
\draw[dashed] (0,0) circle (2cm);

\draw[dashed] (0,0) circle (2cm);
\draw[thick,decoration={markings,mark=at position .6 with {\arrow[scale=1.7]{>}}},
    postaction={decorate},
    shorten >=0.4pt] (0,0)--(15:.9cm);
\node at (55:.65cm) {\tiny $E_{23}$};
\draw [thick,decoration={markings,mark=at position .7 with {\arrow[scale=1.7]{>}}},
    postaction={decorate},
    shorten >=0.4pt] (15:1cm)--(40:1.45cm);
\node at (58:1.2cm) {\tiny $E_{231}$};
\node at (210:.6cm) {\tiny $E_{1}$};
\draw [thick,decoration={markings,mark=at position .6 with {\arrow[scale=1.7]{>}}},
    postaction={decorate},
    shorten >=0.4pt] (50:1.9cm)--(40:1.45cm) ;
\draw [thick] (40:1.45cm)--(50:1.45cm);		
\draw [thick,decoration={markings,mark=at position .6 with {\arrow[scale=1.7]{>}}},
    postaction={decorate},
    shorten >=0.4pt] (5:2cm)--(15:1.15cm);
\draw [thick,decoration={markings,mark=at position .6 with {\arrow[scale=1.7]{>}}},
    postaction={decorate},
    shorten >=0.4pt] (135:1cm)--(135:.15cm);
\node at (165:.7cm) {\tiny $E_{123}$};
\draw [thick,decoration={markings,mark=at position .6 with {\arrow[scale=1.7]{>}}},
    postaction={decorate},
    shorten >=0.4pt] (95:2cm)--(130:1.1cm);
\draw [thick,decoration={markings,mark=at position .6 with {\arrow[scale=1.7]{>}}},
    postaction={decorate},
    shorten >=0.4pt] (145:2cm)--(136:1.15cm);
\draw [thick,decoration={markings,mark=at position .6 with {\arrow[scale=1.7]{>}}},
    postaction={decorate},
    shorten >=0.4pt] (190:2cm)--(140:1cm);
\draw [thick,decoration={markings,mark=at position .6 with {\arrow[scale=1.7]{>}}},
    postaction={decorate},
    shorten >=0.4pt](0,0)--(250:.9cm);
\draw [thick,decoration={markings,mark=at position .6 with {\arrow[scale=1.7]{>}}},
    postaction={decorate},
    shorten >=0.4pt] (235:2cm)-- (245:1.13cm);
\draw [thick,decoration={markings,mark=at position .6 with {\arrow[scale=1.7]{>}}},
    postaction={decorate},
    shorten >=0.4pt] (280:2cm)--(255:1.13cm);
\draw [thick,decoration={markings,mark=at position .6 with {\arrow[scale=1.7]{>}}},
    postaction={decorate},
    shorten >=0.4pt] (250:1cm)--(310:1.3cm) node[pos = .7, above] {\tiny $E_{143}$};
\draw [thick,decoration={markings,mark=at position .6 with {\arrow[scale=1.7]{>}}},
    postaction={decorate},
    shorten >=0.4pt] (325:1.9cm)--(310:1.3cm);
\draw [thick] (310:1.3cm)--(323:1.2cm);

\filldraw[black] (0,0) circle (0.1cm);
\filldraw[white] (15:1cm) circle (.1cm);
\draw (15:1cm) circle (.1cm);
\filldraw[white] (135:1cm) circle (.1cm);
\draw (135:1cm) circle (.1cm);
\filldraw[white] (250:1cm) circle (.1cm);
\draw (250:1cm) circle (.1cm);
\filldraw[black] (0,0) circle (0.1cm);
\filldraw[white] (15:1cm) circle (.1cm);
\draw (15:1cm) circle (.1cm);
\filldraw[white] (135:1cm) circle (.1cm);
\draw (135:1cm) circle (.1cm);
\filldraw[white] (250:1cm) circle (.1cm);
\draw (250:1cm) circle (.1cm);
\end{scope}
\end{tikzpicture}.
\end{equation}
To begin, place the $i$th vector in $E$ at boundary vertex~$i$. Vectors flow along directed arrows until they reach a black vertex, which in this case happens when $E_1 \wedge E_2 \wedge E_3$ arrives at the interior black vertex. The map \eqref{eq:forkeq} splits this tensor up as a signed sum 
\begin{equation}\label{eq:exampleofevaluation}
E_1 \otimes (E_2 \wedge E_3)  - E_2 \otimes (E_1 \wedge E_3)  + E_3 \otimes (E_1 \wedge E_2).
\end{equation} 
The signs come from \eqref{eq:shufflesigns}.
%The signs alternate in \eqref{eq:exampleofevaluation}, with the convention that the leading ``$+$'' term is the one in which the incoming flows cross maximally (i.e., the flow coming from vertex $1$ crosses the incoming flow from both vertices $2$ and $3$). 
The evaluation is a sum of contributions from three terms, but only the first term pairs nontrivially with the vectors at $6$ and $8$. The pairings for this term are $(E_1 \wedge E_4 \wedge E_3) \wedge E_2 = -1$ and $(E_2 \wedge E_3 \wedge E_1) \wedge E_4 = 1$. The evaluation is $\hat{W} \big|_E = 1 \cdot -1 \cdot 1 = -1$, obtained as the product of the sign at the shuffle, at vertex $6$, and at vertex $8$. 
\end{example}

\begin{rmk}
Definition~\ref{df:taggedweb} is \emph{not} the usual definition of a web \cite{SkewHowe,Kuperberg}. Typically, the definition comes with an added requirement that every interior vertex be \emph{trivalent} (as a degenerate case of this, bivalent interior vertices are also allowed). Requiring trivalence is reasonable, because it can be shown that all possible $\SL(U)$-equivariant maps amongst tensor products of various $\bigwedge^a(V)$ come from compositions of maps involving three tensor factors \cite{MorrisonThesis}. Likewise, using (M2), any network $N$ can be transformed to a network $N'$ with interior vertices of valence at most three, without changing the boundary measurements. However, the trivalent restriction is not necessary for the current work.
\end{rmk}

\begin{example}[Webs in small rank] Let us describe $\SL_r$-webs $\hat{W}$ in the multilinear case $\lam = (1,\dots,1)$, for $r = 1, 2,3$. In each case, we have vectors $v_1,\dots,v_n$ sitting at the boundary vertices $1,\dots,n$. When $r=1$, then $U \cong \bbc$, and $\hat{W}$ is a union of isolated interior edges (these do not affect the tensor invariant and can be removed) and edges based at boundary vertices. Thus $\hat{W}$ encodes a monomial in the vectors $v_1,\dots,v_n$.  When $r =2$, an $\SL_2$-web $\hat{W}$ consists of a disjoint union of a) tagged cycles, and b) directed paths $v_i\to v_j$ between boundary vertices. An oriented cycle contributes a multiplicative factor of $\pm 2$ (depending on the tagging) when $\hat{W}$ is evaluated on $v_1,\dots,v_n$. Changing  the orientation on a path changes the web by a minus sign. Thus, ignoring these signs, and removing all cycles, $\SL_2$-webs are spanned by crossingless matchings on the boundary vertices. In fact, these crossingless matchings are a basis for $\mcw(2,2r)$. 

A result in similar spirit is true when $r=3$. In this case, the sign of a web does not depend on tagging. $\SL_3$-webs are typically drawn as bipartite graphs without directed edges, with the convention that the edge multiplicities are given by $m(b,w)=1$, where $b$ is a black vertex and $w$ is a white vertex. The $\SL_3$ \emph{skein relations} provide diagrammatic rules for expressing any $\SL_3$-web in terms of a basis of \emph{non-elliptic webs} \cite{Kuperberg}, i.e. webs that are without $2$-valent vertices or interior faces bounded by four or fewer sides (cf.~\cite{Tensors} for a summary of these results). 

For $r \geq 4$, the set of $\SL_r$-webs is a distinguished spanning set, but the existence of a \emph{web basis} satisfying enough ``desirable'' properties is unknown.   
\end{example}

Finally, let us recall the relationsip between webs and the Grassmannian. Each Pl\"ucker coordinate $\Delta_I \in \bbc[\Gr(k,n)]$ is an $\SL_k$-invariant function of the column vectors $v_1,\dots,v_n \in \bbc^k$ representing a point in $\Gr(k,n)$. 

In general, there is a $\bbz^n$-grading of $\bbc[\Gr(k,n)]$ given by the degree in each column. For example, $\Delta_{123}\Delta_{256} \in \bbc[\Gr(3,6)]$ has degree $(1,2,1,0,1,1)$. The coordinate ring decomposes as 
\begin{equation}\label{eq:coordringgrading}
\bbc[\Gr(k,n)] = \bigoplus_{r=1}^{\infty} \bigoplus_{\lam_1+\cdots \lam_n = rk} \bbc[\Gr(k,n)]_\lam,
\end{equation}
where the inner direct sum is over $\lam$ as in \eqref{eq:rcomposition}. Note that $\bbc[\Gr(k,n)]_\lam$ is given by the invariant space
$$\Hom_{\SL_k}(\bigotimes_{j=1}^n \Sym^{\lam_j}(\bbc^k),\bbc).$$

Our main construction will give a duality between the graded piece $\bbc[\Gr(k,n)]_{\lam}$ of the homogeneous coordinate ring of the Grassmannian and the $\SL_r$-invariant space $\mcw_\lam(U)$, where $\lam$ satisfies \eqref{eq:rcomposition}. We will discuss in Section~\ref{secn:duality} that these spaces are naturally dual using purely representation-theoretic considerations. However, our work manifests this duality explicitly in terms of webs and dimers.

If $n = rk$, the invariant space $\mcw(k,n)$ sits inside $\bbc[\Gr(k,n)]$ as the multilinear piece $\lam = (1,\dots,1)$. It is spanned by the $r$-fold products of Pl\"ucker coordinates $\Delta_{I_1} \cdots \Delta_{I_r}$ satisfying $I_1 \cup \cdots \cup I_r = [n]$. Thus our results will provide a duality between  $\mcw(k,n)$ and $\mcw(r,n)$. We examine this duality pictorially in the Appendix.

\section{The main construction}\label{secn:main}
In this section we make the connection between dimer configurations and webs. Suppose we are given $r$ dimer configurations $\pi_1,\dots,\pi_r$ on $N$. By superimposing them we naturally obtain what we call an $r$-weblike subgraph:
\begin{defn}
An \emph{$r$-weblike subgraph} $W \subset G$ (alternatively,  $W \subset N$) is a subgraph of $G$, using all the vertices of $G$, with each edge $e$ of $W$ labeled by a multiplicity $m(e) \in [r]$, in such a way that the sum of the multiplicities around each interior vertex is $r$. The \emph{weight} of $W$ is the product $\wt(W) = \prod_e \wt(e)^{m(e)}$ of the edge weights raised to the indicated multiplicity. The \emph{degree} of $W$ is its sequence $\lam(W) = (m(b_1),\dots,m(b_n))$ of boundary multiplicities,  where $b_1,\dots,b_n$ are the boundary edges. 
\end{defn}

Each $r$-weblike subgraph $W$ determines an untagged $\SL_r$-web by giving each black-white edge $(b,w)$ the multiplicity of this edge in $W$. Thus, one can get a tensor invariant $\hat{W}(W,\mco)$ from $W$ by choosing a tagging $\mco$.  For any two such choices of tagging $\mco$ and $\mco'$, the resulting tensor invariants are equal up to a sign: $\hat{W}(W,\mco) = \pm \hat{W}(W,\mco')$.  In Section~\ref{secn:proofofmain}, we define a sign $\sgn(W,\mco) \in \{\pm 1\}$ so that the tensor invariant
\begin{equation}\label{eq:eoraII}
\textbf{W} := \sgn(W,\mco)\hat{W}(W,\mco).
\end{equation}
does not depend on $\mco$.  In other words, there is a ``correct'' choice of sign for the tensor invariant represented by $W$.
%We will say more about these signs in Section~\ref{secn:proofofmain}. Let us postpone the details and state the upshot: for a given $r$-weblike subgraph $W$ there is a ``correct'' choice of sign for the tensor invariant represented by $W$, chosen so that the resulting tensor invariant satisfies a certain equation \eqref{eq:eora}.  By agreeing to always make this correct choice of sign, we can unambiguously think of any $r$-weblike subgraph $W$ of degree $\lam$ as a tensor invariant which we denote by boldface $\textbf{W} \in \mcw_\lam(U)$. 
With this in mind, we can formulate the main definition of this paper. 
\begin{defn}Let $N$ be a network of excedance $k$ and $\lam$ satisfy \eqref{eq:rcomposition}. We denote by
\begin{equation}\label{eq:Webrdefn}
\Web_r(N;\lam) := \sum_{W \subset G, \;\; \lam(W) = \lam} \wt(W) \, \textbf{W} \in \mcw_\lam(U),  
\end{equation}
the weighted sum over the $r$-weblike subgraphs of $G$ with degree $\lam$. It is a $\bbc$-linear combination of the various boldface $\textbf{W} \in \mcw_\lam(U)$, with the $\wt(W)$'s serving as coefficients. 
\end{defn}

We think of $\Web_r(N;\lam)$ as an \emph{$r$-fold boundary measurement}.  When $r =1$, the choice of $\lam$ is equivalent to the choice of $I \in \binom{[n]}{k}$, and $\Web_1(N;I)$ is a variant of the boundary measurement $\Delta_I(N)$ \eqref{eq:boundarymeasurmentdefn}.  Whereas $\Delta_I(N)$ is a number, our $\Web_1(N;I)$ lies in $\mcw_\lam(U)$, which in this case is isomorphic to $\bbc$.

We denote by $\Web_r(N) := \sum_\lam \Web_r(N;\lam) \in \bigoplus_\lam \mcw_\lam(U)$.

\begin{example}\label{eg:computeweights} Consider the pair of networks $N$ and $N'$ given by
\begin{equation}\label{eq:exampleofnetworks}
\begin{tikzpicture}[scale=.65]
\draw[dashed] (1,1) circle (3cm);
\draw [thick] (0,0) -- (0,2)-- (2,2)-- (2,0)-- (0,0);
\draw [thick] (2.5,-.5) -- (2,0);
\draw [thick] (-.5,2.5) -- (0,2);
\draw [thick] (-.5,2.5)--(-1.6,2.5);
\draw [thick] (-.5,2.5)--(-.8,3.4);
\draw [thick] (2.5,-.5)--(3.6,-.5);
\draw [thick] (2.5,-.5)--(2.8,-1.4);
\draw [thick] (2,2)--(3.1,3.1);
\draw [thick] (-1.1,-1.1)--(0,0);

\node at (-.4,1) {$a$};
\node at (2.4,1) {$c$};
\node at (1,2.4) {$b$};
\node at (1,-0.4) {$d$};
\node at (-1.2,2.2) {$e$};
\node at (3.1,-.1) {$f$};

\filldraw[black] (2,0) circle (0.1cm);
\filldraw[black] (0,2) circle (0.1cm);
\filldraw[white] (0,0) circle (0.1cm);
\draw (0,0) circle (0.1cm);
\filldraw[white] (2,2) circle (0.1cm);
\draw (2,2) circle (0.1cm);
\filldraw[white] (-.5,2.5) circle (0.1cm);
\draw (-.5,2.5) circle (0.1cm);
\filldraw[white] (2.5,-.5) circle (0.1cm);
\draw (2.5,-.5) circle (0.1cm);
\end{tikzpicture}
\hspace{20pt}
\begin{tikzpicture}[scale=0.7]
%\node at (0,2.2) {$1$};
%\node at (2.2,0) {$2$};
%\node at (0,-2.2) {$3$};
%\node at (-2.2,0) {$4$};
%\draw (0,0) circle (2cm);
\draw [thick] (0,0) -- (2,0)--(2,2)--(0,2)--  (0,0);
\draw [thick] (2.5,2.5) -- (2,2);
\draw [thick] (2.5,2.5)--(3.1,3.1);
\draw [thick] (-.5,-.5)--(-1.1,-1.1);
\draw [thick] (0,2)--(-1.6,2.5);
\draw [thick] (0,2)--(-.8,3.4);
\draw [thick] (-.5,-.5)--(-1.1,-1.1);
\draw [thick] (2,0)--(3.6,-.5);
\draw [thick] (2,0)--(2.8,-1.4);
\draw [thick] (-.5,-.5) -- (0,0);
\node at (-.4,1) {$c'$};
\node at (2.4,1) {$a'$};
\node at (1,2.4) {$b'$};
\node at (1,-0.4) {$d'$};
\node at (-1.2,2.1) {$e$};
\node at (3.2,0) {$f$};

\filldraw[black] (0,0) circle (0.1cm);
\filldraw[black] (2,2) circle (0.1cm);
\filldraw[white] (2,0) circle (0.1cm);
\draw (2,0) circle (0.1cm);
\filldraw[white] (0,2) circle (0.1cm);
\draw (0,2) circle (0.1cm);
\filldraw[white] (2.6,2.6) circle (0.1cm);
\draw (2.6,2.6) circle (0.1cm);
\filldraw[white] (-.6,-.6) circle (0.1cm);
\draw (-.6,-.6) circle (0.1cm);
%\filldraw[white] (-.5,-.5) circle (0.1cm);
%\draw (-.5,-.5) circle (0.1cm);
\draw[dashed] (1,1) circle (3cm);
\end{tikzpicture},
%\node at (0,2.2) {$1$};
%\node at (2.2,0) {$2$};
%\node at (0,-2.2) {$3$};
%\node at (-2.2,0) {$4$};
%\draw (0,0) circle (2cm);
\end{equation}
where $a',b',c',d'$ are related to $a,b,c,d$ according to the spider move \eqref{eq:spiderparameters}
\begin{equation}
a'=\frac{a}{ac+bd} \qquad b'=\frac{b}{ac+bd} \qquad c'=\frac{c}{ac+bd} \qquad d'=\frac{d}{ac+bd}.
\end{equation}

Their parameters are $k=2$ and $n=6$, and the letters $a,a',b,b',\dots,e,f \in \bbc^{\times}$ denote edge weights. Let us consider $r =3$ and $\lam = (1,1,1,1,1,1)$. The network $N$ has three $r$-weblike subgraphs with degree $\lam$. The $3$-fold boundary measurement $\Web_3(N;\lam)$ is a linear combination 
\begin{equation}\label{eg:Web3N}
\begin{tikzpicture}[scale=.45]
\draw[dashed] (1,1) circle (3cm);
\draw [thick] (0,0) -- (0,2)-- (2,2)-- (2,0)-- (0,0);
\draw [thick] (2.5,-.5) -- (2,0);
\draw [thick] (-.5,2.5) -- (0,2);
\draw [thick] (-.5,2.5)--(-1.6,2.5);
\draw [thick] (-.5,2.5)--(-.8,3.4);
\draw [thick] (2.5,-.5)--(3.6,-.5);
\draw [thick] (2.5,-.5)--(2.8,-1.4);
\draw [thick] (2,2)--(3.1,3.1);
\draw [thick] (-1.1,-1.1)--(0,0);
\node at (-4,1) {$abcdef$};

\filldraw[black] (2,0) circle (0.1cm);
\filldraw[black] (0,2) circle (0.1cm);
\filldraw[white] (0,0) circle (0.1cm);
\draw (0,0) circle (0.1cm);
\filldraw[white] (2,2) circle (0.1cm);
\draw (2,2) circle (0.1cm);
\filldraw[white] (-.5,2.5) circle (0.1cm);
\draw (-.5,2.5) circle (0.1cm);
\filldraw[white] (2.5,-.5) circle (0.1cm);
\draw (2.5,-.5) circle (0.1cm);
\end{tikzpicture}
%+ a^2c^2ef
\hspace{15pt}
\begin{tikzpicture}[scale=0.45]
\draw[dashed] (1,1) circle (3cm);
\draw [thick] (0,0) -- (0,2);
\draw [thick] (2,0) -- (2,2);
\draw [thick] (2.5,-.5) -- (2,0);
\draw [thick] (-.5,2.5) -- (0,2);
\draw [thick] (-.5,2.5)--(-1.6,2.5);
\draw [thick] (-.5,2.5)--(-.8,3.4);
\draw [thick] (2.5,-.5)--(3.6,-.5);
\draw [thick] (2.5,-.5)--(2.8,-1.4);
\draw [thick] (2,2)--(3.1,3.1);
\draw [thick] (-1.1,-1.1)--(0,0);
\node at (-4.2,1) {$+ \, \, a^2c^2ef$};
\node [red] at (-.5,1) {\tiny 2};
\node [red] at (2.5,1) {\tiny 2};
\filldraw[black] (2,0) circle (0.1cm);
\filldraw[black] (0,2) circle (0.1cm);
\filldraw[white] (0,0) circle (0.1cm);
\draw (0,0) circle (0.1cm);
\filldraw[white] (2,2) circle (0.1cm);
\draw (2,2) circle (0.1cm);
\filldraw[white] (-.5,2.5) circle (0.1cm);
\draw (-.5,2.5) circle (0.1cm);
\filldraw[white] (2.5,-.5) circle (0.1cm);
\draw (2.5,-.5) circle (0.1cm);
\end{tikzpicture}
\hspace{15pt}
%+b^2d^2ef
\begin{tikzpicture}[scale=0.45]
\draw[dashed] (1,1) circle (3cm);
\draw [thick] (0,0) -- (2,0);
\draw [thick](0,2) -- (2,2);
\draw [thick](2.5,-.5) -- (2,0);
\draw [thick](-.5,2.5) -- (0,2);
\draw [thick](-.5,2.5)--(-1.6,2.5);
\draw [thick](-.5,2.5)--(-.8,3.4);
\draw [thick](2.5,-.5)--(3.6,-.5);
\draw [thick](2.5,-.5)--(2.8,-1.4);
\draw [thick](2,2)--(3.1,3.1);
\draw [thick](-1.1,-1.1)--(0,0);
\node at (-4.2,1) {$+ \, \, b^2d^2ef$};
\node [red] at (1,-.5) {\tiny 2};
\node [red] at (1,2.5) {\tiny 2};
\filldraw[black] (2,0) circle (0.1cm);
\filldraw[black] (0,2) circle (0.1cm);
\filldraw[white] (0,0) circle (0.1cm);
\draw (0,0) circle (0.1cm);
\filldraw[white] (2,2) circle (0.1cm);
\draw (2,2) circle (0.1cm);
\filldraw[white] (-.5,2.5) circle (0.1cm);
\draw (-.5,2.5) circle (0.1cm);
\filldraw[white] (2.5,-.5) circle (0.1cm);
\draw (2.5,-.5) circle (0.1cm);
\end{tikzpicture}
\end{equation} 
with coefficients depending on $a,\dots,f$. On the other hand, $\Web_3(N';\lam)$ is a linear combination of two webs
\begin{equation}\label{eg:Web3N'}
\begin{tikzpicture}[scale=0.45]
%\node at (0,2.2) {$1$};
%\node at (2.2,0) {$2$};
%\node at (0,-2.2) {$3$};
%\node at (-2.2,0) {$4$};
%\draw (0,0) circle (2cm);
\draw [thick](0,0) -- (0,2);
\draw [thick](2,2)--(2,0);
\draw [thick](2.5,2.5) -- (2,2);
\draw [thick](2.5,2.5)--(3.1,3.1);
\draw [thick](-.5,-.5)--(-1.1,-1.1);
\draw [thick](0,2)--(-1.6,2.5);
\draw [thick](0,2)--(-.8,3.4);
\draw [thick](-.5,-.5)--(-1.1,-1.1);
\draw [thick](2,0)--(3.6,-.5);
\draw [thick](2,0)--(2.8,-1.4);
\node at (-4,1) {$a' c'ef$};
\node [red] at (2.05,2.7) {\tiny 2};
\node [red] at (-.7,-.05) {\tiny 2};
\draw (-.5,-.5) -- (0,0);

\filldraw[black] (0,0) circle (0.1cm);
\filldraw[black] (2,2) circle (0.1cm);
\filldraw[white] (2,0) circle (0.1cm);
\draw (2,0) circle (0.1cm);
\filldraw[white] (0,2) circle (0.1cm);
\draw (0,2) circle (0.1cm);
\filldraw[white] (2.6,2.6) circle (0.1cm);
\draw (2.6,2.6) circle (0.1cm);
\filldraw[white] (-.6,-.6) circle (0.1cm);
\draw (-.6,-.6) circle (0.1cm);
%\filldraw[white] (-.5,-.5) circle (0.1cm);
%\draw (-.5,-.5) circle (0.1cm);
\draw[dashed] (1,1) circle (3cm);
\end{tikzpicture}
\hspace{20pt}
\begin{tikzpicture}[scale=0.45]
%\node at (0,2.2) {$1$};
%\node at (2.2,0) {$2$};
%\node at (0,-2.2) {$3$};
%\node at (-2.2,0) {$4$};
%\draw (0,0) circle (2cm);
\draw [thick](0,0) -- (2,0);
\draw [thick](2,2)--(0,2);
\draw [thick](2.5,2.5) -- (2,2);
\draw [thick](2.5,2.5)--(3.1,3.1);
\draw [thick](-.5,-.5)--(-1.1,-1.1);
\draw [thick](0,2)--(-1.6,2.5);
\draw [thick](0,2)--(-.8,3.4);
\draw [thick](-.5,-.5)--(-1.1,-1.1);
\draw [thick](2,0)--(3.6,-.5);
\draw [thick](2,0)--(2.8,-1.4);
\draw [thick](-.5,-.5) -- (0,0);
\node at (-4,1) {$+ \, \, b' d' ef$};
\node [red] at (2.05,2.7) {\tiny 2};
\node [red] at (-.7,-.05) {\tiny 2};
\filldraw[black] (0,0) circle (0.1cm);
\filldraw[black] (2,2) circle (0.1cm);
\filldraw[white] (2,0) circle (0.1cm);
\draw (2,0) circle (0.1cm);
\filldraw[white] (0,2) circle (0.1cm);
\draw (0,2) circle (0.1cm);
\filldraw[white] (2.6,2.6) circle (0.1cm);
\draw (2.6,2.6) circle (0.1cm);
\filldraw[white] (-.6,-.6) circle (0.1cm);
\draw (-.6,-.6) circle (0.1cm);
%\filldraw[white] (-.5,-.5) circle (0.1cm);
%\draw (-.5,-.5) circle (0.1cm);
\draw[dashed] (1,1) circle (3cm);
\end{tikzpicture}.
\end{equation}   
\end{example}

The following theorem says that the $r$-fold boundary measurements factors through Postnikov's boundary measurement map. 

\begin{thm}\label{thm:welldefined}
If $N$ and $N'$ are two networks satisfying $\tX(N) = \tX(N')$, then $\Web_r(N;\lam)  = \Web_r(N';\lam) \in \mcw_\lambda(U)$.   
\end{thm} 

That is, even though the right hand sides of \eqref{eq:Webrdefn} are different elements of~$\mcf \mcs_\lam(r)$, they are equal as tensor invariants in $\mcw_\lambda(U)$. We prove Theorem~\ref{thm:welldefined} in Section~\ref{secn:proofofmain}.

\begin{example}\label{eg:computeweightsii}
Continuing with Example~\ref{eg:computeweights}, recall that $a',b',c',d'$ are related to $a,b,c,d$ by \eqref{eq:spiderparameters}. The networks $N$ and $N'$ are related by (M1), and $\tX(N) = (ac+bd)\tX(N')$.  It follows that $\Web_3(N;\lam) = (ac+bd)^3 \, \Web_3(N';\lam)$. By equating the coefficient of $abcdef$ in \eqref{eg:Web3N} with the coefficient of $\frac{abcdef}{(ac+bd)^3}$ in \eqref{eg:Web3N'}, we deduce the \emph{square move} for $\SL_3$-webs: 
\begin{equation}\label{eg:Web3SquareMove}
\begin{tikzpicture}[scale=.45]
\draw[dashed] (1,1) circle (3cm);
\draw [thick](0,0) -- (0,2)-- (2,2)-- (2,0)-- (0,0);
\draw [thick](2.5,-.5) -- (2,0);
\draw [thick](-.5,2.5) -- (0,2);
\draw [thick](-.5,2.5)--(-1.6,2.5);
\draw [thick](-.5,2.5)--(-.8,3.4);
\draw [thick](2.5,-.5)--(3.6,-.5);
\draw [thick](2.5,-.5)--(2.8,-1.4);
\draw [thick](2,2)--(3.1,3.1);
\draw [thick](-1.1,-1.1)--(0,0);
\node at (6,1) {\large $=$};
\filldraw[black] (2,0) circle (0.1cm);
\filldraw[black] (0,2) circle (0.1cm);
\filldraw[white] (0,0) circle (0.1cm);
\draw (0,0) circle (0.1cm);
\filldraw[white] (2,2) circle (0.1cm);
\draw (2,2) circle (0.1cm);
\filldraw[white] (-.5,2.5) circle (0.1cm);
\draw (-.5,2.5) circle (0.1cm);
\filldraw[white] (2.5,-.5) circle (0.1cm);
\draw (2.5,-.5) circle (0.1cm);
\end{tikzpicture}
\hspace{20pt}
\begin{tikzpicture}[scale=0.45]
%\node at (0,2.2) {$1$};
%\node at (2.2,0) {$2$};
%\node at (0,-2.2) {$3$};
%\node at (-2.2,0) {$4$};
%\draw (0,0) circle (2cm);
\draw [thick](0,0) -- (0,2);
\draw [thick] (2,2)--(2,0);
\draw [thick](2.5,2.5) -- (2,2);
\draw [thick](2.5,2.5)--(3.1,3.1);
\draw [thick](-.5,-.5)--(-1.1,-1.1);
\draw [thick](0,2)--(-1.6,2.5);
\draw [thick](0,2)--(-.8,3.4);
\draw [thick](-.5,-.5)--(-1.1,-1.1);
\draw [thick](2,0)--(3.6,-.5);
\draw [thick](2,0)--(2.8,-1.4);
\node [red] at (2.05,2.7) {\tiny 2};
\node [red] at (-.7,-.05) {\tiny 2};
\draw [thick](-.5,-.5) -- (0,0);

\filldraw[black] (0,0) circle (0.1cm);
\filldraw[black] (2,2) circle (0.1cm);
\filldraw[white] (2,0) circle (0.1cm);
\draw (2,0) circle (0.1cm);
\filldraw[white] (0,2) circle (0.1cm);
\draw (0,2) circle (0.1cm);
\filldraw[white] (2.6,2.6) circle (0.1cm);
\draw (2.6,2.6) circle (0.1cm);
\filldraw[white] (-.6,-.6) circle (0.1cm);
\draw (-.6,-.6) circle (0.1cm);
%\filldraw[white] (-.5,-.5) circle (0.1cm);
%\draw (-.5,-.5) circle (0.1cm);
\draw[dashed] (1,1) circle (3cm);
\end{tikzpicture}
\hspace{20pt}
\begin{tikzpicture}[scale=0.45]
%\node at (0,2.2) {$1$};
%\node at (2.2,0) {$2$};
%\node at (0,-2.2) {$3$};
%\node at (-2.2,0) {$4$};
%\draw (0,0) circle (2cm);
\draw [thick](0,0) -- (2,0);
\draw [thick](2,2)--(0,2);
\draw [thick](2.5,2.5) -- (2,2);
\draw [thick](2.5,2.5)--(3.1,3.1);
\draw [thick](-.5,-.5)--(-1.1,-1.1);
\draw [thick](0,2)--(-1.6,2.5);
\draw [thick](0,2)--(-.8,3.4);
\draw [thick](-.5,-.5)--(-1.1,-1.1);
\draw [thick](2,0)--(3.6,-.5);
\draw [thick](2,0)--(2.8,-1.4);
\draw [thick](-.5,-.5) -- (0,0);
\node at (-3,1) {\large $+ $};
\node [red] at (2.05,2.7) {\tiny 2};
\node [red] at (-.7,-.05) {\tiny 2};
\filldraw[black] (0,0) circle (0.1cm);
\filldraw[black] (2,2) circle (0.1cm);
\filldraw[white] (2,0) circle (0.1cm);
\draw (2,0) circle (0.1cm);
\filldraw[white] (0,2) circle (0.1cm);
\draw (0,2) circle (0.1cm);
\filldraw[white] (2.6,2.6) circle (0.1cm);
\draw (2.6,2.6) circle (0.1cm);
\filldraw[white] (-.6,-.6) circle (0.1cm);
\draw (-.6,-.6) circle (0.1cm);
%\filldraw[white] (-.5,-.5) circle (0.1cm);
%\draw (-.5,-.5) circle (0.1cm);
\draw[dashed] (1,1) circle (3cm);
\end{tikzpicture}.
\end{equation}   
\end{example}

\medskip

%Now we will begin to set up the promised duality. 
\subsection{Immanants}
Now consider any functional $\varphi \in \mcw_\lam(U)^*$ on the $\SL_r$ tensor invariant space, where $\lam$ satisfies \eqref{eq:rcomposition}. From Theorem~\ref{thm:welldefined}, the number 
\begin{equation}\label{eq:pairing}
\varphi(\Web_r(N;\lam))
\end{equation}
is independent of the choice of network $N$ representing the point $\tX(N) \in \tGr(k,n)$. 

Now let $G$ be a bipartite graph that represents the top cell.  As we vary over networks $N$ with underlying graph $G$, the map $\tX(N) \mapsto \varphi(N)$ defines a function on the subset of $\tGr(k,n)$ swept out by $\tX(N)$'s. 

\begin{prop}\label{prop:annoying} The function $\tX(N) \mapsto \varphi(\Web_r(N;\lam))$ extends to a (uniquely defined) element of $\bbc[\Gr(k,n)]_\lam$; this function does not depend on the choice of $G$.
\end{prop}

We will prove Proposition~\ref{prop:annoying} in Section~\ref{secn:proofofmain}. It leads naturally to the following definition: 
\begin{defn}[Immanant map]\label{defn:immanants} The linear map $\Imm: \mcw_\lam(U)^* \to \bbc[\Gr(k,n)]_\lam$ defined by $\Imm(\varphi)(\tX(N)) = \varphi(\Web_r(N;\lam))$ is called the \emph{immanant map}.  
\end{defn}

Another of our main theorems is the following. 

\begin{thm}\label{thm:immanants} The immanant map $\Imm: \mcw_\lam(U)^* \to \bbc[\Gr(k,n)]_\lam$ is an isomorphism. 
\end{thm}

Thus, we obtain a canonical pairing of $\mcw_{\lam}(U)$ with $\bbc[\Gr(k,n)]_\lam$, described more explicitly in \eqref{eq:pairingdumbeddownII} and \eqref{eq:pairingdumbeddown}. In the multilinear case $\lam = (1,\dots,1)$, this pairing is the unique (up to scalars) $S_n$-equivariant pairing of the web spaces $\mcw(r,n)$ and $\mcw(k,n)$, cf.~Section~\ref{secn:duality}. 

\begin{rmk}
When $r$ is $2$ or $3$, the second author previously defined $\SL_2$- and $\SL_3$-\emph{web immanants} \cite{LamDimers}. These are obtained via \eqref{eq:pairing} when the functional $\varphi \in \mcw(r,n)^*$ is the dual functional to an element of the web basis. In the case that $r \geq 4$, since we are without a notion of a web basis, we believe the $r$-fold boundary measurement $\Web_r(N)$ is the more fundamental object. 
\end{rmk}

\section{Proof of the main theorems}\label{secn:proofofmain}

\begin{defn}[Consistent labeling] Let $W$ be an $r$-weblike subgraph of a planar bipartite graph $G$ with $n$ boundary vertices, and with boundary edges $b_1,\dots,b_n$. A \emph{consistent labeling of $W$} is a labeling of each edge $e$ in $W$ by a subset $S(e) \subset [r]$ so that $|S(e)| = m(e)$, and such that the union of the sets around each interior vertex is $[r]$. Equivalently, one can require that the sets $S(e)$ are disjoint around each vertex.
\end{defn}

Let $\ell$ be a consistent labeling. From the following pair of equalities of multisets
\begin{align*}
\bigcup_{\text{all edges}} S(e) \, &= \bigcup_{\text{white vertices}} [r] \\
\bigcup_{\text{all non-boundary edges $e$}} S(e) \, &= \bigcup_{\text{interior black vertices}} [r],
\end{align*}
it follows that the multiset $S(b_1) \cup \cdots \cup S(b_n)$ satisfies 
\begin{equation}\label{eq:SMultiset}
S(b_1) \cup \cdots \cup S(b_n) = \{1^k,2^k,\dots,r^k\}.
\end{equation}
We refer to $\mcs = (S(b_1),\dots,S(b_n))$ as a list of \emph{boundary label subsets}. 

We can also associate to any labeling $\ell$ a list of \emph{boundary location subsets} $\mci = (I_1,\dots,I_r) \subset [n]$, defined by $I_i = \{j \in [n] \colon \, i \in S(b_j)\}$. From the preceding argument, we know that each~$|I_i|$ has size~$k$. Furthermore,
\begin{equation}\label{eq:IMultiset}
I_1 \cup \cdots \cup I_r = \{1^{\lam_1},2^{\lam_2},\dots,n^{\lam_n}\}
\end{equation}
as multisets, where $\lam$ is the degree of $W$.  
 
The data of boundary labels $\mcs = (S_1,\dots,S_n) \subset [r]$ satisfying \eqref{eq:SMultiset}, and the data of boundary locations $\mci = (I_1,\dots,I_r) \subset [n]$, satisfying \eqref{eq:IMultiset}, are dual to each other (in the sense of combinatorial design theory), i.e. one can be recovered from the other. %So any recipe we write in terms of $(S_1,\dots,S_n)$ can be translated into a recipe for $(I_1,\dots,I_r)$ and vice versa. 

Recall we have a vector space $U$ with basis $E_1,\dots,E_r$ satisfying $E_1 \wedge E_2 \wedge \cdots \wedge E_r = 1$. To any subset $S \subset [r]$ with $|S| = a$ we can associate the tensor $E_S \in \bigwedge^a(U)$ by taking the wedge of the basis vectors labeled by $S$, in ascending order. If $\mcs = (S_1,\dots,S_n)$ is a list of boundary label subsets,  we obtain in this way a tensor $E_{\mcs} \in \bigotimes_{j=1}^n \bigwedge^{\lam_j}(U)$ by 
\begin{equation}\label{eq:tensorfromSs}
E_{\mcs} = E_{S_1} \otimes E_{S_2} \otimes \cdots \otimes \cdots \otimes  E_{S_n}. 
\end{equation}

It will be convenient to replace the multiset $\{1^k,2^k,\dots,r^k\}$ by an honest set $\mca = \{1_1,1_2,\dots,1_k,2_1,2_2,\dots,2_k,\dots,r_1,\dots,r_k\}$ which we refer to as the \emph{alphabet}. We introduce the obvious total order $1_1 < 1_2 < \cdots < 1_{k} < 2_1 < \cdots < r_{k}$ on $\mca$. We let $\bbc \la \mca \ra $ be the free vector space with basis $\mca$, and with volume form $1_1 \wedge 1_2 \wedge  \cdots \wedge r_k = 1 \in \bigwedge^{rk}(\bbc \la \mca \ra).$ We will think of the numbers $1,\dots,r$ as \emph{colors}: for $i_s,j_t \in \mca$ we say that $i_s$ and $j_t$ have the same \emph{color} if $i=j$. 

If we read the indices of the basis vectors in \eqref{eq:tensorfromSs} from left to right, we get a word $w'(\mcs)$ in which each of the numbers $1,\dots,r$ appears exactly $k$ times.  We let $\sgn(\mcs)  = 1$ or $-1$ depending on whether the number of inversions of $w'(\mcs)$ is even or odd. Equivalently, we can think of the entries of $w'(\mcs)$ as elements of $\mca$ by making the subscripts increase by one from left to right.
% -- thus the subsequence of of 1's in $w'$ becomes a subsequence $1_1,1_2,1_3,\dots,1_k$, the subsequence of $2$'s becomes a subsequence $2_1,\dots,2_k$, and so on. 
We let $w(\mcs) \in \mca$ denote the word obtained from $w'(\mcs)$ in this way. Then $\sgn(\mcs)$ is the value of the wedge product of $w(\mcs) \in \bigwedge^{rk}(\bbc \la \mca \ra)$. 

\begin{example}\label{eg:SL4WebEvaluationctd} Let $G$ be the bipartite graph underlying Example~\ref{eg:SL4WebEvaluation}. Here is a particular $4$-weblike subgraph $W\subset G$, and a consistent labeling $\ell$ of $W$: 
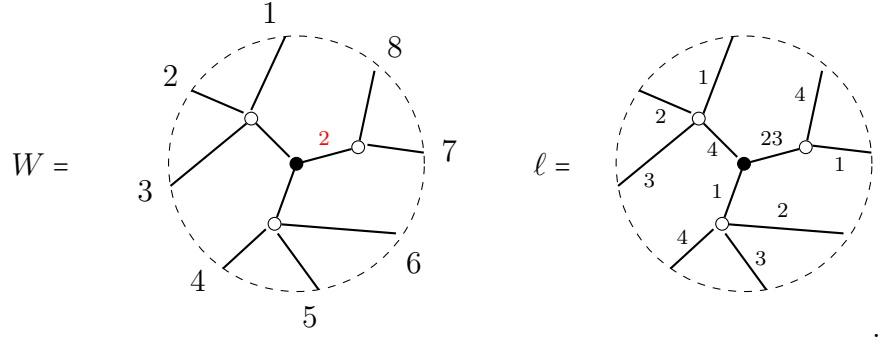
\begin{equation}\label{eq:SL4WebEvaluationctd}
\begin{tikzpicture}[scale=.85]
\node at (190:2.4cm) {3};
\node at (145:2.4cm) {2};
\node at (100:2.4cm) {1}; 
\node at (50:2.4cm) {8};
\node at (5:2.4cm) {7};
\node at (-40:2.4cm) {6};
\node at (-85:2.4cm) {5};
\node at (-130:2.4cm) {4};
\node at (-4,0) {$W = $};

\draw[dashed] (0,0) circle (2cm);
\draw[thick] (0,0)--(15:.9cm) node[pos = .5, red, above] {\tiny 2};
\draw [thick] (15:1cm)--(50:1.9cm);  %node[pos = .5, red, left] {\tiny 3};
\draw [thick] (5:2cm)--(15:1.15cm);
\draw [thick] (135:1cm)--(135:.15cm); % node[pos = .4,red, right] {\tiny 3};
\draw [thick] (95:2cm)--(130:1.1cm);
\draw [thick] (145:2cm)--(136:1.15cm);
\draw [thick] (190:2cm)--(140:1cm);
\draw [thick](0,0)--(250:.9cm);
\draw [thick] (235:2cm)-- (245:1.13cm);
\draw [thick] (280:2cm)--(255:1.13cm);
\draw [thick] (250:1cm)--(325:1.9cm); %node[pos = .5, red, above] {\tiny 3};
\filldraw[black] (0,0) circle (0.1cm);
\filldraw[white] (15:1cm) circle (.1cm);
\draw (15:1cm) circle (.1cm);
\filldraw[white] (135:1cm) circle (.1cm);
\draw (135:1cm) circle (.1cm);
\filldraw[white] (250:1cm) circle (.1cm);
\draw (250:1cm) circle (.1cm);

\begin{scope}[xshift = 7cm]
\node at (190:1.5cm) {\tiny $3$};
\node at (150:1.5cm) {\tiny $2$};
\node at (115:1.5cm) {\tiny $1$};
%\node at (50:2.4cm) {$E_4$};
\node at (0:1.5cm) {\tiny $1$};
%\node at (-40:1.5cm) {\tiny $2$};
\node at (-80:1.5cm) {\tiny $3$};
\node at (-130:1.5cm) {\tiny $4$};
\node at (-3,0) {$\ell = $};

\draw[dashed] (0,0) circle (2cm);
\draw[thick] (0,0)--(15:.9cm) node[pos = .5, above] {\tiny $23$};
\draw [thick] (15:1cm)--(50:1.9cm) node[pos = .7, left] {\tiny $4$};
\draw [thick] (5:2cm)--(15:1.1cm);
\draw [thick] (135:1cm)--(135:.15cm) node[pos = .8, left] {\tiny $4$};
\draw [thick] (95:2cm)--(130:1cm);
\draw [thick] (145:2cm)--(136:1.15cm);
\draw [thick] (190:2cm)--(140:1cm);
\draw [thick](0,0)--(250:.9cm) node[pos = .5, left] {\tiny $1$};
\draw [thick] (235:2cm)-- (245:1.13cm);
\draw [thick] (280:2cm)--(255:1.13cm);
\draw [thick] (250:1cm)--(325:1.9cm) node[pos = .5, above] {\tiny $2$};
\filldraw[black] (0,0) circle (0.1cm);
\filldraw[white] (15:1cm) circle (.1cm);
\draw (15:1cm) circle (.1cm);
\filldraw[white] (135:1cm) circle (.1cm);
\draw (135:1cm) circle (.1cm);
\filldraw[white] (250:1cm) circle (.1cm);
\draw (250:1cm) circle (.1cm);
\end{scope}
\end{tikzpicture}.
\end{equation}
The labeling $\ell$ of $W$ has boundary location subsets $\mci = (\{1,7\},\{2,6\},\{3,5\},\{4,8\})$ and boundary label subsets $\mcs = (\{1\},\{2\},\{3\},\{4\},\{3\},\{2\},\{1\},\{4\})$. The tensor $E_\mcs$ \eqref{eq:tensorfromSs} is the 8-fold tensor $E$ from Example~\ref{eg:SL4WebEvaluation}. The sign $\sgn(\mcs)$ is given by  $1_1 \wedge 2_1\wedge 3_1\wedge 4_1 \wedge 3_2 \wedge 2_2 \wedge 1_2 \wedge 4_2 = -1$. The $\SL_4$-web $\hat{W}$ from Example~\ref{eg:SL4WebEvaluation} is a tagging of $W$, obtained by directing the edges of $G$ according to \eqref{eq:SL4WebEvaluation} (cf.~Remark~\ref{rmk:perfectorientations}). The labeling $\ell$ of $W$ determines the flow of tensors along the edges of $\hat{W}$ pictured in \eqref{eq:SL4WebEvaluation}, specifically we replace $S$ in $\ell$ by $E_S$ for edges directed black to white and by $E_{[4] \backslash S}$ for edges directed white to black.   
\end{example}

As Example~\ref{eg:SL4WebEvaluationctd} suggests, consistent labelings of $W$ with boundary $\mcs$ are closely related to evaluating $\hat{W}$ on the tensor $E_\mcs$.

\begin{defn}
Denote by $a(\mcs;W)$ the number of consistent labelings of $W$ with fixed boundary label subsets $\mcs$. 
\end{defn}

The following seemingly innocuous lemma is the key technical result underpinning our main theorems. The proof of the lemma is somewhat intricate, and has minimal bearing on the rest of the paper, so it may be skipped on a first reading.

\begin{lem}\label{lem:eora}
Let $W$ be an $r$-weblike subgraph of a planar bipartite graph $G$, and $\hat{W} = \hat{W}(W,\mco)$ a choice of tagging for $W$. 
Then there is a sign $\sgn(W,\mco) \in \{\pm 1\}$, such that for any list of boundary label subsets $\mcs$, %satisfying \eqref{eq:SMultiset} 
we have 
\begin{equation}\label{eq:eoralem}
\hat{W} \big|_{E_{\mcs}} =  \sgn(\mcs) \, \sgn(W,\mco) \, a(\mcs; W).
\end{equation}
\end{lem}

It is fairly easy to see (and we elaborate on this below) that consistent labelings with boundary $\mcs$ give rise to terms in the evaluation of $\hat{W}$ on $E_\mcs$. Therefore, the key assertion in this lemma is that each of these terms contributes to the evaluation with the same sign. % Thus the lemma is a statement about sign coherence. 
 Let us also emphasize that the sign $\sgn(W,\mco)$ is independent of the boundary input $E_\mcs$. 
%This fact will play an important role later. 
\medskip

Thus for any $r$-weblike subgraph $W$, the tensor invariant $\textbf{W} = \sgn(W,\mco)\hat{W}(W,\mco) \in \mcw_\lam(U)$ of
\eqref{eq:eoraII} is characterized by the equation
\begin{equation}\label{eq:eora}
\textbf{W}\big|_{E_{\mcs}} =  \sgn(\mcs) \, a(\mcs; W).
\end{equation}
for every list of boundary label subsets $\mcs$.

\begin{proof}
Let us consider a consistent labeling $\ell$ with boundary label subsets $\mcs$. The choice of tagging on $\hat{W}$ prescribes how to evaluate $\hat{W}$ on $E_\mcs$. Our first goal is to explain why the consistent labeling $\ell$ gives a term in the evaluation of $\hat{W}$ on $E_\mcs$. This evaluation involves many terms because of the vertices where we perform the shuffle operation. 

For a given labeling $\ell$, we can place tensors on the edges of $\hat{W}$ as governed by $\ell$. If an edge (or half-edge if the edge is tagged) is labeled by the set $S$ in $\ell$, let us place the tensor $E_S$ along that edge if it is directed black to white and place the complementary tensor $E_{[r] \backslash S}$ along the edge if it is directed white to black. A tensor along a given edge in a given term represents some partial stage of the evaluation $\hat{W}$ on $E_\mcs$. Because $\ell$ is consistent, the incoming flow of basis vectors equals the outgoing flow at each interior vertex, and no two basis vectors flow through the same interior vertex. Thus, once we have placed the appropriate tensor on each edge of $\hat{W}$ as indicated by $\ell$, we get a sign for this term which contributes to the final evaluation.   

%(at least up to sign). 
%The term contributes a final sign to $\hat{W} \big|_{E_\mcs}$ once all of its edges are filled with tensors. 

%There will be one term corresponding to the consistent labelling $\ell$. If an edge (or half-edge if the edge is tagged) is labelled by the set $S$ in $\ell$, let us place the tensor $E_S$ along that edge if it is directed black to white and place the complementary tensor $E_{[r] \backslash S}$ along the edge if it is directed white to black. A tensor along a given edge in a given term represents some partial stage of the evaluation $\hat{W}$ on $E_\mcs$. %(at least up to sign). 
%The term contributes a final sign to $\hat{W} \big|_{E_\mcs}$ once all of its edges are filled with tensors. 

%Then every tensor along an edge represents a term in the evaluation of $\hat{W}$ on $E_\mcs$ at some partial stage of the evaluation (at least up to sign).

%Thus each consistent labelling corresponds to a term in the evaluation $\hat{W}\big{|}_{E_\mcs}$. This term occurs with coefficient $\pm 1$. 
Thus, each consistent labeling corresponds to a term in $\hat{W}\big{|}_{E_\mcs}$, and the nonzero terms in the evaluation of $\hat{W}$ on $E_\mcs$ are counted by $a(\mcs; W)$. The nontrivial assertion is that once $\mcs$ is fixed, every term contributes with the same sign. The content of the proof lies in understanding these signs. The signs come from two places: the shuffles at interior vertices and the pairings at each tag.

We begin our sign analysis by associating to a consistent labeling $\ell$ a flow of vectors along the directed edges of $\hat{W}$. Specifically, if $i \in S \subset [r]$ and $E_S$ is at a given edge (or half-edge) of this term, then we say that the basis vector $E_i$ flows along that edge. The evaluation begins with tensors $E_{S_1},\dots,E_{S_n}$ at the boundary vertices. Each of these $E_{S_i}$ is itself a wedge product of basis vectors. By keeping track of how such a boundary basis vector flows along the web, we get a path from the boundary to a tag. 

In this way, the labeling $\ell$ provides us with $rk$ paths starting at boundary sources and ending at tags. These paths are naturally indexed by the multiset $\{1^k,\dots,r^k\}$, and we index them by elements of the alphabet $\mca$ by adding subscripts in counterclockwise order along the boundary. 

For the sake of simplicity, we will assume that there are no interior sources in $\hat{W}$, so that the term is completely described by the union of these $rk$ paths (if there are interior sources, the argument below can be repeated with slight modifications). 

Let us take these $rk$ paths from the boundary to the tags and isotope them slightly so that they are disjoint at the boundary, and are noncrossing except in transverse crossings in small neighborhoods of interior shuffle vertices. Thus when $a$ of these paths flow along an edge of multiplicity $a$, the $a$ paths have a natural order from left to right. For each tag there will be exactly $r$ paths ending at the tag.

\begin{defn*} A \emph{flow} on a tagged web $\hat{W}$ is a collection of $rk$ paths starting at the boundary and ending at tags, where the number of paths flowing along each edge is the multiplicity of the edge, and where the paths are labeled by the alphabet $\mca$. We will usually isotope these paths so that we get $a$ parallel paths along each edge of multiplicity $a$ and so that the crossings of these paths are concentrated in a neighborhood of shuffle vertices. 
\end{defn*}

To summarize, a consistent labeling $\ell$ provides us with a flow, which we will also call $\ell$. The consistent labeling from Example~\ref{eg:SL4WebEvaluationctd}, which corresponds to the term in the evaluation computed in Example~\ref{eg:SL4WebEvaluation}, has the following flow:
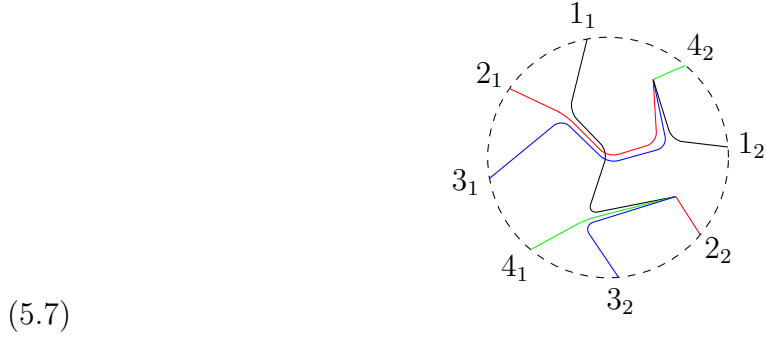
\begin{equation}\label{eq:floweg}
\begin{tikzpicture}[scale = .8]
\node at (190:2.4cm) {$3_1$};
\node at (145:2.4cm) {$2_1$};
\node at (100:2.4cm) {$1_1$};
\node at (50:2.4cm) {$4_2$};
\node at (5:2.4cm) {$1_2$};
\node at (-40:2.4cm) {$2_2$};
\node at (-85:2.4cm) {$3_2$};
\node at (-130:2.4cm) {$4_1$};
\draw[dashed] (0,0) circle (2cm);

\draw [rounded corners] (100:2cm)--(130:1cm)--(0,.1)--(250:1cm)--(330:1.3cm);
\draw [rounded corners,red] (145:2cm)--(135:1cm)--(0,0)--(17:.85cm)--(60:1.5cm);
\draw [rounded corners,blue] (190:2cm)--(140:1cm)--(0,-.1)--(10:1cm)--(60:1.5cm);
\draw [rounded corners,green] (-130:2cm)--(250:1.1cm)--(330:1.3cm);
\draw [rounded corners,blue] (-85:2cm)--(250:1.2cm)--(330:1.3cm);
\draw [rounded corners] (5:2cm)--(15:1.1cm)--(60:1.5cm);
\draw [red] (-40:2cm)--(330:1.3cm);
\draw [green] (50:2cm)--(60:1.5cm);
\end{tikzpicture}
\end{equation}

We will now consider flows which do not necessarily come from consistent labelings. In a general flow, it is not required that the paths along each edge are of different colors (likewise for the paths flowing into a tag). This requirement holds precisely for flows coming from consistent labelings. As an intermediate step, we will associate a sign $\eps(\ell) = \eps(\ell,\mco)$ to \emph{any} flow $\ell$ on $\hat{W}$, not necessarily coming from a consistent labeling. The sign $\eps(\ell)$ is a topological invariant of the $rk$ paths from the boundary to the tags. It knows nothing about structure of the set $\mca$ other than that it consists of $rk$ elements (and thus it does not know about the colors of these $rk$ elements).

Let $t_1,\dots,t_k$ be a list of the tags in $\hat{W}$. For each tag $t_i$, we let $C_i \in \binom \mca r$ be the set of labels of the $r$ paths ending at $t_i$. We refer to $C_i$ as a \emph{tag subset}. Within each tag subset, the terms are ordered (from left to right) according to how they flow into the tag, where we recall that if $x$ and $y$ are complementary tensors flowing into a tag and $x$ is left of $y$ (when the tag points north) then the tensors are ordered $x \wedge y$. The ordered set $C_i$ naturally gives rise to an element of $\bigwedge^{r}\bbc \la \mca \ra$, which by abuse of notation we also call $C_i$. Then the sign $\eps(\ell)$ is defined by 

$$\eps(\ell) := C_1 \wedge C_2 \wedge \cdots \wedge C_k \in \bigwedge^{rk}\bbc \la \mca \ra.$$

For any list of boundary subsets $\mcs$, there is a \emph{canonical} way of flowing according to $\mco$, namely the one in which strands cross maximally at each interior shuffle vertex.  According to \eqref{eq:shufflesigns}, there is no sign at interior shuffle vertices associated with such a flow. We let $\ell_{\text{canon}}(\mcs)$ denote the canonical flow with boundary $\mcs$. This canonical flow is a device that gives us a reference point for analyzing signs. %For different choice of $\mcs$, the labelings $\ell_{\text{canon}}(\mcs)$ will be different, but ``isomorphic.''

Let $\mcs_0$ be list of boundary label subsets whose word is lexicographically smallest, i.e. $w(\mcs_0) = 1_1 \, 1_2 \, 1_3 \, \dots \, r_k$. For any other choice of boundary label subsets $\mcs$, the canonical flows have signs related by 
\begin{equation}\label{eq:signsI}
\eps(\ell_{\text{canon}}(\mcs_0)) =  \eps(\ell_{\text{canon}}(\mcs))\, \sgn(\mcs).
\end{equation}
Now suppose we have a consistent labeling $\ell$ with boundary subset $\mcs$. We can get from $\ell_{\text{canon}}(\mcs)$ to $\ell$ by performing a sequence of swaps at shuffle vertices. Clearly, each such swap changes the sign $\eps$ by a factor of $-1$. We have that 
\begin{equation}\label{eq:signsII}
\eps(\ell) = \eps(\ell_{\text{canon}}(\mcs)) \, (-1)^{\text{$\#$ of shuffle swaps $\ell_{\text{canon}}(\mcs) \to \ell$}} = \eps(\ell_{\text{canon}}(\mcs)) \, (\text{sign of shuffle swaps}).
\end{equation}

Now we compute $\eps(\ell)$ by computing the number of transpositions required to transform $\eps(\ell)$ into  $1_1 \wedge 1_2 \wedge \cdots \wedge r_k$. First, within each tag subset, we can rearrange the terms so that they are increasing in $\mca$. We will refer to this as the \emph{inversions within tags}. In addition to alphabetizing within each tag, we must perform a transposition for each inversion involving two elements of different tag subsets. We will refer to this latter number of inversions as the \emph{inversions between tags}. Thus, we have that 
\begin{equation}\label{eq:signsIII}
\eps(\ell) = (-1)^{\text{inversions within tags}} \, (-1)^{\text{inversions between tags}} = (\text{sign within tags}) \cdot  (\text{sign between tags}). 
\end{equation}

We need one final observation: the contribution of $\ell$ to $\hat{W} \big|_{E_S}$ has sign equal to  (sign of shuffle swaps) times (sign within tags). Putting this together with \eqref{eq:signsI} through \eqref{eq:signsIII}, the sign of the contribution of $\ell$ is thus
\begin{equation}\label{eq:signsIV}
(\text{sign within tags})\,(\text{sign of shuffle swaps}) = \eps(\ell_{\text{canon}}(\mcs_0)) \, (\text{sign between tags})\, \sgn(\mcs). 
\end{equation}
Thus the proof is finished once we show the following: the sign between tags is the same for \emph{any} consistent labeling $\ell$ of $\hat{W}$, regardless of the choice of boundary label subsets $\mcs$. Recall that for a consistent labeling $\ell$ the $rk$ paths are labeled by the alphabet $\mca$ by adding subscripts in counterclockwise order. 

This last claim follows from two subclaims: first, since $\ell$ is a consistent labeling, all of the tag subsets are copies of $[r]$. Thus, the number of inversions between tags involving \emph{elements of a different color} does not depend on $\ell$. For example, if $4$ is in the tag subset $C_3$, then it will be in inversion with the copies of $1,2,3$ in the tag subsets $C_4,C_5,\dots$. Second, we need to check that the total number of inversions \emph{between elements of the same color} (and in different tags) does not depend on the consistent labeling $\ell$. Thus we have reduced the proof of the lemma to the second subclaim, which we now state formally:

\begin{claim*} Fix a tagged web $\hat{W}$. Consider any consistent labeling $\ell$ of $\hat{W}$. Suppose that $\ell$ has boundary label subset $\mcs$. Adding subscripts in counterclockwise order, we get a flow on $\hat{W}$ with paths labeled by $\mca$. Then the signs between tags for elements of $\mca$ of the same color is independent of $\ell$ (and hence $\mcs$).
\end{claim*}

Up to this point, the sign analysis has been mostly ``soft'' reasoning. Proving this second subclaim requires using the planarity of $\hat{W}$ in an essential way. 

Let us cut the boundary of the disk between boundary vertices $n$ and $1$, flattening the web $\hat{W}$ so that the boundary vertices are drawn on the $x$-axis in $\bbr^2$ in the order $1,\dots,n$, and the web $\hat{W}$ is drawn in the upper half plane.  The flow $\ell$ gives $rk$ directed paths from the $x$-axis to the tags, colored by the elements of $[r]$.  Let us make the assumption that the $x$-coordinates of the boundary vertices, the tags $t_i$, and the intersection points between the $rk$ paths are pairwise distinct.  Let us assume furthermore that the tags are ordered $t_1,\dots,t_k$ by the $x$-coordinates, from left to right. 

Fixing the boundary data $\mcs$, the flow $\ell_{\text{canon}}(\mcs)$ is the one in which these paths cross each other maximally, and any consistent labeling $\ell$ with boundary $\mcs$ is obtained by resolving certain of the crossings in $\ell_{\text{canon}}(\mcs)$. Furthermore, a key observation that we use later is that if the labeling is consistent, then two paths of the same color never intersect. 
%
%For each tag $t_i$ choose a path $Z_i$, called the \emph{zipper}, from $t_i$ to the $x$-axis, in such a way that the zipper avoids vertices of $\hat{W}$ and intersects edges of $\hat{W}$ transversely. We can arrange that the zippers do not cross each other. 
%
%We analyze the number of inversions between elements of the same color by considering a \emph{nesting number} for any flow. If $p$ is one of the $rk$ paths and $p$ ends at tag $t_i$, we denote by $pZ_I$ the augmented path obtained by concatenating $p$ with the zipper $Z_i$. The augmented path $pZ_i$ starts and ends on the $x$-axis. It cuts the upper half plane into bounded regions and an unbounded region. There may be more than one bounded region if $p$ crosses $Z_i$ on it way to the tag $Z_i$, but there is only one unbounded region.

For each pair $(p,t_j)$ of a path $p$ ending at $t_i$, and a tag $t_j$ (where $i \neq j$), we define the \emph{nesting number} $n(p,t_j)$ of the pair $(p,t_j)$ to be the total number of intersection points of $p$ with the vertical ray $R_j$ going upwards from $t_j$ to $\infty$. We define $n(p,t_i) = 0$ if $p$ ends at $t_i$.  
%If $p$ crosses $R_j$ from left to right it is counted with $+1$, and  counted 
For any labeling $\ell$, we can define the \emph{total nesting number} of $\ell$ as the sum of the nesting numbers for all pairs $(p,t_j)$.

We now argue that when we resolve a crossing in a flow, the total nesting number does not change.  Let $p, q$ be two paths ending at $t_i$ and $t_j$ respectively.  Since we have assumed that $\hat W$ has no oriented cycles, there are no oriented cycles in the flow.  It follows that 
\begin{align} \label{eq:inter1}
&\mbox{$p$ (resp. $q$) has no self intersections, and} \\
 \label{eq:inter2}
&\mbox{the intersection points of $p \cap q$ appear in the same order on $p$ and $q$.}
\end{align}  When we resolve a crossing of $p$ and $q$,  the quantity $n(p,t_k) + n(q,t_k)$ for $k \neq i,j$ clearly remains unchanged. Let us decompose the path $p$ (resp. $q$) into two segments $p_bp_a$ (resp. $q_bq_a$) so that the crossing happens when $p_a$ ends and $p_b$ begins (likewise for $q_a$ and $q_b$). We will be done if we show that $n(p,t_j)+n(q,t_i) \equiv n(q_bp_a,t_i)+n(p_bq_a,t_j) \mod 2$. We can decompose $n(p,t_j)$ as $\# p_a \cap R_j + \# p_b \cap R_j$ and so on for the other paths. Certain of these contributions cancel modulo 2, and our claim reduces to checking that 
\begin{equation}
\# (p_a\cup q_a) \cap (R_i \cup R_j) \equiv 0 \mod 2. 
\end{equation}
Indeed, each of $\# (p_a\cup q_a) \cap (R_i)$ and  $\# (p_a\cup q_a) \cap (R_i)$ is even, which follows from the fact that $t_j$ (resp. $t_i$) cannot lie inside a bounded region enclosed by $p_a, q_a$ and the $x$-axis by \eqref{eq:inter1} and \eqref{eq:inter2}.

%Likewise, the sum of nesting numbers $n(p,t_j)+n(q,t_j)$ (resp. $n(q,t_i)+n(p,t_i)$) is also unchanged because $t_j$ (resp. $t_i$) cannot lie inside a bounded region enclosed by $p, q$ and the $x$-axis by \eqref{eq:inter1} and \eqref{eq:inter2}. 

Thus the total nesting number of any consistent labeling $\ell$ with boundary $\mcs$ is the same as the total nesting number of $\ell_{\text{canon}}(\mcs)$.
The total nesting number of $\ell_{\text{canon}}(\mcs)$ is independent of $\mcs$ because the flow of the paths for $\ell_{\text{canon}}$ is defined topologically (without regard to color).

Moreover, because $\ell$ is a consistent labeling, and therefore two paths of the same color never cross, the total nesting number coincides with the number of inversions between elements of the same color modulo $2$. This completes the proof of the claim and the lemma.
\end{proof}

\begin{figure}
\begin{tikzpicture}
\node (A) at (2.5,2) {$ \ast$};
\node (B) at (5.5,3) {$\ast$};
\node (C) at (8.5,3) {$\ast$};
\draw [dashed] (A)--(2.5,4);
\draw [dashed] (B)--(5.5,4);
\draw [dashed] (C)--(8.5,4);

\node at (0,-.5) {$1_1$};
\node at (1,-.5) {$1_2$};
\node at (2,-.5) {$2_1$};
\node at (3,-.5) {$4_1$};
\node at (4,-.5) {$3_1$};
\node at (5,-.5) {$3_2$};
\node at (6,-.5) {$3_4$};
\node at (7,-.5) {$2_2$};
\node at (8,-.5) {$1_3$};
\node at (9,-.5) {$4_2$};
\node at (10,-.5) {$2_3$};
\node at (11,-.5) {$4_3$};

\draw [black, thick, rounded corners] (0,0)--(1,1.2)--(2,1.5)--(A);
\draw [black, thick, rounded corners] (1,0)--(1.7,1)--(2.2,1.5)--(A);
\draw [green, thick, rounded corners] (3,0)--(2.8,1)--(2.7,1.5)--(A);
\draw [blue, thick, rounded corners] (4,0)--(3.5,1)--(3,1.5)--(A);

\draw [red, thick, rounded corners] (2,0)--(2,.5)--(1,1.5)--(1.5,2.7)--(2,3)--(3,3.5)--(4,3.3)--(B);
\draw [black, thick, rounded corners] (8,0)--(7,1.5)--(6,2.5)--(B);
\draw [blue, thick, rounded corners] (5,0)--(5.15,1)--(5.3,2)--(B);
\draw [green, thick, rounded corners] (9,0)--(7.5,1.5)--(6.5,2.5)--(B);

\draw [red, thick, rounded corners] (7,0)--(7.7,1.5)--(8.2,2.5)--(C);
\draw [red, thick, rounded corners] (10,0)--(9.2,1.5)--(8.8,2.5)--(C);
\draw [green, thick, rounded corners] (11,0)--(10.15,1)--(9.5,2.8)--(C);
\draw [blue, thick, rounded corners] (6,0)--(7,1.5)--(8,2.5)--(C);

\end{tikzpicture}
\caption{An example of a flow that could come from an $\SL_4$-web on 12 vertices. The $*$'s are tags and the dashed lines are the rays. The total number of nestings for this flow is $1$, witnessed by the path through $2_1$ going above the first tag. The flow is not consistent because two black strands end at tag~1. It can be made consistent (for example) by resolving the crossing between strands $1_1$ and $2_1$ and also between $2_2$ and $1_3$, which does not change the total nesting number. Once these two resolutions are made, the total number of inversions between colors is 1 (witnessed by the inversion between $1_2$ and $1_1$), which matches the initial total number of nestings.   
\label{fig:concepts}}
\end{figure}
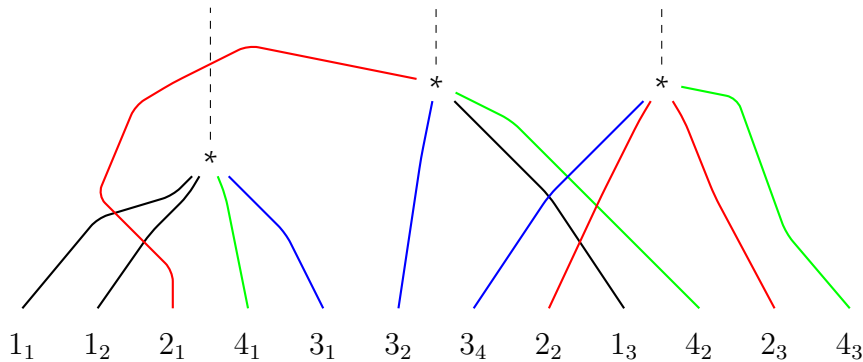

%With our technical lemma in hand, we can now state the core result of this paper, from which all our other results follow readily.
The following result relates tensor evaluation of webs to boundary measurements of networks; it serves as the fundamental link between webs and dimers. 
%Note that $\mci$ and $\mcs$ play dual roles on each side of the equation.

\begin{prop}\label{lem:PluckerstoWebr} 
Let $N$ be a network with excedance $k$. Let $\mci = (I_1,\dots,I_r)$ be a list of subsets satisfying \eqref{eq:IMultiset}, dual to $\mcs = (S_1,\dots,S_n)$ satisfying \eqref{eq:SMultiset}. Then 
\begin{equation}\label{eq:PluckerstoWebr}
\Delta_{I_1}(N) \cdots \Delta_{I_r}(N) = \sgn(\mcs) \Web_r(N;\lam) \big|_{E_\mcs}.
\end{equation} 
\end{prop} 

%Note that this is the first time that we see an equation involving both webs and dimers. It relates tensor evaluation of webs to boundary measurements of networks. This proposition serves as the fundamental link between webs and dimers. Note that $\mci$ and $\mcs$ play dual roles on each side of the equation.

We still find the statement of Proposition~\ref{lem:PluckerstoWebr} surprising, yet the importance of this statement for us is belied by the simplicity of the proof.

\begin{proof}
% Let $T$ be a tabloid with row sets $I_1,\dots,I_r$. Choose a basis $\mcb$ for $\SL_r$ web space. For each basis element $\omega \in \mcb$, let $\omega^*$ denote the dual basis element with respect to $\mcb$. 
We have that
\begin{align*}
\Delta_{I_1}(N) \cdots \Delta_{I_r}(N) &= \sum_{\text{$r$-weblike subgraphs $W$, degree } =\lam} a(\mcs; W)\wt(W) \\
&= \sgn(\mcs) \sum_{W} \textbf{W}\wt(W) \\
&= \sgn(\mcs) \Web_r(N;\lam) \big|_{E_\mcs}.
\end{align*}
The first equality follows by considering superpositions of dimer configurations $(\pi_1,\dots,\pi_r)$ with boundary subsets $(I_1,\dots,I_r)$ and grouping them according to the $r$-weblike subgraph they determine. The second equality follows from Lemma~\ref{lem:eora}, and the third equality from the definition \eqref{eq:Webrdefn}. 
\end{proof}

\begin{proof}[Proof of Theorem~\ref{thm:welldefined}] 
From \eqref{eq:PluckerstoWebr}, the evaluation of $\Web_r(N;\lam)$ on a basis tensor $E_\mcs$ can be expressed in terms of boundary measurements, which only depend on $\tX(N)$. Evaluation at $E_\mcs$ for varying $\mcs$ determines an element of $\mcw_\lam(U)$. This proves Theorem~\ref{thm:welldefined}.
\end{proof}

\begin{proof}[Proof of Proposition~\ref{prop:annoying} and Theorem \ref{thm:immanants}] 
First note that for each list of boundary label subsets $\mcs$, there is a functional $\text{eval}(E_\mcs)  \in \mcw_\lam(U)^*$ given by evaluation at $E_\mcs$. Furthermore, by \eqref{eq:PluckerstoWebr}, the function $\tX(N) \to \text{eval}(E_\mcs) (\Web_r(N; \lam))$ agrees with the function $\tX(N) \mapsto \sgn(\mcs) \Delta_{I_1}(N) \cdots \Delta_{I_r}(N)$. This latter function is a regular function on $\tGr(k,n)$. Since the set of points of the form $\tX(N)$ is Zariski dense in $\tGr(k,n)$, the function $\tX(N) \to \text{eval}(E_\mcs) (\Web_r(N; \lam))$ extends uniquely to a regular function on all of $\tGr(k,n)$. The resulting extension is the image of the immanant map. Now we note the functionals $\text{eval}(E_\mcs) $ span $\mcw_\lam(U)^*$, so we conclude that every element of $\varphi \in \mcw_\lam(U)^*$ gets sent to a a regular function on $\tGr(k,n)$ by the immanant map.  This proves Proposition~\ref{prop:annoying}.

Since $r$-fold products $\Delta_{I_1}\cdots \Delta_{I_r}$ of Pl\"ucker coordinates span the graded piece of $\bbc[\Gr(k,n)]_\lam$, the immanant map is surjective. A dimension count verifies that the map is an isomorphism. We will see another argument for the equality of dimensions in Section~\ref{secn:duality}.  This proves Theorem \ref{thm:immanants}.
%\begin{proof}[Proof of Theorem \ref{thm:immanants}]
%Let $\text{eval}(E_\mcs) \in \mcw_\lam(U)^*$ denote the functional on web space defined by evaluation at $E_\mcs$. Then the immanant map sends the functional $\text{eval}(E_\mcs)$ to the function on networks 
%$N \mapsto \sgn(\mcs) \Delta_{I_1}(N) \cdots \Delta_{I_r}(N)$. 
\end{proof}

Finally, let us extract some consequences.  First, the pairing $\mcw_\lam(U) \otimes \bbc[\Gr(k,n)]_\lam \to \bbc$ can be given in simple terms as follows:
\begin{align}
\la \textbf{W}, \Delta_{I_1} \cdots \Delta_{I_r}\ra &= \sgn(\mcs) \textbf{W} \big|_{E_\mcs} = a(\mcs;W) \label{eq:pairingdumbeddownII} \\
\la \Web_r(N;\lam),f \ra &= f(\tX(N)) \label{eq:pairingdumbeddown}.
\end{align}
In \eqref{eq:pairingdumbeddownII} we let $(I_1,\dots,I_r)$ be a list of subsets, dual to $\mcs$, and let $W$ be an $r$-weblike subgraph with tensor invariant $\textbf{W} \in \mcw_\lam(U)$.  In \eqref{eq:pairingdumbeddown} we let $f \in \bbc[\Gr(k,n)]_\lam$ and $N$ be any network of excedance $k$.  Note in particular that \eqref{eq:pairingdumbeddownII} is always nonnegative. The equation \eqref{eq:pairingdumbeddown} follows from checking the special case where $f=\Delta_{I_1} \cdots \Delta_{I_r}$, which was the content of Proposition~\ref{lem:PluckerstoWebr}.

Second, we can give another interpretation of the duality between $\bbc[\Gr(k,n)]_\lam$ and $\mcw_\lam(U)$. Let $G$ be a bipartite graph representing the top cell. Let us pick a collection of edges of $G$ so that varying the weights $w_e$ of those edges (keeping the rest of the weights as 1) gives a a rational parametrization $(\bbc^\times)^{k(n-k)} \to \Gr(k,n)$.  We now think of the weights $w_e$ as variables, and form the sum
\begin{equation}\label{eq:Webrdefn2}
\Web_r(G;\lam) := \sum_{W \subset G \, \lam(W) = \lam} \wt(W) \, \textbf{W}. 
\end{equation}
Here, $\wt(W)$ are monomials in the edge weight variables $w_e$. Specializing the edge-weights to complex numbers gives an element of $\mcw_\lam(U)$. Thus we get elements of $\mcw_\lam(U)$ depending algebraically on points of $\tGr(k,n)$. Our above analysis shows that we can then view $\Web_r(G;\lam)$ as giving rise to an element
$$\Web_r(G;\lam) \in \bbc[\Gr(k,n)]_\lam \otimes \mcw_\lam(U)$$
which \emph{is} the pairing between these two spaces. 
%Note that $\wt(W)$ are in general \emph{not} functions on $\tGr(k,n)$, so that this interpretation of $\Web_r(G;\lam)$ is slightly subtle. Rather, 
We have that
$$\Web_r(G;\lam) \big|_{E_\mcs} = \sgn(\mcs) \Delta_{I_1} \cdots \Delta_{I_r}$$
for any boundary label subsets $\mcs$, and in fact, this characterizes $\Web_r(G;\lam) \in  \bbc[\Gr(k,n)]_\lam \otimes \mcw_\lam(U).$ In particular, we emphasize that $\Web_r(G;\lam) $ is independent of $G$ (as long as $G$ represents the top cell).

Since $\Web_r(G;\lam)$ realizes the duality between $\bbc[\Gr(k,n)]_\lam$ and $\mcw_\lam(U)$, it must be a full-rank tensor in $\bbc[\Gr(k,n)]_\lam \otimes \mcw_\lam(U)$. From this, it can be seen that as $N$ varies, the various $\Web_r(N;\lam)$ span $\mcw_\lam(U)$. This will also follow from Theorem~\ref{thm:positroidspanning}.

\medskip 
Let us also formulate the following positivity conjecture.

\begin{conj}\label{posconj}
Under the isomorphism $\mcw_\lam(U) \cong \bbc[\Gr(r,n)]_\lam$, the functions {\bf W} take nonnegative values on the cone $\tGr(r,n)_{\geq 0}$ over the totally nonnegative Grassmannian.
\end{conj}

In Section~\ref{secn:positroids}, we show that webs in $\mcw(r,n)$ are naturally dual to certain elements of $\bbc[\Gr(k,n)]$, and that this duality behaves well under restriction to positroid subvarieties $\Pi \subset \Gr(k,n)$. Our positivity conjecture~\ref{posconj} says that $\SL_r$-webs are positive when thought of as functions on the Grassmannian $\Gr(r,n)$. To briefly explain the reasoning, we first note that the positivity property of Conjecture \ref{posconj} is known to hold for the canonical basis by the work of Lusztig \cite{Lusztig}.  In the $r = 2$ or $r = 3$ cases, \textbf{W} expands positively as a sum of non-elliptic web basis elements.  These web basis elements are known to coincide with the canonical basis for $r = 2$, and share many similar properties with the canonical basis in the case $r = 3$. Finally, we view the sign coherence in Lemma~\ref{lem:eora} as a manifestation of positivity. 
 
\begin{example}
As a simple example of Conjecture \ref{posconj}, consider the \emph{single cycle web} $W \in \mcw(3,9)$ (the second web in the first row of Figure~\ref{fig:dualitypics}), with its boundary vertices numbered so that vertices $1,4,7$ connect to the interior hexagon of $W$. Multiplying by the Pl\"ucker coordinate $\Delta_{479}$ and using the skein relations, one sees that $\Delta_{479} \cdot W$ is a sum of two cluster monomials. It follows that $W$ can be expressed as a Laurent monomial with positive coefficients in Pl\"ucker coordinates for $\bbc[\Gr(3,9)]$, and therefore takes positive values on $\tGr(3,9)_{\geq 0}$. See \cite[Proposition 16]{TensorsII}), 
for a more general version of this argument (that does not explicitly mention positivity) in the case $r=3$.  
\end{example}

\section{Deducing  skein relations from moves on networks}\label{secn:skein}
According to Theorem~\ref{thm:welldefined}, the element $\Web_r(N)$ only depends on $\tX(N)$. Thus it is unchanged when we perform the local moves from Theorem~\ref{thm:connectedness} to $N$. However, the \emph{expression} for $\Web_r(N)$ as a weighted sum of $r$-weblike subgraphs will change after each local move, since the weblike subgraphs change. In this way, we obtain relations between certain linear combinations of $\SL_r$-webs. 

%Furthermore, each local move on networks encodes \emph{several} relations amongst webs by thinking of each $\wt(W)$ as a monomial in the edge weights of $N$, and grouping terms in $\Web_r(N)$ according to their monomial, as we did in Example~\ref{eg:computeweightsii}. 

As it turns out, the diagrammatic relations on webs we obtain in this way are exactly the Cautis-Kamnitzer-Morrison relations \cite{SkewHowe}. This will be the content of Theorem~\ref{thm:kernel}] and Corollary~\ref{cor:rwebker}. We can interpret this fact in two ways: On the one hand, it is possible (though perhaps tedious) to use the Cautis-Kamnitzer-Morrison relations to give an alternative proof Theorem~\ref{thm:welldefined} -- one checks that if $N$ and $N'$ are related by a local move on networks, then $\Web_r(N)$ and $\Web_r(N')$ differ by a skein relation. We believe that our approach to Theorem~\ref{thm:welldefined}, based on Proposition \ref{lem:PluckerstoWebr} and Lemma~\ref{lem:eora}, is somewhat more illuminating. On the other hand, in this paper we will use Theorem~\ref{thm:welldefined} to derive diagrammatic relations amongst webs, and use our duality setup to prove that the relations we obtain in this way are a complete set of diagrammatic relations.

Our $r$-weblike graphs are closely related to, but slightly different from the webs considered in \cite{SkewHowe}. These differences require some care when comparing moves on bipartite graphs and diagrammatic relations on webs. Thus we will begin with diagrammatic relations on $r$-weblike subgraphs before proceeding to the case of webs.

\subsection{$r$-weblike graph relations} We begin by illustrating how the various skein relations on $r$-weblike graphs can be derived from relations on networks. 

Consider the space $\mcf \mcw_{\lam}(r)$ of finite formal $\bbc$-linear combinations of $r$-weblike graphs (i.e, labeled bipartite graphs such that the sum of labels around every interior vertex equals~$r$). We will later compare this with the space $\mcf \mcs_{\lam}(r)$ of formal $\bbc$-linear combinations of $\SL_r$-web diagrams.

There is a surjection $\mcf \mcw_{\lam}(r)\surjects \mcw_\lam(U)$ from the evaluation equation \eqref{eq:eora}, that is, by sending a $r$-weblike graph $W$ to the tensor invariant given by its boldface version $\textbf{W}$.

We begin by listing all the diagrammatic moves that can be performed to an $r$-weblike graph $W$ without changing the value of the tensor invariant~$\textbf{W}$. Of these moves, the last move is the only truly interesting one.
\begin{itemize}
\item Two-valent vertex removal \eqref{eq:nonsquaremove}. There is also an analogous move to \eqref{eq:nonsquaremove} which has the colors reversed. For simplicity, we drew all neighboring edges in \eqref{eq:nonsquaremove} with multiplicity $1$, though in general they may have any multiplicity.
%comes from (M2), encodes associativity of wedge
\item Bigon removal \eqref{eq:nonsquaremove}. Parallel edges, with multiplicities $a$ and $b$, can be replaced by one edge with multiplicity $a+b$, at the cost of a multplicative factor $\binom {a+b} b$. As before, colors can be reversed. 
\item Leaf and dipole removal. Edges with multiplicity $r$ or $0$ can be removed from $r$-weblike graphs. %comes from (R1) 
\item The square move for $r$-weblike graphs \eqref{eq:squaremove}. Notice that the diagrams in \eqref{eq:squaremove} connect to the ``outside'' with multiplicity $j$ from the southwest, $\ell$ from the southeast, $r+v-s-\ell$ from the northeast,  and $r+s-v-j$ from the northwest. If $j,\ell,v$ are fixed, there is a square move for each $s$ satisfying $\max(0,v-\ell,v+j-r) \leq s \leq \min(j,r-\ell)$. 
\end{itemize} 

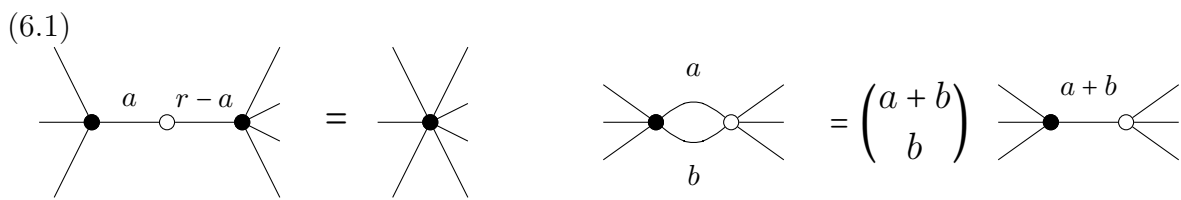
\begin{equation}\label{eq:nonsquaremove}
\begin{tikzpicture}[scale =1]
\draw (-1,0)--(1,0);
\draw (-1.5,1)--(-1,0)--(-1.7,0);
\draw (-1.5,-1)--(-1,0);
\draw (1.5,1)--(1,0)--(1.5,.25);
\draw (1.5,-.25)--(1,0)--(1.5,-1);
\node at (-.5,.3) {\small $a$};
\node at (.5,.3) {\small $r-a$};
%\node at (-1.3,1.15) {$b_1$};
%\node at (-1.6,.3) {$b_2$};
%\node at (-1.7,-.7) {$b_3$};
%\node at (1.8,1.1) {$b_4$};
%\node at (1.8,.5) {$b_5$};
%\node at (1.8,-.5) {$b_6$};
%\node at (1.8,-1.1) {$b_7$};

\filldraw[black] (-1,0) circle (0.1cm);
\draw (-1,0) circle (0.1cm);
\filldraw[black] (1,0) circle (0.1cm);
\draw (1,0) circle (0.1cm);
\filldraw[white] (0,0) circle (0.1cm);
\draw (0,0) circle (0.1cm);

\node at (2.25,0) {\Large $=$};
\draw (3,1)--(3.5,0)--(2.8,0);
\draw (3,-1)--(3.5,0);
\draw (4,1)--(3.5,0)--(4,.25);
\draw (4,-.25)--(3.5,0)--(4,-1);
\filldraw[black] (3.5,0) circle (0.1cm);
\draw (3.5,0) circle (0.1cm);
%\node at (3.2,1.15) {$b_1$};
%\node at (2.8,.4) {$b_2$};
%\node at (2.8,-.5) {$b_3$};
%\node at (4.4,1.1) {$b_4$};
%\node at (4.4,.5) {$b_5$};
%\node at (4.4,-.5) {$b_6$};
%\node at (4.4,-1.1) {$b_7$};

\begin{scope}[xshift = 7.5cm]
\draw (-1.7,0)--(-1,0);
\draw (-1.7,.5)--(-1,0);
\draw (-1.7,-.5)--(-1,0);
\draw (0,0)--(.7,0);
\draw (.7,.5)--(0,0);
\draw (.7,-.5)--(0,0);
\node at (-.5,.7) {\small $a$};
\node at (-.5,-.7) {\small $b$};
\filldraw[black] (-1,0) circle (0.1cm);
\draw (-1,0) circle (0.1cm);
\draw [rounded corners] (-1,0)--(-.75,.2)--(-.5,.3)--(-.25,.2)--(0,0); 
\draw [rounded corners] (-1,0)--(-.75,-.2)--(-.5,-.3)--(-.25,-.2)--(0,0); 
\filldraw[white] (0,0) circle (0.1cm);
\draw (0,0) circle (0.1cm);
\node at (2.24,0) {\large $\displaystyle = \binom {a+b} b$};
\draw (3.55,0)--(4.25,0);
\draw (5.25,0)--(5.95,0);
\draw (5.25,0)--(4.25,0);
\draw (3.55,.5)--(4.25,0);
\draw (3.55,-.5)--(4.25,0);
\draw (5.95,.5)--(5.25,0);
\draw (5.95,-.5)--(5.25,0);

\node at (4.75,.5) {\small $a+b$};
\filldraw[black] (4.25,0) circle (0.1cm);
\draw (4.25,0) circle (0.1cm);
\filldraw[white] (5.25,0) circle (0.1cm);
\draw (5.25,0) circle (0.1cm);
\end{scope}
\end{tikzpicture}
\end{equation}

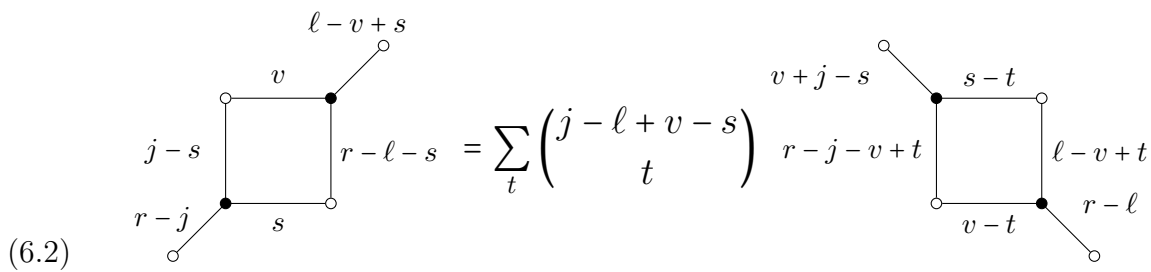
\begin{equation}\label{eq:squaremove}
\begin{tikzpicture}[scale=0.7]
%\node at (0,2.2) {$1$};
%\node at (2.2,0) {$2$};
%\node at (0,-2.2) {$3$};
%\node at (-2.2,0) {$4$};
%\draw (0,0) circle (2cm);
\draw (0,0) -- (2,0)--(2,2)--(0,2)--  (0,0);
\draw (3,3) -- (2,2);
\draw (-1,-1) -- (0,0);
\node at (-1,1) {\small $j-s$};
\node at (3.1,1) {\small $r-\ell-s$};
\node at (1,2.4) {\small $v$};
\node at (1,-0.4) {\small $s$};
\node at (-1.2,-.3) {\small $r-j$};
\node at (2.5,3.45) {\small $\ell-v+s$};

\filldraw[black] (0,0) circle (0.1cm);
\filldraw[black] (2,2) circle (0.1cm);
\filldraw[white] (2,0) circle (0.1cm);
\draw (2,0) circle (0.1cm);
\filldraw[white] (0,2) circle (0.1cm);
\draw (0,2) circle (0.1cm);
\filldraw[white] (3,3) circle (0.1cm);
\draw (3,3) circle (0.1cm);
\filldraw[white] (-1,-1) circle (0.1cm);
\draw (-1,-1) circle (0.1cm);

\node at (7.3,1) {\large $\displaystyle  =\sum_t \binom {j-\ell+v-s}{t}$};

\begin{scope}[xshift =13.5cm]
\draw (0,0) -- (0,2)-- (2,2)-- (2,0)-- (0,0);
\draw (3,-1) -- (2,0);
\draw (-1,3) -- (0,2);
\node at (-1.6,1) {\small $r-j-v+t$};
\node at (-2.2,2.4) {\small $v+j-s$};
\node at (3.1,1) {\small $\ell-v+t$};
\node at (3.3,0) {\small $r-\ell$};
\node at (1,2.4) {\small $s-t$};
\node at (1,-0.4) {\small $v-t$};

\filldraw[black] (2,0) circle (0.1cm);
\filldraw[black] (0,2) circle (0.1cm);
\filldraw[white] (0,0) circle (0.1cm);
\draw (0,0) circle (0.1cm);
\filldraw[white] (2,2) circle (0.1cm);
\draw (2,2) circle (0.1cm);
\filldraw[white] (-1,3) circle (0.1cm);
\draw (-1,3) circle (0.1cm);
\filldraw[white] (3,-1) circle (0.1cm);
\draw (3,-1) circle (0.1cm);
\end{scope}
\end{tikzpicture}
\end{equation}

\begin{thm}\label{thm:kernel} The relations amongst networks imposed by Theorem~\ref{thm:welldefined} generate the kernel $\mcf \mcw_\lam(r) \surjects \mcw_\lam(U)$.  
\end{thm}

\begin{cor}\label{cor:rwebker}
The four diagrammatic relations listed above generate the kernel $\mcf \mcw_{\lam}(r) \surjects \mcw_\lam(U)$.
\end{cor}

%We find the diagrammatic presentation of $\mcw_\lam(U)$ given by Corollary~\ref{cor:rwebker} elegant because the relations \eqref{eq:nonsquaremove} and \eqref{eq:squaremove} do not involve any signs. 
%the signs are ``hidden'' in the evaluation equation \eqref{eq:eora}. 

Before proving Theorem~\ref{thm:kernel}, we will show how Corollary \ref{cor:rwebker} follows from it.

\begin{proof}[Proof of Corollary \ref{cor:rwebker}] 
Let us briefly explain how the relations \eqref{eq:nonsquaremove} and \eqref{eq:squaremove} can be deduced from network moves. The two-valent vertex removal (M2) for networks implies two-valent vertex removal \eqref{eq:nonsquaremove} for $r$-weblike graphs. The parallel edge removal (R1) for networks implies the bigon removal move \eqref{eq:nonsquaremove} for $r$-weblike graphs. Leaf and dipole removal (R2), (R3) for networks give leaf and dipole removal for $r$-weblike graphs.

Thus we now focus on the most interesting move on networks (M1), and show that it implies the square move for $r$-weblike graphs \eqref{eq:squaremove}. 

Suppose $N$ and $N'$ are are related by a square move involving edge weights $a,b,c,d \in \bbc^*$ 
\begin{equation}\label{eq:localfragment}
\begin{tikzpicture}[scale=0.7]
%\node at (0,2.2) {$1$};
%\node at (2.2,0) {$2$};
%\node at (0,-2.2) {$3$};
%\node at (-2.2,0) {$4$};
%\draw (0,0) circle (2cm);
\draw (0,0) -- (2,0)--(2,2)--(0,2)--  (0,0);
\draw (3,3) -- (2,2);
\draw (-1,-1) -- (0,0);
\node at (-.4,1) {\small $a$};
\node at (2.4,1) {\small $c$};
\node at (1,2.4) {\small $b$};
\node at (1,-0.4) {\small $d$};

\filldraw[black] (0,0) circle (0.1cm);
\filldraw[black] (2,2) circle (0.1cm);
\filldraw[white] (2,0) circle (0.1cm);
\draw (2,0) circle (0.1cm);
\filldraw[white] (0,2) circle (0.1cm);
\draw (0,2) circle (0.1cm);
\filldraw[white] (3,3) circle (0.1cm);
\draw (3,3) circle (0.1cm);
\filldraw[white] (-1,-1) circle (0.1cm);
\draw (-1,-1) circle (0.1cm);
\end{tikzpicture}
%\node at (0,2.2) {$1$};
%\node at (2.2,0) {$2$};
%\node at (0,-2.2) {$3$};
%\node at (-2.2,0) {$4$};
%\draw (0,0) circle (2cm);
\hspace{30pt}
\begin{tikzpicture}[scale=0.7]
\draw (0,0) -- (0,2)-- (2,2)-- (2,0)-- (0,0);
\draw (3,-1) -- (2,0);
\draw (-1,3) -- (0,2);
\node at (-.4,1) {\small $c'$};
\node at (2.4,1) {\small $a'$};
\node at (1,2.4) {\small $d'$};
\node at (1,-0.4) {\small $b'$};

\filldraw[black] (2,0) circle (0.1cm);
\filldraw[black] (0,2) circle (0.1cm);
\filldraw[white] (0,0) circle (0.1cm);
\draw (0,0) circle (0.1cm);
\filldraw[white] (2,2) circle (0.1cm);
\draw (2,2) circle (0.1cm);
\filldraw[white] (-1,3) circle (0.1cm);
\draw (-1,3) circle (0.1cm);
\filldraw[white] (3,-1) circle (0.1cm);
\draw (3,-1) circle (0.1cm);
\end{tikzpicture}.
\end{equation}

We need to compare the $r$-weblike subgraphs coming from the networks $N$ and $N'$. Each such weblike subgraph can be built by fixing the edge multiplicites $m(e)$ for edges~$e$ outside the local fragment ($j$ from the southwest, $\ell$ from the southeast, and so on, as described above), then choosing the edge multiplicities  inside the local fragment in a way that is compatible with the outside. Thus the numbers $j,\ell$ and $s-v$ are fixed throughout this computation. We will let $s$ be variable (so that $s$ determines $v$). 
\begin{equation}\label{eq:Wjlrs}
\begin{tikzpicture}[scale=0.7]
%\node at (0,2.2) {$1$};
%\node at (2.2,0) {$2$};
%\node at (0,-2.2) {$3$};
%\node at (-2.2,0) {$4$};
%\draw (0,0) circle (2cm);
\draw (0,0) -- (2,0)--(2,2)--(0,2)--  (0,0);
\draw (3,3) -- (2,2);
\draw (-1,-1) -- (0,0);
\node at (-1,1) {\small $j-s$};
\node at (3.1,1) {\small $r-\ell-s$};
\node at (1,2.4) {\small $v$};
\node at (1,-0.4) {\small $s$};
\node at (-1.2,-.3) {\small $r-j$};
\node at (2.5,3.4) {\small $\ell-v+s$};

\filldraw[black] (0,0) circle (0.1cm);
\filldraw[black] (2,2) circle (0.1cm);
\filldraw[white] (2,0) circle (0.1cm);
\draw (2,0) circle (0.1cm);
\filldraw[white] (0,2) circle (0.1cm);
\draw (0,2) circle (0.1cm);
\filldraw[white] (3,3) circle (0.1cm);
\draw (3,3) circle (0.1cm);
\filldraw[white] (-1,-1) circle (0.1cm);
\draw (-1,-1) circle (0.1cm);
\end{tikzpicture}.
%\node at (0,2.2) {$1$};
%\node at (2.2,0) {$2$};
%\node at (0,-2.2) {$3$};
%\node at (-2.2,0) {$4$};
%\draw (0,0) circle (2cm);
\end{equation}
The contribution to $\wt(W)$ from the edges in the local fragment is $a^{j-s}b^v c^{r-\ell-s}d^s$.

On the other hand, in $N'$, for the same value of outside multiplicities, the filling inside will look like 
\begin{equation}\label{eq:Wjlrsu}
\begin{tikzpicture}[scale=0.7]
\draw (0,0) -- (0,2)-- (2,2)-- (2,0)-- (0,0);
\draw (3,-1) -- (2,0);
\draw (-1,3) -- (0,2);
\node at (-1.2,1) {\small $r-j-u$};
\node at (-2.4,2.4) {\small $v+j-s$};
\node at (2.8,1) {\small $\ell-u$};
\node at (3.3,0) {\small $r-\ell$};
\node at (1,2.4) {\small $u-v+s$};
\node at (1,-0.4) {\small $u$};

\node at (-.5,-.3) {};

\filldraw[black] (2,0) circle (0.1cm);
\filldraw[black] (0,2) circle (0.1cm);
\filldraw[white] (0,0) circle (0.1cm);
\draw (0,0) circle (0.1cm);
\filldraw[white] (2,2) circle (0.1cm);
\draw (2,2) circle (0.1cm);
\filldraw[white] (-1,3) circle (0.1cm);
\draw (-1,3) circle (0.1cm);
\filldraw[white] (3,-1) circle (0.1cm);
\draw (3,-1) circle (0.1cm);
\end{tikzpicture}
\end{equation}
where the values of $j,\ell,v-s$ agree with the fixed values, but $u$ can vary. Notice that $s$ and $v$ only occur in \eqref{eq:Wjlrsu} together via their difference $s-v$.

Now we denote by $W_s$ the weblike subgraph indexed by $s$ in \eqref{eq:Wjlrs} and denote by $W'_u$ the one indexed by $u$ in \eqref{eq:Wjlrsu}. %(in principle, there might be more than one $W_s$ if there are multiple ways to fill the outside edges, but each of these contribute different weights, so we can assume there is only one way for simplicity). 
We sum up the boldface $\textbf{W}_s$ with their weights and the $\textbf{W}_u$ with their weights. Since the points in the affine cone are related by 
$\tX(N) = (ac+bd)\tX(N')$, it follows that $\Web_r(N) = (ac+bd)^r\Web_r(N')$. Writing out what this means: 

%Second, there are generally many consistent ways 
%of choosing edge multiplicities outside the local fragment. However, each of these different choices contributes a different factor to $\wt(W)$, so we can group terms in \eqref{eq:WebNreln} according to their outside. We fix such a choice of edge multiplicities on the outside  and label by $\mcw_{j,\ell,v,s}$ the weblike subgraph that fills in this choice with values $j,\ell,v,s$ as in \eqref{eq:Wjlrs}. Furthermore, we can ignore the monomial coming from the edge weights outside the local fragment since this factor will show up on both sides of \eqref{eq:WebNreln}. Thirdly, we are going to ignore signs for the moment. 

\begin{align*}
\sum_s a^{j-s}b^v c^{r-\ell-s}d^s \, \textbf{W}_s   &= (a c +b d)^r \sum_{u} (a')^{\ell-u}(b')^u(c')^{r-j-u}(d')^{u-v+s} \, \textbf{W}'_u\\
&= (a c +b d)^{j-\ell+v-s} \sum_{u} a^{\ell-u}b^u c^{r-j-u}d^{u-v+s} \, \textbf{W}'_u. 
\end{align*}

Now we fix a value of $s$, and we also fix the exponent $t$ of $b d$ in the binomial formula expansion of $(a c +b d)^{v+j-\ell-s}$. In order to match the weights on the left and right hand side, the value of $u$ must be $v-t$. The corresponding equality is 
\begin{equation}\label{eq:squareswitchmove}
\textbf{W}_s   = \sum_t \binom {j-\ell+v-s}{t} \textbf{W}'_{v-t} \in \mcw_\lam(U), 
\end{equation}
which is the square move \eqref{eq:squaremove}. 

Thus we see that the local moves on networks imply the four diagrammatic moves for $r$-weblike graphs that preserve the corresponding tensor invariants. 
\end{proof}

\begin{proof}[Proof of Theorem \ref{thm:kernel}]
Now, we will give an ``abstract proof'' that diagrammatic moves amongst $r$-weblike subgraphs we obtain in this way are a complete set of relations. 

Let us denote by $K$ the kernel $\mcf \mcw_\lam(r) \to \mcw_\lam(U)$ and abbreviate $\mcf \mcw = \mcf \mcw_\lam(r)$. Let $\mcn(k)$ denote the set of networks of excedance $k$ whose underlying graph $G$ represents the top cell. We will denote by $\bbc[\mcn(k)]$ the linear space of functions $f \colon \mcn(k) \to \bbc$ with the property that for any fixed graph $G$, the map $f$ is a polynomial in the edge weights on $G$. The space $\bbc[\mcn(k)]$ is a very large space of functions.

We say that $f \in \bbc[\mcn(k)]$ is \emph{consistent of order $r$} if $f(N) = \aa^r f(N')$ whenever $\tX(N) = \aa \tX(N')$.  The space of consistent functions of order $r$ is canonically identified with the space
$$\bbc[\Gr(k,n)]_{(r)} := \bigoplus_\lam \bbc[\Gr(k,n)]_{(\lam)}$$
where the sum is over $\lam$ such that $\lam_1+\cdots+\lam_n=kr$.

There is a map $\Imm^\dagger : \mcf \mcw^* \to \bbc[\mcn]$ sending a functional $\varphi \in \mcf \mcw^*$ to the function on networks $N \mapsto \varphi(\Web_r(N;\lam))$, where we now think of $\Web_r(N;\lam)$ as an element of $\mcf \mcw$. It is defined similarly to the immanant map $\Imm$ but on the space of formal sums of weblike graphs. This map obviously fits into a commutative diagram 
\begin{equation}\label{eq:kerneldiagram}
\xymatrix{ (\mcf \mcw / K)^*  \ar@{^(->}[d] \ar[drr]^\cong & &  \\
\mcf \mcw^*  \ar[r]_{\Imm^\dagger}  & \bbc[\mcn(k)] & \ar@{_(->}[l]  \bbc[\Gr(k,n)]_{(\lam)}
}.
\end{equation}

The diagonal arrow in this diagram is the immanant map $\Imm: \mcw_\lam(U)^* \to \bbc[\Gr(k,n)]_\lam$. Going diagonally and then left, we arrive at the space of consistent functions of order $r$ on networks inside $\bbc[\mcn(k)]$ which have degree $\lam$. 

The image of the downward arrow is $K^\perp$.

The main claim which needs to be checked is that $\Imm^\dagger$ is an injection. Let us first see that the fact that $\Imm^\dagger$ is an injection allows us to finish the proof.  Suppose that $\Imm^\dagger$ is an injection.  Then it follows that $\varphi \in \mcf \mcw^*$ defines a consistent function on networks if and only if  $\varphi \in K^\perp$. To put it another way, whether or not $\varphi$ lies in $K^\perp$ is detected by whether $\Imm^\dagger(\varphi)$ is a consistent function on networks. A function in $\bbc[\mcn(k)]$ is consistent if an only if it transforms appropriately under local moves. Therefore, the relations defining $K^\perp$ are exactly those forced on $\varphi$ by the local moves on networks (cf.~Proposition~\ref{prop:localmoves}). In other words, we need that
$$\Imm^\dagger(\varphi)(N)=\alpha^r\Imm^\dagger(\varphi)(N') \qquad \text{whenever} \qquad \tX(N)=\alpha\tX(N').$$

We now prove the injectivity of $\Imm^\dagger$.  Namely, if a functional $\varphi \in \mcf \mcw^*$ satisfies $\varphi(\Web_r(N;\lam)) = 0$ for all networks $N $ of excedance $k$, then $\varphi$ must be $0$.  Equivalently, we must prove that the various $\{\Web_r(N;\lam)\}$ span $\mcf \mcw $ as we vary $N$ over all networks representing the top cell.  %Any web $W \in \mcf \mcw $ occurs as a tagging of an $r$-weblike subgraph of some large enough bipartite graph $G$.

Now let $G$ be a fixed bipartite graph that represents the top cell.  The graph $G$ has a finite number of $r$-weblike subgraphs. Let  $W_1,\dots,W_\ell$ be the list of all of the $r$-weblike subgraphs of $G$ having degree $\lam$. We claim that there exists a list $N_1,\dots,N_\ell$ of networks with underlying graph $G$ such that the matrix
\begin{equation}\label{eq:transitionmatrix}
(\wt_{N_i}(W_j))_{i,j=1,\dots,\ell}
\end{equation} 
is invertible. First suppose that such $N_1,\dots,N_\ell$ exist. The rows of the matrix give the terms of $\Web_r(N_i;\lam)$, so for each $j$ there exists a linear combination of the $\{\Web_r(N_i;\lam)\}$ giving the $r$-weblike subgraph $W_j$. By varying $G$, we then conclude that any $r$-weblike subgraph $W$ is in the span of $\{\Web_r(N;\lam)\}$, and $\{\Web_r(N;\lam)\}$ spans $\mcf \mcw$.

The existence of $N_1,\dots,N_\ell$ follows from a standard argument. The expression for $\Web_r(N; \lam)$ consists of a monomial in the edge weights of $G$ multiplied by each weblike subgraph $W_i$. Because the $W_i$ all differ, the monomial attached to each web has a different multidegree. Then a standard Vandermonde-type argument shows that one can specialize the edge-weights appropriately so that the matrix $(\wt_{N_i}(W_j))$ is invertible.
\end{proof}

%Thus, as will follow from Theorem~\ref{thm:kernel}, the kernel of $\mcf \mcs_{\lam,\text{subgraphs}}(r) \surjects \mcw_\lam(U)$ is generated by the square switch move, the bigon-removal move, and the valent $2$-vertex move for $r$-weblike subgraphs. We find this combinatorial presentation of $\mcw_\lam(U)$ (without tagging or signs) to be very clean. 

\subsection{Cautis-Kamnitzer-Morrison relations}\label{CKMrel}

Recall we denote by $\mcf \mcs_\lam(r)$ the space of formal sums of tagged $\SL_r$-web diagrams of degree $\lam$. Any tagged web gives a tensor invariant, so that we get a surjection $\mcf \mcs (r) \surjects \mcw_\lam(r)$. Cautis, Kamnitzer, and Morrison \cite{SkewHowe} gave a set of diagrammatic relations describing the kernel of this surjection. Let us briefly write down these relations (see \cite[Section 2]{SkewHowe} for pictures and discussion):%, noting that the ``dual version'' of a relation is the one in which all arrows are reversed: 
\begin{itemize}
\item Switching a tag \cite[(2.3)]{SkewHowe}. If $e$ is a tagged edge with multiplicities $a$ and $r-a$, then changing which side the tag is on contributes a factor $(-1)^{a(r-a)}$. 
\item Tag migration \cite[(2.7),(2.8)]{SkewHowe}. Suppose an interior white vertex $v$ has an incident edge $e$ with a tag pointing in the clockwise direction around $v$ (the tag can be either a pair or a source tag, and the vertex $v$ can be either a shuffle or wedge vertex). Let $e'$ be the next edge incident to $v$ in the clockwise direction from $v$. Then the tag on $e$ can be migrated to a tag on $e'$ pointing in the counterclockwise direction, without any change of sign. 
\item Wedge product is associative \cite[(2.6)]{SkewHowe} and shuffle is coassociative.
\item Bigon removal \cite[(2.4)]{SkewHowe}. A pair of directed edges $v \to u$, of multiplicities $a$ and $b$, can be replaced by a single edge $v \to u$ with multiplicity $a+b$, at the cost of a factor $(-1)^{ab}\binom {a+b}{b}$. 

%Let $e'$ be the next edge incident to $v$ in the clockwise direction from $v$. Then the pair tag on $e$ can be migrated to a pair tag on $e'$ (that points counterclockwise around $v$) without any change in sign. There are three other similar moves -- according to whether $v$ is a black or white vertex, and according to whether the incident tagged edge $e$ has a pair or source tag. We elide stating these other moves carefully.
%Let $v$ be an interior white vertex and suppose $e$ is the unique edge directed way from $v$. Suppose furthermore that $e$ is a tagged pair edge and the tag is directed clockwise from $e$ around $v$. Let $e'$ be the next edge in the clockwise direction from $e$. Then tag migration says that we can make $e$ into an ordinary edge directed away from $v$, and make $e'$ a tagged pair edge (whose tag is directed counterclockwise). There is a dual tag migration rule around an interior black vertex (in which ``clockwise'' becomes ``counterclockwise,''). 
%as well as ``away from'' and ``towards,'' everywhere swapped).
\item The square move for tagged webs: 
\begin{equation}\label{eq:squareswitchtaggedII}
\begin{tikzpicture}[scale=0.5]
%\node at (0,2.2) {$1$};
%\node at (2.2,0) {$2$};
%\node at (0,-2.2) {$3$};
%\node at (-2.2,0) {$4$};
%\draw (0,0) circle (2cm);
\node at (8.5,1) { \small $=\displaystyle \sum_t \binom {j-\ell+v-s}{t}$};
\draw [decoration={markings,mark=at position .6 with {\arrow[scale=1.7]{>}}},
    postaction={decorate},
    shorten >=0.4pt] (0,0) -- (2,0);
\draw [decoration={markings,mark=at position .6 with {\arrow[scale=1.7]{>}}},
    postaction={decorate},
    shorten >=0.4pt] (2,0)--(2,2);
\draw [decoration={markings,mark=at position .6 with {\arrow[scale=1.7]{>}}},
    postaction={decorate},
    shorten >=0.4pt] (0,0) -- (0,2);
\draw [decoration={markings,mark=at position .6 with {\arrow[scale=1.7]{>}}},
    postaction={decorate},
    shorten >=0.4pt] (2,2) -- (3,3);
\draw [decoration={markings,mark=at position .6 with {\arrow[scale=1.7]{>}}},
    postaction={decorate},
    shorten >=0.4pt] (0,2) -- (-1,3);
\draw [decoration={markings,mark=at position .6 with {\arrow[scale=1.7]{>}}},
    postaction={decorate},
    shorten >=0.4pt] (-1,-1) -- (0,0);
\draw [decoration={markings,mark=at position .6 with {\arrow[scale=1.7]{>}}},
    postaction={decorate},
    shorten >=0.4pt] (3,-1)--(2,0);
\draw [decoration={markings,mark=at position .6 with {\arrow[scale=1.7]{>}}},
    postaction={decorate},
    shorten >=0.4pt] (2,2)--(0,2);
\node at (-1,1) {\small $j-s$};
\node at (3.3,1) {\small $\ell+s$};
\node at (1,2.4) {\small $v$};
\node at (1,-0.4) {\small $s$};
\node at (-1.6,-1) {\small $j$};
\node at (3,3.6) {\small $\ell-v+s$};
\node at (-1,3.6) {\small $j-s+v$};
\node at (3.5,-1) {\small $\ell$};

\filldraw[black] (0,0) circle (0.1cm);
\filldraw[black] (2,2) circle (0.1cm);
\filldraw[white] (2,0) circle (0.1cm);
\draw (2,0) circle (0.1cm);
\filldraw[white] (0,2) circle (0.1cm);
\draw (0,2) circle (0.1cm);
%\filldraw[white] (-1,-1) circle (0.1cm);
%\draw (-1,-1) circle (0.1cm);

\begin{scope}[xshift = 15cm]
\draw [decoration={markings,mark=at position .6 with {\arrow[scale=1.7]{>}}},
    postaction={decorate},
    shorten >=0.4pt] (2,0)--(.2,0);
\draw [decoration={markings,mark=at position .6 with {\arrow[scale=1.7]{>}}},
    postaction={decorate},
    shorten >=0.4pt] (0,0)--(0,2);
\draw [decoration={markings,mark=at position .6 with {\arrow[scale=1.7]{>}}},
    postaction={decorate},
    shorten >=0.4pt] (0,2)--(2,2);
\draw [decoration={markings,mark=at position .6 with {\arrow[scale=1.7]{>}}},
    postaction={decorate},
    shorten >=0.4pt] (2,0)--(2,2);
\draw [decoration={markings,mark=at position .6 with {\arrow[scale=1.7]{>}}},
    postaction={decorate},
    shorten >=0.4pt](3,-1) -- (2,0);
\draw [decoration={markings,mark=at position .6 with {\arrow[scale=1.7]{>}}},
    postaction={decorate},
    shorten >=0.4pt](-1,-1) -- (0,0);
\draw [decoration={markings,mark=at position .6 with {\arrow[scale=1.7]{>}}},
    postaction={decorate},
    shorten >=0.4pt] (2,2) -- (3,3);
\draw [decoration={markings,mark=at position .6 with {\arrow[scale=1.7]{>}}},
    postaction={decorate},
    shorten >=0.4pt] (0,2) -- (-1,3);
\node at (-1.5,1) {\small $j+v-t$};
\node at (3,3.6) {\small $\ell-v+s$};
\node at (-1,3.6) {\small $j-s+v$};
\node at (3.5,1) {\small $\ell-v+t$};
\node at (3.5,-1) {\small $\ell$};
\node at (1,2.4) {\small $s-t$};
\node at (1,-0.4) {\small $v-t$};
\node at (-.5,-.3) {};
\node at (-1.6,-1) {\small $j$};
\filldraw[black] (2,0) circle (0.1cm);
\filldraw[black] (0,2) circle (0.1cm);
\filldraw[white] (0,0) circle (0.1cm);
\draw (0,0) circle (0.1cm);
\filldraw[white] (2,2) circle (0.1cm);
\draw (2,2) circle (0.1cm);
%\filldraw[white] (-1,3) circle (0.1cm);
%\draw (-1,3) circle (0.1cm);
%\filldraw[white] (3,-1) circle (0.1cm);
%\draw (3,-1) circle (0.1cm);
\end{scope}
\end{tikzpicture},
\end{equation}

\end{itemize}

\begin{thm}\label{thm:CKM}[Cautis-Kamnitzer-Morrison]
The five relations listed above generate the kernel $\mcf \mcs (\SL_r) \surjects \mcw_\lam(U)$.
\end{thm}

We note that the results in \cite{SkewHowe} are more general than ours -- they describe all diagrammatic relations amongst fundamental representations of the quantum group $U_q(\mathfrak{s}\mathfrak{l}_n)$, whereas our current proof only makes sense in the ``classical'' ($q=1$) setting. 
%We are hopeful that a similar proof could be formulated for generic $q$. 
Additionally, they provide certain redundant relations that are needed to describe the kernel if one wishes to work over $\bbz[q,q^{-1}]$ instead of $\bbc(q)$. In our version, we work over $\bbc$.

\medskip

Now we explain how Theorem \ref{thm:CKM} can be deduced from our Corollary \ref{cor:rwebker}.

There are three differences between $r$-weblike graphs and webs. First, and most important, $r$-weblike graphs do not come with a tagging. Recall that tagging an $r$-weblike graph $W$ gives a web $\hat{W} = \hat{W}(W,\mco)$, where $\mco$ refers to the data of the tagging. From any $r$-weblike graph $W$, we associate a canonical invariant $\textbf{W} := \sgn(W,\mco)\hat{W}$. One of the advantages of working with our tensor invariants $\textbf{W}$ is that the signs come naturally built-in.

The second difference between $r$-weblike graphs and webs is purely a matter of convention and is less serious: Cautis-Kamnitzer-Morrison require that vertices in their webs are of degree three, whereas ours can be arbitrary. For this reason, they need additional associativity relations that we do not. These relations follow from two-valent vertex removal on networks. 

The third difference is also less serious. We require our graphs to be bipartite, so that in order to go from a web to an $r$-weblike graph, we will sometimes need to contract edges. The fact that edge contractions give $r$-weblike graphs that are related by allowable moves again follows from two-valent vertex removal.

We now proceed with the proof. The first two relations in Theorem \ref{thm:CKM} have to do with tagging. These tagging relations generate a subspace of relations $\mcr \subset \mcf \mcs_\lam(r)$. 

There is a linear map $\mcf \mcs_\lam(r) \surjects \mcf \mcw_\lam(r)$ defined by replacing a tagged web $\hat{W} = \hat{W}(W,\mco)$ by its ``underlying $r$-weblike subgraph'' with a sign: $\hat{W} \mapsto \sgn(W,\mco)W \in \mcf \mcw_\lam(r)$. 

Let us be more specific about how to obtain an $r$-weblike graph from $\hat{W}$. Any edges in $\hat{W}$ joining vertices of the same color should be contracted, so that we end up with a bipartite graph. If necessary, we may need to use tag switches and migrations before contracting. Once this is done, if $e$ is an edge (or half-edge) with multiplicity $a$, and if $e$ points from a white vertex to a black vertex, then $e$ has multiplicity $r-a$ in $W$, while if it points away from a black vertex towards a white vertex, it has multiplicity $a$. Finally, forget about all the directions on the edges. The resulting labeled diagram is an $r$-weblike graph.

The quotient $\mcf \mcs_\lam(r) / \mcr$ is exactly the image of $\mcf \mcs_\lam(r)$ in $\mcf \mcw_\lam(r)$. We have that the map $\mcf \mcs_\lam(r) \surjects \mcw_\lam(r)$ factors as 
$$\mcf \mcs_\lam(r) \surjects  \mcf \mcw_\lam(r) \surjects \mcw_\lam(r)$$

There are relations on $\mcf \mcs_\lam(r)$ that come from the kernel of the first map. These are precisely the relations $\mcr \subset \mcf \mcs_\lam(r)$ coming from tagging. These relations follow from our analysis of tags in Lemma~\ref{lem:eora}.

The remaining relations on $\mcf \mcs_\lam(r)$ come from ``tagging'' (i.e., adding tags to) relations in $\mcf \mcw_\lam(r)$. Thus, to find a complete set of diagrammatic relations amongst tagged webs, we need to lift the relations defining $\ker(\mcf \mcw_\lam(r) \surjects \mcw_\lam(r))$ to $\mcf \mcs_\lam(r)$. 

Let us do this in the most interesting case of the square move. The square move for $r$-weblike graphs \eqref{eq:squaremove} can be lifted to a relation in $\mcf \mcs_\lam(r)$ by tagging as follows 
\begin{equation}\label{eq:squareswitchtagged}
\begin{tikzpicture}[scale=0.5]
%\node at (0,2.2) {$1$};
%\node at (2.2,0) {$2$};
%\node at (0,-2.2) {$3$};
%\node at (-2.2,0) {$4$};
%\draw (0,0) circle (2cm);
\node at (-4.3,1) {\small $\sgn(W,\mco)$};
\node at (10.5,1) { \small $=\displaystyle \sum_t \sgn(W'_t,\mco'_t) \, \binom {j-\ell+v-s}{t}$};
\draw [decoration={markings,mark=at position .6 with {\arrow[scale=1.7]{>}}},
    postaction={decorate},
    shorten >=0.4pt] (0,0) -- (2,0);
\draw [decoration={markings,mark=at position .6 with {\arrow[scale=1.7]{>}}},
    postaction={decorate},
    shorten >=0.4pt] (2,0)--(2,2);
\draw [decoration={markings,mark=at position .6 with {\arrow[scale=1.7]{>}}},
    postaction={decorate},
    shorten >=0.4pt] (0,0) -- (0,2);
\draw [decoration={markings,mark=at position .6 with {\arrow[scale=1.7]{>}}},
    postaction={decorate},
    shorten >=0.4pt] (2,2) -- (3,3);
\draw [decoration={markings,mark=at position .6 with {\arrow[scale=1.7]{>}}},
    postaction={decorate},
    shorten >=0.4pt] (0,2) -- (-1,3);
\draw [decoration={markings,mark=at position .6 with {\arrow[scale=1.7]{>}}},
    postaction={decorate},
    shorten >=0.4pt] (-1,-1) -- (0,0);
\draw [decoration={markings,mark=at position .6 with {\arrow[scale=1.7]{>}}},
    postaction={decorate},
    shorten >=0.4pt] (3,-1)--(2,0);
\draw [decoration={markings,mark=at position .6 with {\arrow[scale=1.7]{>}}},
    postaction={decorate},
    shorten >=0.4pt] (2,2)--(0,2);
\node at (-1,1) {\small $j-s$};
\node at (3.3,1) {\small $\ell+s$};
\node at (1,2.4) {\small $v$};
\node at (1,-0.4) {\small $s$};
\node at (-1.6,-1) {\small $j$};
\node at (3,3.6) {\small $\ell-v+s$};
\node at (-1,3.6) {\small $j-s+v$};
\node at (3.5,-1) {\small $\ell$};

\filldraw[black] (0,0) circle (0.1cm);
\filldraw[black] (2,2) circle (0.1cm);
\filldraw[white] (2,0) circle (0.1cm);
\draw (2,0) circle (0.1cm);
\filldraw[white] (0,2) circle (0.1cm);
\draw (0,2) circle (0.1cm);
%\filldraw[white] (-1,-1) circle (0.1cm);
%\draw (-1,-1) circle (0.1cm);

\begin{scope}[xshift = 19cm]
\draw [decoration={markings,mark=at position .6 with {\arrow[scale=1.7]{>}}},
    postaction={decorate},
    shorten >=0.4pt] (2,0)--(.2,0);
\draw [decoration={markings,mark=at position .6 with {\arrow[scale=1.7]{>}}},
    postaction={decorate},
    shorten >=0.4pt] (0,0)--(0,2);
\draw [decoration={markings,mark=at position .6 with {\arrow[scale=1.7]{>}}},
    postaction={decorate},
    shorten >=0.4pt] (0,2)--(2,2);
\draw [decoration={markings,mark=at position .6 with {\arrow[scale=1.7]{>}}},
    postaction={decorate},
    shorten >=0.4pt] (2,0)--(2,2);
\draw [decoration={markings,mark=at position .6 with {\arrow[scale=1.7]{>}}},
    postaction={decorate},
    shorten >=0.4pt] (3,-1) -- (2,0);
\draw [decoration={markings,mark=at position .6 with {\arrow[scale=1.7]{>}}},
    postaction={decorate},
    shorten >=0.4pt] (-1,-1) -- (0,0);
\draw [decoration={markings,mark=at position .6 with {\arrow[scale=1.7]{>}}},
    postaction={decorate},
    shorten >=0.4pt] (2,2) -- (3,3);
\draw [decoration={markings,mark=at position .6 with {\arrow[scale=1.7]{>}}},
    postaction={decorate},
    shorten >=0.4pt] (0,2) -- (-1,3);
\node at (-1.5,1) {\small $j+v-t$};
\node at (3,3.6) {\small $\ell-v+s$};
\node at (-1,3.6) {\small $j-s+v$};
\node at (3.5,1) {\small $\ell-v+t$};
\node at (3.5,-1) {\small $\ell$};
\node at (1,2.4) {\small $s-t$};
\node at (1,-0.4) {\small $v-t$};
\node at (-.5,-.3) {};
\node at (-1.6,-1) {\small $j$};
\filldraw[black] (2,0) circle (0.1cm);
\filldraw[black] (0,2) circle (0.1cm);
\filldraw[white] (0,0) circle (0.1cm);
\draw (0,0) circle (0.1cm);
\filldraw[white] (2,2) circle (0.1cm);
\draw (2,2) circle (0.1cm);
%\filldraw[white] (-1,3) circle (0.1cm);
%\draw (-1,3) circle (0.1cm);
%\filldraw[white] (3,-1) circle (0.1cm);
%\draw (3,-1) circle (0.1cm);
\end{scope}
\end{tikzpicture},
\end{equation}
where we have chosen a tagging outside the local fragment that is compatible with the tagging inside the fragment shown above. 
%where $\mco_s$ and $\mco'_t$ denote the fixed outside tagging of $W_s$ and $W'_t$ together with the indicated tagging inside the fragment. 
Then \eqref{eq:squareswitchtaggedII} follows once one checks that all of the signs, $\sgn(W,\mco)$ and $\sgn(W'_t,\mco'_t)$, are equal and do not depend on $t$. This can be proved by analyzing the canonical flows through the local fragments, as described in the proof of Lemma~\ref{lem:eora}. 

We briefly mention how to lift the remaining relations. Two-valent vertex removal for $r$-weblike graphs allows us to contract and expand edges to derive the associativity of wedge for tagged webs. In a similar fashion, bigon removal for $r$-weblike graphs lifts to the bigon removal for tagged webs. %(we note that in our conventions, the tagged version of this move has a multiplicative factor $(-1)^{ab}$ not present under the conventions \cite{SkewHowe}). 

%On the other hand, any two lifts of a given relation come from making different choices of tagging for the $r$-weblike subgraphs $W$ appearing in the relation. Thus, what is left is describing the diagrammatic relations imposed when a given tagging $\hat{W}(W,\mco)$ is replaced by another tagging $\hat{W}(W,\mco')$. But any two taggings differ by a  series of edge contractions (which are encoded by the associativity of wedge relation), tag switches, and tag migrations.  

%As is explained in \cite{SkewHowe}, the relations \cite[Equations 2.3,2.4,2.6,2.8,2.10]{SkewHowe} are a complete set of diagrammatic relations (provided one works over $\bbc(q)$, rather than $\bbz[q,q^{-1}]$). As we have just seen, these diagrammatic relations can be derived from the local moves on networks (and the moves imposed by changing $\mco \to \mco'$). 

%\remind{Edit this to explain briefly how we obtain Theorem \ref{thm:CKM} as well.}
%Our next Theorem uses our duality statements to give an ``abstract'' proof that the relations amongst webs encoded by Theorem~\ref{thm:welldefined} are a complete set of relations, as a consequence of duality. 

%
%\begin{thm}\label{thm:kernel} The relations amongst networks imposed by Theorem~\ref{thm:welldefined} generate the kernel $\mcf \mcs (\SL_r) \surjects \mcw(r,n)$.  
%\end{thm}
%Thus, Theorem \ref{thm:kernel} is an independent proof of the sufficiency of the $\SL_r$-skein relations. 
%On the other hand, our proof exhibits the full cyclic symmetry of networks in a disk.    

\section{Positroids and webs}\label{secn:positroids}
In this section we fix a positroid $\mcm$ and consider a graph $G$ such that $\mcm(G) = \mcm$ (see Section \ref{ssec:positroids}).  Let $\Pi_\mcm = \Pi_G$ be the corresponding positroid variety.  

\subsection{The immanant map for a positroid}
Let $\mfi(\mcm) \subset \bbc[\Gr(k,n)]$ be the (homogeneous) ideal generated by  $\{\Delta_I: I \notin M\}$.  By \cite{KLS}, $\mfi(\mcm)$ is
a prime ideal, and its vanishing set in $\Gr(k,n)$ is the positroid variety $\Pi_\mcm$.  The homogeneous coordinate ring $\Pi_\mcm$ is thus given as the quotient
\begin{equation}
\label{eq:quotient}
\bbc[\Pi_\mcm] =\bbc[\Gr(k,n)]/\mfi(\mcm),
\end{equation}
noting that by \cite{KLS}, $\Pi_\mcm \subseteq \Gr(k,n)$ is projectively normal.
The grading \eqref{eq:coordringgrading} descends to a grading on $\bbc[\Pi_\mcm]$. One of our goals will be to give a description of the graded pieces~$\bbc[\Pi_\mcm]_\lam$ via tensor invariants.

Recall that the immanant map (for the top cell) is an isomorphism
$$\Imm: \mcw_\lam(U)^* \rightarrow \bbc[\Gr(k,n)]_\lam.$$
%We know that there is a surjective quotient map
%$$\bbc[\Gr(k,n)]_\lam \surjects \bbc[\tPi_\mcm]_\lam$$
%and an inclusion 
%$$\mcw_\lam(U)(\mcm) \injects \mcw_\lam(U).$$
%We would like to show that the immanant map realizes a duality between these two maps, and that $\mcw_\lam(U)(\mcm)$ and $\bbc[\tPi_\mcm]_\lam$ are therefore dual.
We now define an immanant map $\Imm_\mcm: \mcw_\lam(U)^* \rightarrow \bbc[\Pi_\mcm]_\lam$ for the positroid $\mcm$ by the equation
$$\Imm_\mcm(\varphi)(\tX(N))=\varphi(\Web_r(N;\lam))$$
for $N$ any assignment of edge weights to $G$.

\begin{prop}\label{prop:annoying2} The function $\tX(N) \mapsto \varphi(\Web_r(N;\lam))$ extends to a (uniquely defined) element of $\bbc[\Pi_\mcm]_\lam$. 
\end{prop}

\begin{proof}[Proof of Proposition~\ref{prop:annoying2}] 
The proof is the same as the proof of Proposition~\ref{prop:annoying} after noting that the set of points of the form $\tX(N)$ is Zariski dense in $\tPi_\mcm$. We get that $\Imm_\mcm(\varphi)$ is a regular function on $\tPi_\mcm$, and that thus the immanant map $\Imm_\mcm$ is well-defined.
\end{proof}

Let us note that the immanant map is characterized as the map that (up to sign) sends $\eval(E_\mcs)$ to the function $\Delta_{I_1} \cdots \Delta_{I_r}$ where $\mcs$ and $\mci=(I_1,\dots I_r)$ are dual. Therefore $\Imm_\mcm =  \pi_\mcm \circ \Imm$, where $\pi_\mcm: \bbc[\Gr(k,n)]_\lam \surjects \bbc[\Pi_\mcm]_\lam$ denotes the quotient map.

\subsection{Partial evaluation of webs}
Let $\mfi(\mcm)_\lam$ denote the degree $\lam$ part of the positroid ideal $\mfi(\mcm)$.  We now give a description of the subspace 
$$
(\mfi(\mcm)_\lam)^\perp:= \{ x \in W_\lam(U) :  \la x, f \ra = 0 \text{ for all } f \in \mfi(\mcm)_\lam\} \subseteq W_\lam(U)
$$
where $\la \cdot, \cdot \ra$ denotes the pairing between $W_\lam(U)$ and $\bbc[\Gr(k,n)]_\lam$.  Thus the inclusion $(\mfi(\mcm)_\lam)^\perp \hookrightarrow W_\lam(U)$ is dual to the surjection $\pi_\mcm: \bbc[\Gr(k,n)]_\lam \surjects \bbc[\Pi_\mcm]_\lam$.

We now define the notion of the partial evaluation of a web.
Let $\hat{W} \in W_\lam(U)$ be an $\SL_r$-web, where $\lam_1+ \dots +\lam_n = kr$. The copies of $\bigwedge^{\lam_1}(U),\dots,\bigwedge^{\lam_n}(U)$ gives us $kr$ locations to plug in vectors $v_1,\dots,v_{kr}$. For example, the first $\lam_1$ vectors are plugged into boundary vertex $1$ as the wedge $v_1 \wedge \cdots \wedge v_{\lam_1}$, and so on. 

\begin{defn}\label{defn:partialeval}
Let $I = \{i_1,\dots,i_k\} \in \binom {[n]} k$. Let $U' \subset U$ be the subspace spanned by $E_1,\dots,E_{r-1}$. Then the \emph{partial evaluation map along $I$} 
\begin{align}
\mcw_\lam(U) &\to \mcw_{\mu}(U') \label{eq:partialevaluation}\\
x &\mapsto x \big|_{I \mapsto E_r} \label{eq:partialevaluationII}
\end{align}
is the linear map obtained by specializing the input vectors $v_{i_1},v_{i_2},\ldots,v_{i_k}$ to the last basis vector $E_r$, and then restricting the remaining vectors to lie in the subspace $U'$. The resulting function is $\SL_{r-1}$-invariant. We have $\mu_i = \lambda_i-1$ or $\mu_i = \lambda_i$ depending on whether $i \in I$ or $i \notin I$.  It satisfies $\mu_1+\dots+\mu_n = (r-1)k$. %,and we abusively denote it by $\mu  = \lam -I$. 
In the multilinear case $\lam = (1,\dots,1)$, partial evaluation along $I$ is a linear map $\mcw(r,n) \mapsto \mcw(r-1,n-k)$. 
\end{defn}

We can now make the key definition of this section. For a positroid $\mcm$, we denote by 
\begin{equation}\label{eq:dualspace}
\mcw_\lam(U)(\mcm) = \{x \in \mcw_\lam(U) \colon \, x \big|_{I \mapsto E_r} = 0 \text{ for all } I \notin \mcm \}.
\end{equation}
We get a collection of subspaces $\mcw_\lam(U)(\mcm)$ of the $\SL_r$-tensor invariant space indexed by positroids $\mcm$ for the Grassmannian $\Gr(k,n)$. 

If $W$ is an $r$-weblike subgraph, its partial evaluation along $I$ is zero if and only if $W$ has no consistent labelings in which there are $r$'s at the boundary edges indicated by $I$, cf.~Example~\ref{eg:37and5}. We note that the partial evaluation of a web $W$ will often decompose as a \emph{sum} of several $\SL_{r-1}$-web invariants. 

\begin{thm}\label{thm:positroidpairing}
Let $\mcm$ be a positroid with positroid variety $\Pi_\mcm$. Then we have $\mcw_\lam(U)(\mcm) = (\mfi(\mcm)_\lam)^\perp$.  In particular, these two spaces have the same dimension.
\end{thm}

\begin{proof}
$\mfi(\mcm)_\lam$ is spanned by the $r$-fold products of Pl\"ucker coordinates $\Delta_{I_1} \cdots \Delta_{I_r}$ satisfying \eqref{eq:IMultiset} such that at least one of the $I_i$ satisfies $I_i \notin \mcm$. By reordering indices, let us assume that $I_r \notin \mcm$.

Thus, the dual space $(\mfi(\mcm)_\lam)^\perp$ is the subspace of tensor invariants $x \in \mcw_\lam(U)$ that pair to zero with each of these $r$-fold products. From \eqref{eq:PluckerstoWebr}, pairing to $0$ with $(I_1,\dots,I_r)$ means that $x \big|_{E_\mcs} = 0$ where $\mcs$ is dual to $(I_1,\dots,I_r)$. If we fix $I_r$, this holds for all choices of $I_1,\dots,I_{r-1}$ if and only if the partial evaluation $x \big|_{I_r \mapsto E_r} = 0 \in \mcw_\mu(U')$. By varying $I_r$ over all subsets not in $\mcm$, we see that the dual space is exactly \eqref{eq:dualspace} as claimed.  
\end{proof}

\begin{thm}\label{thm:positroidspanning}
Let $G$ be a planar bipartite graph with positroid $\mcm$.  Then the subspace $\mcw_\lam(U)(\mcm)$ is spanned by either of the following sets: 
\begin{itemize}
\item the elements $\Web_r(N;\lam)$, as $N$ varies over the (infinitely many) networks whose underlying graph is $G$;
\item the elements ${\bf W}$, as $W$ varies over the (finitely many) $r$-weblike subgraphs of $G$ with degree $\lam$.  
\end{itemize}  
\end{thm}

This characterization of $\bbc[\Pi_\mcm]_{\lam}$ may be easier to work with than the characterization using promotion and cyclic Demazure crystals given in \cite[Section 12]{LamNotesII}.  

\begin{proof}
First we prove that the various $\Web_r(N;\lam)$ span the subspace $\mcw_\lam(U)(\mcm)$. From \eqref{eq:pairingdumbeddown}, we see immediately that $\Web_r(N;\lam)$ pairs to zero with any $r$-fold product as in the proof of Theorem~\ref{thm:positroidpairing} (since $\Delta_{I_r}(N) = 0$ on $\Pi_\mcm$).  Thus $\Web_r(N;\lam) \in \mcw_\lam(U)(\mcm)$. 

On the other hand, if $f \in \bbc[\Gr(k,n)]_\lam$ pairs to $0$ with every $\Web_r(N;\lam)$, then by \eqref{eq:pairingdumbeddownII}
we have $f(\tX(N)) = 0$ for every network $N$ representing $G$. Since the set of points of the form $\tX(N)$ is Zariski dense in the affine cone over $\Pi_\mcm$, we deduce that $f$ is the zero polynomial, i.e. $f \in \mfi(\mcm)_\lam$. This shows the various $\Web_r(N;\lam)$ span the dual space $(\mfi(\mcm)_\lam)^\perp$. 

Finally, the span of the various $\Web_r(N;\lam)$ coincides with the span of the weblike subgraphs of $G$, as follows from the from the argument inverting the transition matrix \eqref{eq:transitionmatrix} in the proof of Theorem~\ref{thm:kernel}.   
\end{proof}

Thus for a fixed $G$ associated to a positroid $\mcm$, the $r$-weblike subgraphs contained in $G$ are forced to evaluate to $0$ on certain basis tensors. This seems to be a completely new phenomenon. Also surprising is the fact that the elements ${\bf W}$, for $r$-weblike subgraphs of a bipartite graph $G$ representing the top cell, must span the invariant space $\mcw(k,n)$.

The statement that the weblike subgraphs of $G$ span the subspace \eqref{eq:dualspace} also implies some non-obvious compatibilities between webs, the partial evaluation map \eqref{eq:partialevaluation} and positroids. We indicate one such compatibility in the following examples.  

%For example, if $G$ is a planar bipartite graph with positroid $\mcm$, we have the following sequence of containments 
%\begin{align}
%(\bbc[\Pi]_\lam)^\perp &= \{x \in \mcw_\lam(U) \colon x \big|_{I \mapsto E_r} = 0, \, I \notin \mcm  \} \\
%&\superset \{\text{webs } $W$ \in \mcw_\lam(U) \colon W \big|_{I \mapsto E_r} = 0, \, I \notin \mcm  \} \\
%&\superset \{\text{$r$-weblike subgraphs $W \subset G$} \in \mcw_\lam(U)\},
%\end{align}
%but by Theorem \ref{thm:positroidpairing} this sequence of containments is a sequence of equalities. The first equality is essentially the definition of $(\bbc[\Pi]_\lam)^\perp$, but both the second and third equalities are nontrivial. The third equality says that $G$ has to have ``a lot '' of $r$-weblike subgraphs. Let us interpret the second inequality. Suppose we choose a web basis $W_1,\dots,W_s,W_{s+1},\dots,W_{s+t}$ such that $W_1,\dots,W_s$ are a basis for $\mcw_\lam(U)(\mcm)$, and such that $W_{s+1},\dots,W_{s+t} \notin \mcw_\lam(U)(\mcm)$. Then the second equality says that for any $x \in \span(W_{s+1},\dots,W_{s+t)}$, there exists an $I \notin \mcm$ such that the partial evaluation $x \big |_{I \mapsto E_r}$ is nonzero.    

\begin{example}
Let $I_{\text{last}} = \{n-k+1,\dots,n\}$ so that $\Delta_{I_{\text{last}}}$ is the Pl\"ucker coordinate occupying the last $k$ columns. We consider the positroid $\mcm = \{\binom {[n]}{k}\} \backslash I_{\text{last}}$. It is a positroid, because it can be written as an appropriate intersection of cyclically rotated Schubert matroids \cite{OhSchubert}. The ideal $\mfi(\Pi)_{\lam}$ is the set of multiples of $\Delta_{I_{\text{last}}}$ inside $\mfi(\Pi)_{\lam}$, i.e. $\mfi(\Pi)_{\lam} = \Delta_{I_{\text{last}}}\bbc[\Gr(k,n)]_{\mu}$, where $\mu$ is obtained from $\lam$ by decrementing along $I$. 
%where $\mu = \lam - I_{\text{last}}$. 

If $G$ is a bipartite graph with positroid $\mcm$, let $W_1,\dots,W_s$ be a set of $r$-weblike subgraphs of $G$ of degree $\lam$ such that $\textbf{W}_1,\dots,\textbf{W}_s$ is a basis for $\mcw_\lam(U)(\mcm)$. We can extend $\textbf{W}_1,\textbf{W}_2,\dots,\textbf{W}_s$ to a basis $\textbf{W}_1,\dots,\textbf{W}_s,\textbf{W}_{s+1},\dots,\textbf{W}_{s+t}$ of $\mcw_\lam(U)$ consisting entirely of weblike graphs. Then the number $t$ is given by 
\begin{equation}\label{eq:coincidenceofdims}
t = \dim(\mfi(\Pi)_{\lam}) = \dim(\bbc[\Gr(k,n)]_{\mu}) = \dim(W_\mu(U')),
\end{equation}
where the last equality follows from duality. On the other hand, since $\spann(\textbf{W}_{s+1},\dots,\textbf{W}_{s+t}) \cap \mcw_\lam(U)(\mcm) = \{0\}$, we deduce that the partial evaluations of $\textbf{W}_{s+1},\dots,\textbf{W}_{s+t}$ along $I_{\text{last}}$ are linearly independent in $W_{\mu}(U')$. Thus, these partial evaluations are a basis for $W_\mu(U')$. 
As we have already said, these partial evaluations are not guaranteed to be $\SL_{r-1}$-webs.
\end{example}

\begin{example}\label{eg:37and5} Let us the illustrate the previous example in a particular case. We let $k=r=3$ and $\lam = (1,\dots,1)$, so we are considering the space $\mcw(3,9)$ of $\SL_3$-invariant multilinear functions of 9 vectors $v_1,\dots,v_9$. Its dimension is the number of $3 \times 3$ standard Young tableaux, which is $42$.  The non-elliptic basis webs (up to rotation) are listed in the top row of Figure~\ref{eq:websuptorotation}. We will focus on the positroid $\mcm = \binom{[9]}{3} \backslash \{\Delta_{789}\}$ as in the preceding example. The subspace $\mcw(3,9)(\mcm)$ is the kernel of the linear map $\mcw(3,9) \to \mcw(2,6)$ given by specializing $v_7=v_8=v_9=E_3$.  Concretely, this means placing $3$'s at boundary vertices $7,8$, and $9$, which will force certain interior edges to also be labeled by $3$'s, and then studying consistent labelings of the leftover edges once these $3$'s have been placed.  

A web will be in the subspace~$\mcw(3,9)(\mcm)$ if and only if it either has a fork between vertices $7$ and $8$ or has a fork between vertices $8$ and $9$. By examining the forks in Figure~\ref{eq:websuptorotation}, we see that there are exactly 37 webs that are in $\mcw(3,9)(\mcm)$, and five webs that are not. We have placed $3$'s at the appropriate locations in four of the nonvanishing webs. The fifth nonvanishing web (not pictured) is obtained by reflecting the third web in Figure~\ref{eq:websuptorotation} along the vertical axis (with $3$'s at the same boundary vertices). From Theorem~\ref{thm:positroidspanning}, it follows that $\dim(\bbc[\Pi(\mcm)]_{(1,\dots,1)}) = 37$. 

To see that that $\dim(\bbc[\Pi(\mcm)]_{(1,\dots,1)}) = 37$ directly, let $W \notin \mcw(3,9)(\mcm)$ be one of the five non-vanishing webs. As we said above, the placement of $3$'s at the boundary forces certain interior edges to be $3$'s. The remaining edges in $W$ must be labeled by $1$'s and $2$'s. These remaining edges naturally decompose into a union of three noncrossing paths joining boundary vertices in pairs, as we demonstrate schematically via the downward arrows in Figure~\ref{eq:websuptorotation}. The five $\SL_2$-webs we get in this way are exactly the $5$ crossingless matchings on $6$ vertices, i.e. a set of basis webs for $\mcw(2,6)$. In confirmation of Theorem~\ref{thm:positroidpairing}, we can see concretely that no nontrivial linear combination of the five $\SL_3$-webs combines to give an element of $\mcw(3,9)(\mcm)$ (and thus, $\mcw(3,9)(\mcm)$ is spanned by the $37$ webs that vanish under partial evaluation). 
\end{example} 

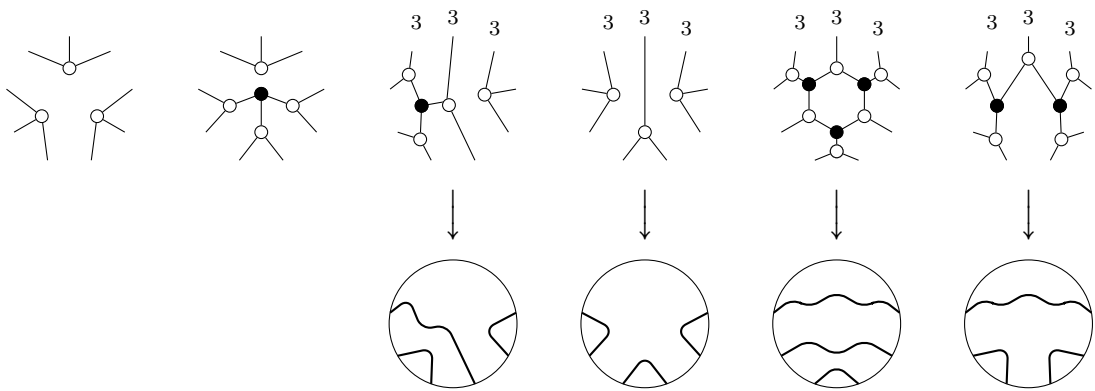
\begin{figure}
\begin{tikzpicture}[scale = .85]
\draw (50:1cm)--(90:.5cm);
\draw (90:1cm)--(90:.5cm);
\draw (130:1cm)--(90:.5cm);
\draw (170:1cm)--(210:.5cm);
\draw (210:1cm)--(210:.5cm);
\draw (250:1cm)--(210:.5cm);
\draw (290:1cm)--(330:.5cm);
\draw (330:1cm)--(330:.5cm);
\draw (370:1cm)--(330:.5cm);

\filldraw[white] (90:.5cm) circle (0.1cm);
\draw (90:.5cm) circle (0.1cm);
\filldraw[white] (210:.5cm) circle (0.1cm);
\draw (210:.5cm) circle (0.1cm);
\filldraw[white] (330:.5cm) circle (0.1cm);
\draw (330:.5cm) circle (0.1cm);
%%%%%%%%%%%%%%%%
\begin{scope}[xshift = 3cm]
\draw (50:1cm)--(90:.5cm);
\draw (90:1cm)--(90:.5cm);
\draw (130:1cm)--(90:.5cm);
\filldraw[white] (90:.5cm) circle (0.1cm);
\draw (90:.5cm) circle (0.1cm);

\draw (170:1cm)--(190:.5cm);
\draw (210:1cm)--(190:.5cm);
\draw (250:1cm)--(270:.5cm);
\draw (290:1cm)--(270:.5cm);
\draw (330:1cm)--(350:.5cm);
\draw (370:1cm)--(350:.5cm);
\draw (190:.5cm)--(90:.1cm);
\draw (270:.5cm)--(90:.1cm);
\draw (350:.5cm)--(90:.1cm);

\filldraw[white] (190:.5cm) circle (0.1cm);
\draw (190:.5cm) circle (0.1cm);
\filldraw[white] (270:.5cm) circle (0.1cm);
\draw (270:.5cm) circle (0.1cm);
\filldraw[white] (350:.5cm) circle (0.1cm);
\draw (350:.5cm) circle (0.1cm);
\filldraw[black] (90:.1cm) circle (0.1cm);
\draw (90:.1cm) circle (0.1cm);
\end{scope}
%%%%%%%%%%%%%%%%%%%%%%%%%%%
\begin{scope}[xshift = 6cm]
\node at (60:1.3cm) {\tiny 3};
\node at (90:1.3cm) {\tiny 3};
\node at (115:1.35cm) {\tiny 3};
\draw (330:1cm)--(10:.5cm);
\draw (10:1cm)--(10:.5cm);
\draw (50:1cm)--(10:.5cm);
\draw (130:1cm)--(150:.8cm);
\draw (170:1cm)--(150:.8cm);
\draw (210:1cm)--(230:.8cm);
\draw (250:1cm)--(230:.8cm);
\draw (290:1cm)--(190:.1cm);
\draw (90:1cm)--(190:.1cm);
\draw (150:.8cm)--(190:.5cm);
\draw (230:.8cm)--(190:.5cm);
\draw (190:.1cm)--(190:.5cm);
\filldraw[white] (150:.8cm) circle (0.1cm);
\draw (150:.8cm) circle (0.1cm);
\filldraw[white] (230:.8cm) circle (0.1cm);
\draw (230:.8cm) circle (0.1cm);
\filldraw[black] (190:.5cm) circle (0.1cm);
\draw (190:.5cm) circle (0.1cm);
\filldraw[white] (230:.1cm) circle (0.1cm);
\draw (230:.1cm) circle (0.1cm);
\filldraw[white] (10:.5cm) circle (0.1cm);
\draw (10:.5cm) circle (0.1cm);
\end{scope}
\begin{scope}[xshift = 9cm]
\node at (60:1.3cm) {\tiny 3};
\node at (90:1.3cm) {\tiny 3};
\node at (115:1.35cm) {\tiny 3};
\draw (50:1cm)--(10:.5cm);
\draw (10:1cm)--(10:.5cm);
\draw (330:1cm)--(10:.5cm);
\draw (130:1cm)--(170:.5cm);
\draw (170:1cm)--(170:.5cm);
\draw (210:1cm)--(170:.5cm);
\draw (90:1cm)--(270:.5cm);
\draw (250:1cm)--(270:.5cm);
\draw (290:1cm)--(270:.5cm);
\filldraw[white] (270:.5cm) circle (0.1cm);
\draw (270:.5cm) circle (0.1cm);
\filldraw[white] (170:.5cm) circle (0.1cm);
\draw (170:.5cm) circle (0.1cm);
\filldraw[white] (370:.5cm) circle (0.1cm);
\draw (370:.5cm) circle (0.1cm);
\end{scope}
\begin{scope}[xshift=12cm]
\node at (60:1.3cm) {\tiny 3};
\node at (90:1.3cm) {\tiny 3};
\node at (115:1.35cm) {\tiny 3};
\draw (130:1cm)--(150:.8cm)--(170:1cm);
\draw (250:1cm)--(270:.8cm)--(290:1cm);
\draw (10:1cm)--(30:.8cm)--(50:1cm);
\draw (30:.5cm)--(90:.5cm)--(150:.5cm)--(210:.5cm)--(270:.5cm)--(330:.5cm)--(30:.5cm);
\draw (90:.5cm)--(90:1cm);
\draw (330:.5cm)--(330:1cm);
\draw (210:.5cm)--(210:1cm);
\draw (150:.5cm)--(150:.8cm);
\draw (30:.5cm)--(30:.8cm);
\draw (270:.5cm)--(270:.8cm);
\filldraw[white] (90:.5cm) circle (0.1cm);
\draw (90:.5cm) circle (0.1cm);
\filldraw[white] (210:.5cm) circle (0.1cm);
\draw (210:.5cm) circle (0.1cm);
\filldraw[white] (330:.5cm) circle (0.1cm);
\draw (330:.5cm) circle (0.1cm);
\filldraw[black] (30:.5cm) circle (0.1cm);
\draw (30:.5cm) circle (0.1cm);
\filldraw[black] (150:.5cm) circle (0.1cm);
\draw (150:.5cm) circle (0.1cm);
\filldraw[black] (270:.5cm) circle (0.1cm);
\draw (270:.8cm) circle (0.1cm);
\filldraw[white] (30:.8cm) circle (0.1cm);
\draw (30:.8cm) circle (0.1cm);
\filldraw[white] (150:.8cm) circle (0.1cm);
\draw (150:.8cm) circle (0.1cm);
\filldraw[white] (270:.8cm) circle (0.1cm);
\draw (270:.8cm) circle (0.1cm);
\end{scope}
\begin{scope}[xshift=15cm]
\node at (60:1.3cm) {\tiny 3};
\node at (90:1.3cm) {\tiny 3};
\node at (115:1.35cm) {\tiny 3};
\draw (130:1cm)--(150:.8cm)--(170:1cm);
\draw (210:1cm)--(230:.8cm)--(250:1cm);
\draw (290:1cm)--(310:.8cm)--(330:1cm);
\draw (10:1cm)--(30:.8cm)--(50:1cm);
\draw (150:.8cm)--(190:.5cm);
\draw (230:.8cm)--(190:.5cm);
\draw (310:.8cm)--(-10:.5cm);
\draw (30:.8cm)--(-0:.5cm);
\draw (90:1cm)--(90:.65cm)--(190:.5cm);
\draw (90:.65cm)--(-10:.5cm);
\filldraw[white] (150:.8cm) circle (0.1cm);
\draw (150:.8cm) circle (0.1cm);
\filldraw[white] (230:.8cm) circle (0.1cm);
\draw (230:.8cm) circle (0.1cm);
\filldraw[white] (30:.8cm) circle (0.1cm);
\draw (30:.8cm) circle (0.1cm);
\filldraw[white] (310:.8cm) circle (0.1cm);
\draw (310:.8cm) circle (0.1cm);
\filldraw[white] (90:.65cm) circle (0.1cm);
\draw (90:.65cm) circle (0.1cm);
\filldraw[black] (190:.5cm) circle (0.1cm);
\draw (190:.5cm) circle (0.1cm);
\filldraw[black] (-10:.5cm) circle (0.1cm);
\draw (-10:.5cm) circle (0.1cm);
\end{scope}
%\end{tikzpicture}
\begin{scope}[yshift = -3.5cm]
\draw [white] (0,0) circle (1cm); 
\node at (6,1.7) {$\big \downarrow$};
\node at (9,1.7) { $\big \downarrow$};
\node at (12,1.7) { $\big \downarrow$};
\node at (15,1.7) { $\big \downarrow$};

%%%%%%%%%%%%%%%%
\begin{scope}[xshift = 3cm]
\draw [white] (0,0) circle (1cm); 
\end{scope}
%%%%%%%%%%%%%%%%%%%%%%%%%%%
\begin{scope}[xshift = 6cm]
\draw (0,0) circle (1cm); 
\draw [thick,rounded corners] (170:1cm)--(150:.8cm)--(190:.5cm)--(190:.1cm)--(290:1cm);
\draw [thick,rounded corners] (210:1cm)--(230:.5cm)--(250:1cm);
\draw [thick,rounded corners] (330:1cm)--(350:.5cm)--(370:1cm);
\end{scope}
\begin{scope}[xshift = 9cm]
\draw (0,0) circle (1cm); 
\draw [thick,rounded corners](10:1cm)--(-10:.5cm)--(330:1cm);
\draw [thick,rounded corners](170:1cm)--(190:.5cm)--(210:1cm);
\draw [thick,rounded corners](250:1cm)--(270:.5cm)--(290:1cm);
\end{scope}
\begin{scope}[xshift=12cm]
\draw (0,0) circle (1cm); 
\draw [thick,rounded corners] (170:1cm)--(150:.8cm)--(150:.5cm)--(90:.5cm)--(30:.5cm)--(30:.8cm)--(10:1cm);
\draw [thick,rounded corners] (210:1cm)--(210:.5cm)--(270:.5cm)--(330:.5cm)--(330:1cm);
\draw [thick,rounded corners](250:1cm)--(270:.65cm)--(290:1cm);
\end{scope}
\begin{scope}[xshift=15cm]
\draw (0,0) circle (1cm); 
\draw [thick,rounded corners] (170:1cm)--(150:.8cm)--(150:.5cm)--(90:.5cm)--(30:.5cm)--(30:.8cm)--(10:1cm);
\draw [thick,rounded corners] (290:1cm)--(310:.5cm)--(330:1cm);
\draw [thick,rounded corners](210:1cm)--(230:.5cm)--(250:1cm);
\end{scope}
\end{scope}
\end{tikzpicture}
\caption{$\SL_2$-webs from $\SL_3$-webs via partial evaluation. \label{eq:websuptorotation}}
\end{figure}

\section{Duality of symmetric group representations}\label{secn:duality}
In this section, we deal only with the multilinear web spaces $\mcw(r,n)$ and $\mcw(k,n)$ where $n=kr$. These spaces are $\bbc[\Gr(r,n)]_{\lam}$ and for $\bbc[\Gr(k,n)]_{\lam}$ for $\lam = (1,\dots,1)$, so we no longer use the symbol $\lam$ to denote the degree of a web. 
Instead $\lam$ will be used to denote a partition of $n$.

Both tensor invariant spaces $\mcw(k,n)$ and $\mcw(r,n)$ carry an action of the symmetric group by permutations of the vectors. Furthermore, these $S_n$-modules are irreducible, and are related to each other by tensoring with the sign representation~$\epsilon$. %$\mcw(k,n) \cong \mcw(r,n) \otimes \epsilon$.
\begin{thm}\label{thm:symmetricgroupduality}
The immanant map $\mcw(r,n)^* \to \mcw(k,n) \otimes \epsilon$ is an isomorphism of $S_n$-modules. 
\end{thm}

Since the $S_n$-modules in Theorem~\ref{thm:symmetricgroupduality} are irreducible, the map in Theorem~\ref{thm:symmetricgroupduality} 
is unique up to a scalar factor. The $S_n$-equivariance of the immanant map does not seem obvious, since there is not a natural action of $S_n$ on the space of networks. 

\medskip

Before proving Theorem \ref{thm:symmetricgroupduality} we recall the relevant notions from $S_n$ representation theory.  

Let $\lam = (\lam_1,\dots,\lam_r) \vdash n$ be a Young diagram, with $\lam_1$ the number of boxes in the first row, $\lam_2$ the number of boxes in the second row, and so on. Let $S_\lam = S_{\lam_1} \times \cdots \times S_{\lam_r} \subset S_n$ be the corresponding Young subgroup. A \emph{tableau} for a shape $\lam \vdash n$ is a filling of $\lam$ by the numbers $1,\dots,n$. A \emph{tabloid} is an equivalence class of tableaux, where we identify two tableaux if their entries differ by a permutation in each row. The free vector space $M_\lam$ on the set of tabloids of shape $\lam$ is a right $S_n$-module. It is the induced representation  $\Ind \mathds{1}_{S_\lam}^{S_n}$ where $\mathds{1}$ is the trivial representation. We use the notation $[T]$ for the tabloid determined by  a tableau $T$. 

We denote by $N_\lam = M_\lam \otimes \eps = \Ind \eps_{S_\lam}^{S_n}$. It has a basis consisting of \emph{anti-tabloids}, which are defined similarly to tabloids but with the added requirement that swapping two entries in a given row contributes a multiplicative factor of $-1$.  We use the notation $\{T\}$ for the antitabloid determined by $T$. 

Both of the $S_n$-modules $M_\lam$ and $N_\lam$ are reducible. To obtain $S_n$ irreducibles, we recall that the \emph{polytabloid} determined by $T$ is the signed sum of tabloids
\begin{equation}\label{eq:polytabloid}
\textnormal{poly}(T) = \sum_{\sigma \in \textnormal{col}(T)} \sgn(\sigma)[T \cdot \sigma] \in M_\lam,
\end{equation}
where $\textnormal{col}(T)$ is the subgroup of $S_n$ consisting of permutations that fix the entries in each column of $T$. %The element $\kappa_T = \sum_{\sigma \in \col(T)} \sgn(\sigma)$ is called the \emph{column symmetrizer} of $\lam$.  
Then the \emph{Specht Module} $S_\lam$ is the subspace of $M_\lam$ spanned 
by the polytabloids $T$, as $T$ varies over all fillings of $\lam$. The various Specht modules $S_\lam$, for $\lam \vdash n$, are exactly the irreducibles for $S_n$. 

For a Young diagram $\lam$, we let $\lam^t$ denote the conjugate or transpose partition (likewise $T^t$ denotes the conjugate tableau to $T$). Then the Specht modules $S_\lam$ and $S_{\lam^t}$ are related by $S_\lam = S_{\lam^t} \otimes \eps$.

\medskip

For the space $\mcw(k,n)$, a list of boundary location subsets $(I_1,\dots,I_r)$ is the same as an $r\times k$ tableau $T(I_1,\dots,I_r)$ whose rows are $I_1,\dots,I_r$. The boundary label subsets $\mcs = (S_1,\dots,S_n)$ are singletons such that $S_j$ records the row of $T(I_1,\dots,I_r)$ in which $j$ appears. The sign of $\mcs$, which we can denote by $\sgn(T)$, is determined by the well-known \emph{descent number} of $T$ modulo $2$, where we recall that the number of descents of a tableaux is the number of pairs $i>j$ but $i$ is in a lower row of $T$ than $j$.  

Letting $\lam$ denote the $r \times k$ rectangle, there is an isomorphism $S_n$-modules
\begin{align}
\mcw(k,n) & \leftrightarrow S_{\lam^t} \\
\Delta_{I_1}\Delta_{I_2} \cdots \Delta_{I_r} &\mapsto \textnormal{poly}(T^t) \label{eq:specialcase}  
\end{align}
where $T = T(I_1,\dots,I_r)$. Indeed, consider a $k \times n$ matrix of indeterminates $x_{ij}$.  The homogeneous coordinate ring $\bbc[\Gr(k,n)]$ can be identified with the invariant subring $\bbc[x_{ij}]^{\SL_k} \subset \bbc[x_{ij}]$. %(clearly any Pl\"ucker coordinate is a polynomial in the $x_{ij}$). 
For a set $J=\{j_1, \dots, j_r\} \in \binom {[n]}{r}$, denote by $x_{iJ}$ the monomial $x_{ij_1}x_{ij_2} \dots x_{ij_r}$. For a tableau $T = T(I_1,\dots,I_r)$ whose columns are $J_1,\dots,J_k$, we can consider the monomial $$\prod_{i=1}^k x_{iJ_i}.$$
Clearly this monomial only depends on $[T^{t}]$. The expansion of $\Delta_{I_1} \cdots \Delta_{I_r}$ as a signed sum of such monomials agrees with the expression of \eqref{eq:polytabloid} of poly$(T^t)$ as a signed sum of tabloids. Thus the subspace spanned by $\textnormal{poly}(T^t)$, which is the Specht module $S_{\lam^t}$, is isomorphic to $\mcw(k,n)$ as a representation of $S_n$.

%For a set $I=\{i_1, \dots, i_k\}$, denote by $x_{iI}$ the monomial $x_{ii_1}x_{ii_2} \dots x_{ii_k}$. For a tableaux $T = T(I_1,\dots,I_r)$ we can consider the monomial $$\prod_{i=1}^r x_{iI_i}.$$
%Clearly this monomial only depends on $[T^{\vee}]$. The span of such monomials forms a subspace inside $\bbc[x_{ij}]$ isomorphic to the representation $M_{\lambda^\vee}$ under the natural action of $S_n$. The subspace spanned by $\textnormal{poly}(T^\vee)$ is the Specht module $S_{\lam^\vee}$, and it is isomorphic to $\mcw(k,n)$ as a representation of $S_n$.
%Clearly this monomial only depends on $[T]$. The expansion of $\Delta_{I_1} \cdots \Delta_{I_r}$ as a signed sum of such monomials agrees with the expression of \eqref{eq:polytabloid} of poly$(T^\vee)$ as a signed sum of tabloids. 

\begin{proof}[Proof of Theorem~\ref{thm:symmetricgroupduality}]
The idea is to fit the immanant map into a commuting diagram of maps of $S_n$-modules, from which we deduce that the immanant map is $S_n$-equivariant. We remark that, all the maps in this diagram do not require $\lam$ to be of rectangular shape, and are unique up to scalars. The only assertion that requires $\lam$ to be of rectangular shape is the interpretation of one of these maps as the immanant map.

First, it easy to check that the pairing 
\begin{align*}%\label{bigduals}
M_\lam \otimes N_\lam &\to \eps \\
[S] \otimes \{T\} &\mapsto \delta_{S,T} \,  \sgn(S).
\end{align*}
is $S_n$-equivariant, where $\delta_{S,T}$ is the Kronecker delta. In this way, we get an $S_n$-equivariant map $N_\lam \to M_\lam^* \otimes \eps$. Since $S_\lam \subset M_\lam$, we can further compose 
$N_\lam \to M_\lam^* \otimes \eps \to S_\lam^* \otimes \eps$.

On the other hand, there is always an $S_n$-equivariant map $N_\lam \to S_{\lam^t}$ induced by $\{T\} \mapsto \textnormal{poly}(T^t)$. The map \eqref{eq:specialcase} can be thought of as a special case of this. (The space of formal sums of $\Delta_{I_1}\cdots \Delta_{I_r}$ -- before imposing Pl\"ucker relations -- is naturally identified with $N_\lam$).  Then the immanant map $(S_\lam)^* \otimes \epsilon \to S_{\lam^t}$ is the unique (up to scalars) map making the diagram 
\begin{equation}\label{eq:commutativediagram}
 \xymatrix{ N_\lam \ar[d] \ar[r] & (M_\lam)^* \otimes \eps \ar[r] & (S_\lam)^* \otimes \eps  \ar[dll]^{\Imm} \\
S_{\lam^t} &   & 
}
\end{equation}
Let us explain why this diagonal map is the immanant map. Going down in \eqref{eq:commutativediagram} is the map that replaces a formal sum of products $\Delta_{I_1} \cdots \Delta_{I_r}$ with their images in $\bbc[\Gr(k,n)]$ (i.e., going down is imposing the Pl\"ucker relations).

Let us now examine the top row of the commutative diagram. The space $M_\lam$ can be identified with monomials in $y_{ij}$ which are entries in an $r \times n$ matrix, and which are multilinear in each column, by the analysis above. The map $N_\lam \to M_\lam^* \otimes \eps$ takes an antitabloid $\{T\}$ to $\sign(T)$ times the functional that picks out the coefficient of the tabloid $[T]$. Putting this together, the antitabloid $\{T\}$ gets sent to $\sign(T)$ times the coefficient of $\prod_{i} y_{iI_i}.$ When we restrict this to the Specht module, we get exactly the functional $\sgn(\mcs)\text{eval}(E_\mcs)$, where  eval$(E_\mcs)$ is as in the proof of Proposition~\ref{prop:annoying} . The commutativity now follows since the immanant map is characterized by~\eqref{eq:PluckerstoWebr}.  
\end{proof}

\section*{Appendix: web duality pictures}
This appendix is jointly written with Darlayne Addabbo, Eric Bucher, Sam Clearman, Laura Escobar, Ningning Ma, Suho Oh, and Hannah Vogel. For simplicity, we restrict attention to the multilinear case $\lam = (1,\dots,1)$ and $n = kr$. The immanant map provides us with an isomorphism $\Imm \colon \mcw(r,n)^* \to \mcw(k,n)$. When $r$ is equal to $2$ or $3$, the space $\mcw(r,n)$ has a distinguished choice of \emph{web basis} $\mcb$ (given by crossingless matchings, and non-elliptic $\SL_3$ webs, respetively). For a given basis web $W$, we denote by $\varphi_W \in \mcw(r,n)^*$ the dual basis element with respect to $\mcb$. Its image $\Imm(\varphi_W) \in \mcw(k,n)$ is called a \emph{web immanant} in \cite{LamDimers}. 

Our main observation in this appendix is the following: 
\begin{observation}
Let $r=2$ or $3$, and $W \in \mcw(r,n)$ be a basis web. Then the web immanants $\Imm(\varphi_W) \in \mcw(k,n)$ are also web invariants (up to a sign). 
\end{observation}

That is, the unique $S_n$-equivariant pairing between the tensor invariant spaces $\mcw(k,n)$ and $\mcw(r,n)$ induces a duality between basis webs in these spaces, in small cases. We verified this observation in small cases by straightforward calculations which we omit here (see \cite{LamDimers} for some examples). The resulting pairing is rotation-equivariant, and we find the resulting pictures appealing, cf. Figure~\ref{fig:dualitypics}. 

Our second observation is as follows. When $r=2$ or $3$, there is a well-known bijection between standard young tableaux and non-elliptic webs, due to Khovanov and Kuperberg \cite{KhovanovKuperberg} (see also \cite{Tymoczko}). It has the property that rotation of webs is given by promotion of standard young tableaux, and this can be used to give an elementary proof \cite{A2Promotion} of the cyclic sieving phenomenon for rectangular tableaux with $2$ or $3$ or rows. Our next observation is that under this bijection, dual basis webs correspond to transposed tableaux: 

\begin{observation}
In small cases, duality between basis webs is given by transposing standard young tableaux. 
\end{observation}

Specifically, we checked that this is true for the dualities $\mcw(3,6) \leftrightarrow \mcw(2,6)$ and $\mcw(3,9) \leftrightarrow \mcw(3,9)$. For $r \geq 4$, Westbury has given a basis of $\SL_r$-webs indexed by standard young tableaux \cite{Westbury}. In the special case of $\mcw(4,8)$, the basis of web immanants listed in the second column of 
Figure~\ref{fig:dualitypics} is \emph{different} than Westbury's basis for $\mcw(4,8)$. The web immanant basis consists of the $4$ rotation classes of the first type of web, the $8$ rotation classes of the second type of web, and the $2$ rotation classes of the third type of web. We remark that this third type of $\SL_4$-web is fixed by rotating two units, which is not obvious, but is an instance of the $\SL_4$-square move. The web immanant basis disagrees with Westbury's basis in these last two elements -- the element Westbury assigns to these two tableaux has order $4$ under rotation. Once the elements of Westbury's basis are replaced by the corresponding web immanants, we again have that rotation of webs is given by promotion. This suggests that there might be a slightly different choice of web basis, indexed by tableaux, that is better behaved with respect to promotion. 

Let us also note that the last $\SL_5$-web in Figure~\ref{fig:dualitypics} \emph{must} be fixed by rotating two units, because this is true of its dual $\SL_2$-web. And indeed, this can be checked using the $\SL_5$ diagrammatic relations -- applying a square move to the top square of this web produces the same web, but rotated 4 units clockwise (so applying the square move three times, we get rotation by 2 units).

We mention that Observation 8.3 is probably implied by similar statements guaranteeing duality between the \emph{canonical bases} for $\mcw(r,n)$ and $\mcw(k,n)$ \cite{Lusztig}, and from an agreement between the web basis and canonical basis in these instances. Nonetheless, the pictures in Figure~\ref{fig:dualitypics} are new. When $r$ and $k$ are both $>3$, the immanant map gives us a bilinear pairing of $\mcw(k,n)$ with $\mcw(r,n)$. It would be interesting to understand the resulting pairing between $\SL_k$-webs and $\SL_r$-webs combinatorially. 
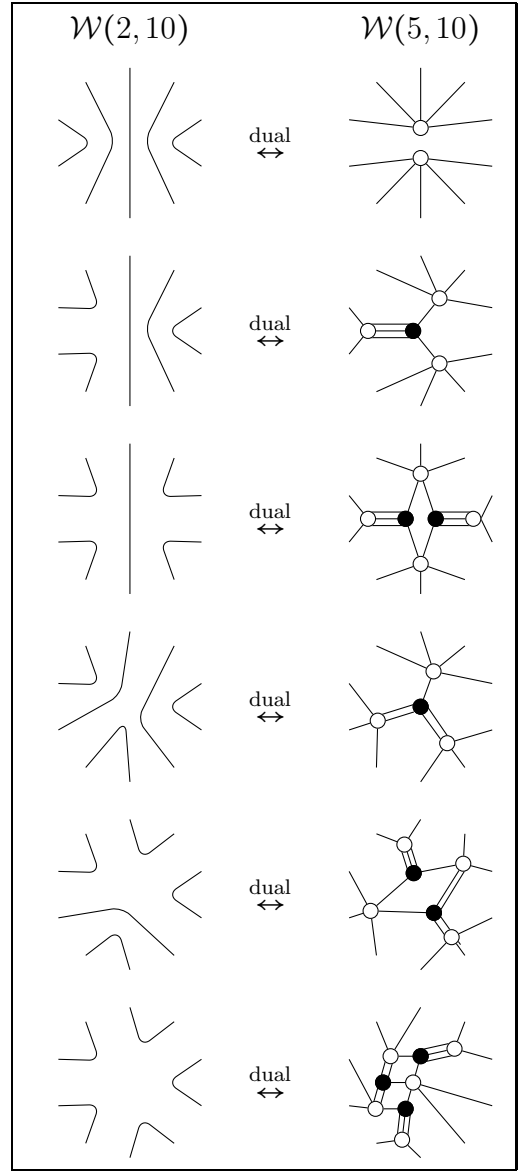
\begin{center}
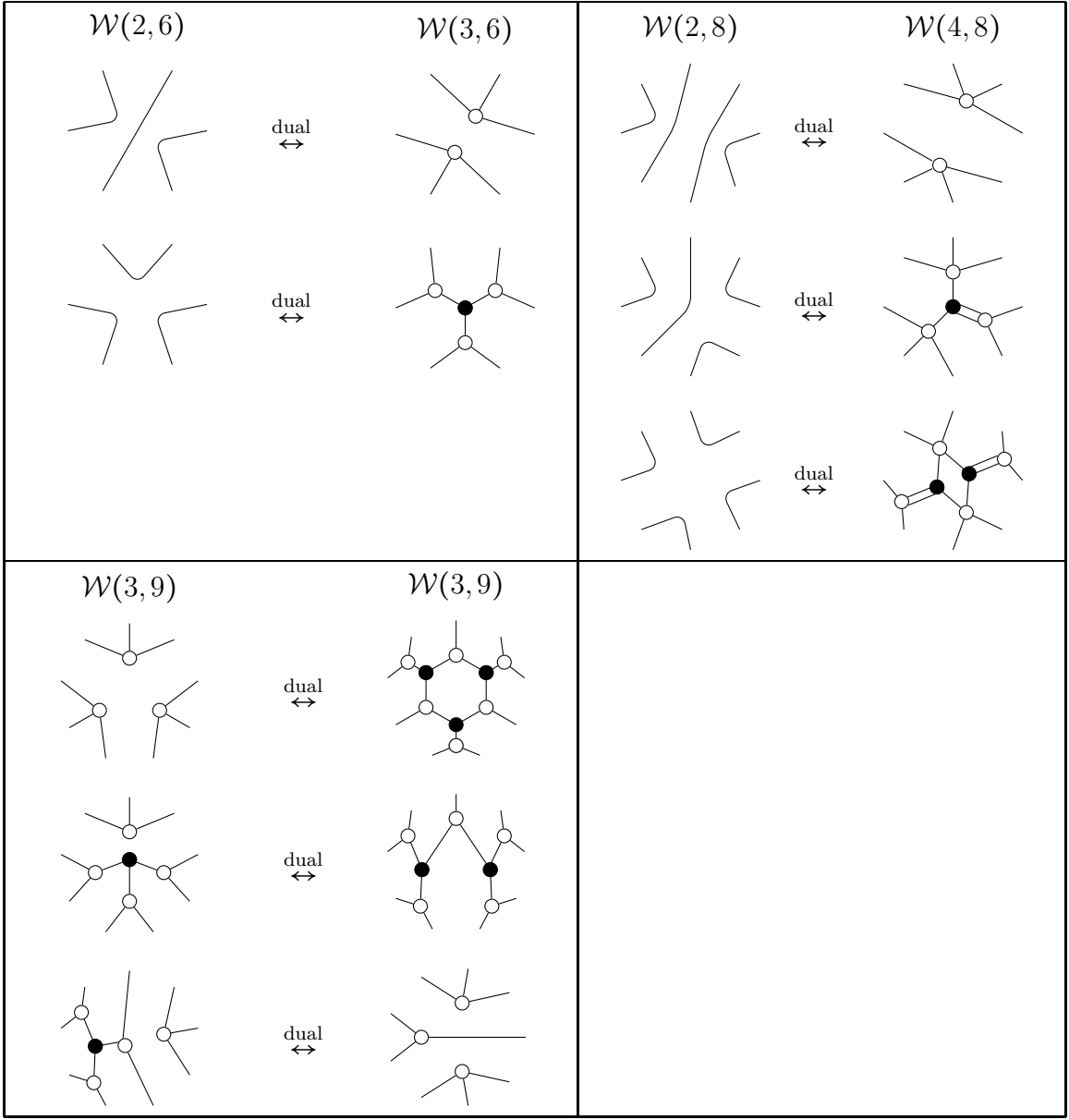
\begin{figure}[h]
\begin{tabular}{|cc|cc|}
\hline
\hspace{-1.2cm}
\begin{tikzpicture}[scale =1]
\node at (0,1.5) {$\mcw(2,6)$};
\draw (60:1cm)--(240:1cm);
\draw [rounded corners] (0:1cm)--(330:.3cm)--(300:1cm);
\draw [rounded corners] (120:1cm)--(150:.3cm)--(180:1cm);
\begin{scope}[yshift = -2.5cm]
\draw [rounded corners] (60:1cm)--(90:.3cm)--(120:1cm);
\draw [rounded corners] (180:1cm)--(210:.3cm)--(240:1cm);
\draw [rounded corners] (300:1cm)--(330:.3cm)--(360:1cm);
\node at (0,-3.4) {\hspace{.01cm}};
\end{scope}
\end{tikzpicture}
& 
\hspace{-1.5cm}
\begin{tikzpicture}[scale = 1]
\node at (0,1.5) {$\mcw(3,6)$};
\node at (-2.5,0) {$\overset{\text{dual}}{\leftrightarrow}$};
\draw (0:1cm)--(60:.3cm)--(60:1cm);
\draw (120:1cm)--(60:.3cm);
\draw (180:1cm)--(240:.3cm)--(300:1cm);
\draw (240:1cm)--(240:.3cm);
\filldraw[white] (240:.3cm) circle (0.1cm);
\draw (240:.3cm) circle (0.1cm);
\filldraw[white] (60:.3cm) circle (0.1cm);
\draw (60:.3cm) circle (0.1cm);
\node at (0,-5.85) {\hspace{.01cm}};
%%%%%%%%%%%%%%%%
\begin{scope}[yshift = -2.5cm]
\draw (0:1cm)--(30:.5cm)--(60:1cm);
\draw (120:1cm)--(150:.5cm)--(180:1cm);
\draw (240:1cm)--(270:.5cm)--(300:1cm);
\draw (30:.5cm)--(0,0)--(150:.5cm);
\draw (270:.5cm)--(0,0);
\node at (-2.5,0) {$\overset{\text{dual}}{\leftrightarrow}$};
\filldraw[white] (30:.5cm) circle (0.1cm);
\draw (30:.5cm) circle (0.1cm);
\filldraw[white] (150:.5cm) circle (0.1cm);
\draw (150:.5cm) circle (0.1cm);
\filldraw[white] (270:.5cm) circle (0.1cm);
\draw (270:.5cm) circle (0.1cm);
\filldraw[black] (0,0) circle (0.1cm);
\draw (0,0) circle (0.1cm);
\end{scope}
\end{tikzpicture}
\hspace{.3cm}
&
\hspace{.3cm}
\begin{tikzpicture}[scale = 1]
\node at (0,1.5) {$\mcw(2,8)$};
\draw [rounded corners] (90:1cm)--(157.5:.25cm)--(225:1cm);
\draw [rounded corners] (45:1cm)--(337.5:.25cm)--(270:1cm);
\draw [rounded corners] (0:1cm)--(337.5:.5cm)--(310:1cm);
\draw [rounded corners] (135:1cm)--(157.5:.5cm)--(180:1cm);
\begin{scope}[yshift = -2.5cm]
\draw [rounded corners] (90:1cm)--(0,0)--(225:1cm);
\draw [rounded corners] (135:1cm)--(157.5:.5cm)--(180:1cm);
\draw [rounded corners] (0:1cm)--(22.5:.5cm)--(45:1cm);
\draw [rounded corners] (315:1cm)--(292.5:.5cm)--(270:1cm);
\end{scope}
\begin{scope}[yshift = -5cm]
\draw [rounded corners] (90:1cm)--(67.5:.5cm)--(45:1cm);
\draw [rounded corners] (180:1cm)--(157.5:.5cm)--(135:1cm);
\draw [rounded corners] (270:1cm)--(257.5:.5cm)--(225:1cm);
\draw [rounded corners] (360:1cm)--(337.5:.5cm)--(315:1cm);
\end{scope}
\end{tikzpicture}
&
\begin{tikzpicture}[scale = 1]
\node at (0,1.5) {$\mcw(4,8)$};
\node at (-2,0) {$\overset{\text{dual}}{\leftrightarrow}$};
\draw (0:1cm)--(67.5:.5cm)--(45:1cm);
\draw (90:1cm)--(67.5:.5cm)--(135:1cm);
\draw (180:1cm)--(247.5:.5cm)--(225:1cm);
\draw (270:1cm)--(247.5:.5cm)--(315:1cm);
\filldraw[white] (245.5:.5cm) circle (0.1cm);
\draw (247.5:.5cm) circle (0.1cm);
\filldraw[white] (67.5:.5cm) circle (0.1cm);
\draw (67.5:.5cm) circle (0.1cm);
\begin{scope}[yshift = -2.5cm]
\node at (-2,0) {$\overset{\text{dual}}{\leftrightarrow}$};
\draw (45:1cm)--(90:.5cm)--(135:1cm);
\draw (180:1cm)--(225:.5cm)--(270:1cm);
\draw (90:1cm)--(90:.5cm);
\draw (225:1cm)--(225:.5cm);
\draw (315:1cm)--(337.5:.5cm)--(0:1cm);
\draw (0,0)--(90:.5cm);
\draw (0,0)--(225:.5cm);
\draw (.05cm,.05cm)--(.5119cm,-.1413cm);
\draw (-.05cm,-.05cm)--(.4119cm,-.2413cm);
\filldraw[white] (90:.5cm) circle (0.1cm);
\draw (90:.5cm) circle (0.1cm);
\filldraw[white] (225:.5cm) circle (0.1cm);
\draw (225:.5cm) circle (0.1cm);
\filldraw[black] (0,0) circle (0.1cm);
\draw (0,0) circle (0.1cm);
\filldraw[white] (337.5:.5cm) circle (0.1cm);
\draw (337.5:.5cm) circle (0.1cm);
\end{scope}
\begin{scope}[yshift = -5cm]
\node at (-2,0) {$\overset{\text{dual}}{\leftrightarrow}$};
\draw (22.5:.25cm)--(112.5:.5cm)--(202.5:.25cm)--(292.5:.5cm)--(22.5:.25cm);
\draw (45:1cm)--(22.5:.8cm)--(0:1cm);
\draw (225:1cm)--(202.5:.8cm)--(180:1cm);
\draw (.6399cm,.356cm)--(.13cm,.145cm);
\draw (.689cm,.256cm)--(.18cm,.045cm);
\draw (-.6399cm,-.356cm)--(-.13cm,-.145cm);
\draw (-.689cm,-.256cm)--(-.18cm,-.045cm);
\draw (90:1cm)--(112.5:.5cm)--(135:1cm);
\draw (270:1cm)--(292.5:.5cm)--(315:1cm);
\filldraw[black] (22.5:.25cm) circle (0.1cm);
\draw (22.5:.25cm) circle (0.1cm);
\filldraw[black] (202.5:.25cm) circle (0.1cm);
\draw (202.5:.25cm) circle (0.1cm);
\filldraw[white] (112.5:.5cm) circle (0.1cm);
\draw (112.5:.5cm) circle (0.1cm);
\filldraw[white] (292.5:.5cm) circle (0.1cm);
\draw (292.5:.5cm) circle (0.1cm);
\filldraw[white] (22.5:.8cm) circle (0.1cm);
\draw (22.5:.8cm) circle (0.1cm);
\filldraw[white] (202.5:.8cm) circle (0.1cm);
\draw (202.5:.8cm) circle (0.1cm);
\end{scope}
\end{tikzpicture}
\hspace{.3cm}
\\
\hline 
\hspace{.5cm}
\begin{tikzpicture}[scale = 1]
\node at (0,1.5) {$\mcw(3,9)$};
%\node at (2,2) {$\overset{\text{dual}}{\leftrightarrow}$};
\node at (2.5,0) {$\overset{\text{dual}}{\leftrightarrow}$};
\draw (50:1cm)--(90:.5cm);
\draw (90:1cm)--(90:.5cm);
\draw (130:1cm)--(90:.5cm);
\draw (170:1cm)--(210:.5cm);
\draw (210:1cm)--(210:.5cm);
\draw (250:1cm)--(210:.5cm);
\draw (290:1cm)--(330:.5cm);
\draw (330:1cm)--(330:.5cm);
\draw (370:1cm)--(330:.5cm);

\filldraw[white] (90:.5cm) circle (0.1cm);
\draw (90:.5cm) circle (0.1cm);
\filldraw[white] (210:.5cm) circle (0.1cm);
\draw (210:.5cm) circle (0.1cm);
\filldraw[white] (330:.5cm) circle (0.1cm);
\draw (330:.5cm) circle (0.1cm);
%%%%%%%%%%%%%%%%
\begin{scope}[yshift = -2.5cm]
\node at (2.5,0) {$\overset{\text{dual}}{\leftrightarrow}$};
\draw (50:1cm)--(90:.5cm);
\draw (90:1cm)--(90:.5cm);
\draw (130:1cm)--(90:.5cm);
\filldraw[white] (90:.5cm) circle (0.1cm);
\draw (90:.5cm) circle (0.1cm);

\draw (170:1cm)--(190:.5cm);
\draw (210:1cm)--(190:.5cm);
\draw (250:1cm)--(270:.5cm);
\draw (290:1cm)--(270:.5cm);
\draw (330:1cm)--(350:.5cm);
\draw (370:1cm)--(350:.5cm);
\draw (190:.5cm)--(90:.1cm);
\draw (270:.5cm)--(90:.1cm);
\draw (350:.5cm)--(90:.1cm);

\filldraw[white] (190:.5cm) circle (0.1cm);
\draw (190:.5cm) circle (0.1cm);
\filldraw[white] (270:.5cm) circle (0.1cm);
\draw (270:.5cm) circle (0.1cm);
\filldraw[white] (350:.5cm) circle (0.1cm);
\draw (350:.5cm) circle (0.1cm);
\filldraw[black] (90:.1cm) circle (0.1cm);
\draw (90:.1cm) circle (0.1cm);
\end{scope}
%%%%%%%%%%%%%%%%%%%%%%%%%%%
\begin{scope}[yshift = -5cm]
\node at (2.5,0) {$\overset{\text{dual}}{\leftrightarrow}$};
\draw (330:1cm)--(10:.5cm);
\draw (10:1cm)--(10:.5cm);
\draw (50:1cm)--(10:.5cm);
\draw (130:1cm)--(150:.8cm);
\draw (170:1cm)--(150:.8cm);
\draw (210:1cm)--(230:.8cm);
\draw (250:1cm)--(230:.8cm);
\draw (290:1cm)--(190:.1cm);
\draw (90:1cm)--(190:.1cm);
\draw (150:.8cm)--(190:.5cm);
\draw (230:.8cm)--(190:.5cm);
\draw (190:.1cm)--(190:.5cm);
\filldraw[white] (150:.8cm) circle (0.1cm);
\draw (150:.8cm) circle (0.1cm);
\filldraw[white] (230:.8cm) circle (0.1cm);
\draw (230:.8cm) circle (0.1cm);
\filldraw[black] (190:.5cm) circle (0.1cm);
\draw (190:.5cm) circle (0.1cm);
\filldraw[white] (230:.1cm) circle (0.1cm);
\draw (230:.1cm) circle (0.1cm);
\filldraw[white] (10:.5cm) circle (0.1cm);
\draw (10:.5cm) circle (0.1cm);
\end{scope}
\end{tikzpicture}
&
\begin{tikzpicture}[scale = 1]
\node at (0,1.5) {$\mcw(3,9)$};
\draw (130:1cm)--(150:.8cm)--(170:1cm);
\draw (250:1cm)--(270:.8cm)--(290:1cm);
\draw (10:1cm)--(30:.8cm)--(50:1cm);
\draw (30:.5cm)--(90:.5cm)--(150:.5cm)--(210:.5cm)--(270:.5cm)--(330:.5cm)--(30:.5cm);
\draw (90:.5cm)--(90:1cm);
\draw (330:.5cm)--(330:1cm);
\draw (210:.5cm)--(210:1cm);
\draw (150:.5cm)--(150:.8cm);
\draw (30:.5cm)--(30:.8cm);
\draw (270:.5cm)--(270:.8cm);
\filldraw[white] (90:.5cm) circle (0.1cm);
\draw (90:.5cm) circle (0.1cm);
\filldraw[white] (210:.5cm) circle (0.1cm);
\draw (210:.5cm) circle (0.1cm);
\filldraw[white] (330:.5cm) circle (0.1cm);
\draw (330:.5cm) circle (0.1cm);
\filldraw[black] (30:.5cm) circle (0.1cm);
\draw (30:.5cm) circle (0.1cm);
\filldraw[black] (150:.5cm) circle (0.1cm);
\draw (150:.5cm) circle (0.1cm);
\filldraw[black] (270:.5cm) circle (0.1cm);
\draw (270:.8cm) circle (0.1cm);
\filldraw[white] (30:.8cm) circle (0.1cm);
\draw (30:.8cm) circle (0.1cm);
\filldraw[white] (150:.8cm) circle (0.1cm);
\draw (150:.8cm) circle (0.1cm);
\filldraw[white] (270:.8cm) circle (0.1cm);
\draw (270:.8cm) circle (0.1cm);
%%%%%%%%%%%%%%%%
\begin{scope}[yshift = -2.5cm]
\draw (130:1cm)--(150:.8cm)--(170:1cm);
\draw (210:1cm)--(230:.8cm)--(250:1cm);
\draw (290:1cm)--(310:.8cm)--(330:1cm);
\draw (10:1cm)--(30:.8cm)--(50:1cm);
\draw (150:.8cm)--(190:.5cm);
\draw (230:.8cm)--(190:.5cm);
\draw (310:.8cm)--(-10:.5cm);
\draw (30:.8cm)--(-0:.5cm);
\draw (90:1cm)--(90:.65cm)--(190:.5cm);
\draw (90:.65cm)--(-10:.5cm);
\filldraw[white] (150:.8cm) circle (0.1cm);
\draw (150:.8cm) circle (0.1cm);
\filldraw[white] (230:.8cm) circle (0.1cm);
\draw (230:.8cm) circle (0.1cm);
\filldraw[white] (30:.8cm) circle (0.1cm);
\draw (30:.8cm) circle (0.1cm);
\filldraw[white] (310:.8cm) circle (0.1cm);
\draw (310:.8cm) circle (0.1cm);
\filldraw[white] (90:.65cm) circle (0.1cm);
\draw (90:.65cm) circle (0.1cm);
\filldraw[black] (190:.5cm) circle (0.1cm);
\draw (190:.5cm) circle (0.1cm);
\filldraw[black] (-10:.5cm) circle (0.1cm);
\draw (-10:.5cm) circle (0.1cm);
\end{scope}
%%%%%%%%%%%%%%%%%%%%%%%%%%%
\begin{scope}[yshift = -5cm]
\draw (-40:1cm)--(-80:.5cm);
\draw (-80:1cm)--(-80:.5cm);
\draw (240:1cm)--(-80:.5cm);
\draw (40:1cm)--(80:.5cm);
\draw (80:1cm)--(80:.5cm);
\draw (120:1cm)--(80:.5cm);
\draw (0:1cm)--(180:.5cm);
\draw (160:1cm)--(180:.5cm);
\draw (200:1cm)--(180:.5cm);
\filldraw[white] (180:.5cm) circle (0.1cm);
\draw (180:.5cm) circle (0.1cm);
\filldraw[white] (80:.5cm) circle (0.1cm);
\draw (80:.5cm) circle (0.1cm);
\filldraw[white] (280:.5cm) circle (0.1cm);
\draw (280:.5cm) circle (0.1cm);
\end{scope}
\end{tikzpicture}
\hspace{.5cm}&&\\
\hline
\end{tabular}
\caption{Duality between webs in small cases, drawn in terms of weblike subgraphs. The last case is on the next page. Edges of multiplicity two or three are depicted by doubled or tripled edges. \label{fig:dualitypics}}
\end{figure}
\end{center}

\newpage
\begin{center}
\begin{figure}
\begin{tabular}{|cc|}
\hline 
\hspace{.3cm}
\begin{tikzpicture}[scale = 1]
\node at (0,1.5) {$\mcw(2,10)$};
\draw (90:1cm)--(270:1cm);
\draw [rounded corners] (54:1cm)--(9:.2cm)--(306:1cm);
\draw [rounded corners] (18:1cm)--(0:.5cm)--(342:1cm);
\draw [rounded corners] (162:1cm)--(162+18:.5cm)--(198:1cm);
\draw [rounded corners] (126:1cm)--(126+45:.2cm)--(234:1cm);
\begin{scope}[yshift = -2.5cm]
\draw (90:1cm)--(270:1cm);
\draw [rounded corners] (54:1cm)--(9:.2cm)--(306:1cm);
\draw [rounded corners] (18:1cm)--(0:.5cm)--(342:1cm);
\draw [rounded corners] (234:1cm)--(198+18:.5cm)--(198:1cm);
\draw [rounded corners] (126:1cm)--(126+18:.5cm)--(162:1cm);
\end{scope}
\begin{scope}[yshift = -5cm]
%\node at (0,1.5) {$\mcw(2,10)$};
\draw (90:1cm)--(270:1cm);
\draw [rounded corners] (54:1cm)--(36:.5cm)--(18:1cm);
\draw [rounded corners] (306:1cm)--(306+18:.5cm)--(342:1cm);
\draw [rounded corners] (234:1cm)--(198+18:.5cm)--(198:1cm);
\draw [rounded corners] (126:1cm)--(126+18:.5cm)--(162:1cm);
\end{scope}
\begin{scope}[yshift = -7.5cm]
\draw [rounded corners] (90:1cm)--(130:.2cm)--(198:1cm);
\draw [rounded corners] (54:1cm)--(300:.2cm)--(306:1cm);
\draw [rounded corners] (234:1cm)--(234+18:.2cm)--(270:1cm);
\draw [rounded corners] (18:1cm)--(0:.5cm)--(342:1cm);
\draw [rounded corners] (126:1cm)--(126+18:.5cm)--(162:1cm);
\end{scope}
\begin{scope}[yshift= -10cm]
\node at (0,1.5) {\hfill};
\draw [rounded corners] (90+36+36+36:1cm)--(130+36+36+36:.2cm)--(198+36+36+36:1cm);
\draw [rounded corners] (126+36+36+36:1cm)--(126+18+36+36+36:.5cm)--(162+36+36+36:1cm);
\draw [rounded corners] (54+36+36+36:1cm)--(36+36+36+36:.5cm)--(18+36+36+36:1cm);
\draw [rounded corners] (342+36+36+36:1cm)--(306+18+36+36+36:.5cm)--(306+36+36+36:1cm);
\draw [rounded corners] (270+36+36+36:1cm)--(234+18+36+36+36:.5cm)--(234+36+36+36:1cm);
\end{scope}
\begin{scope}[yshift = -12.5cm]
\draw [rounded corners] (18:1cm)--(0:.5cm)--(342:1cm);
\draw [rounded corners] (306:1cm)--(306-18:.5cm)--(270:1cm);
\draw [rounded corners] (198:1cm)--(198+18:.5cm)--(234:1cm);
\draw [rounded corners] (126:1cm)--(126+18:.5cm)--(162:1cm);
\draw [rounded corners] (54:1cm)--(54+18:.5cm)--(90:1cm);
\end{scope}
\end{tikzpicture}
&
\vspace{.3cm}
\begin{tikzpicture}[scale = 1]
\node at (-2,0) {$\overset{\text{dual}}{\leftrightarrow}$};
\node at (0,1.5) {$\mcw(5,10)$};
\draw (18:1cm)--(90:.2cm)--(54:1cm);
\draw (90:1cm)--(90:.2cm)--(126:1cm);
\draw (162:1cm)--(90:.2cm);
\draw (342:1cm)--(270:.2cm)--(306:1cm);
\draw (234:1cm)--(270:.2cm)--(270:1cm);
\draw (198:1cm)--(270:.2cm);
\filldraw[white] (270:.2cm) circle (0.1cm);
\draw (270:.2cm) circle (0.1cm);
\filldraw[white] (90:.2cm) circle (0.1cm);
\draw (90:.2cm) circle (0.1cm);
\begin{scope}[yshift = -2.5cm]
\node at (-2,0) {$\overset{\text{dual}}{\leftrightarrow}$};
\draw (18:1cm)--(60:.5cm)--(54:1cm);
\draw (90:1cm)--(60:.5cm)--(126:1cm);
\draw (342:1cm)--(-60:.5cm)--(306:1cm);
\draw (234:1cm)--(-60:.5cm)--(270:1cm);
\draw (60:.5cm)--(180:.1cm)--(-60:.5cm);
\draw (162:1cm)--(180:.7cm)--(198:1cm);
\draw (-.7,.09cm)--(-.1,.09cm);
\draw (-.7,.0cm)--(-.1,.0cm);
\draw (-.7,-.09cm)--(-.1,-.09cm);
\filldraw[white] (60:.5cm) circle (0.1cm);
\draw (60:.5cm) circle (0.1cm);
\filldraw[white] (-60:.5cm) circle (0.1cm);
\draw (-60:.5cm) circle (0.1cm);
\filldraw[black] (180:.1cm) circle (0.1cm);
\draw (180:.1cm) circle (0.1cm);
\filldraw[white] (180:.7cm) circle (0.1cm);
\draw (180:.7cm) circle (0.1cm);
\end{scope}
\begin{scope}[yshift= -5cm]
\node at (-2,0) {$\overset{\text{dual}}{\leftrightarrow}$};
%\node at (0,1.5) {$\mcw(5,10)$};
\draw (54:1cm)--(90:.6cm)--(90:1cm);
\draw (126:1cm)--(90:.6cm);
\draw (306:1cm)--(270:.6cm)--(270:1cm);
\draw (234:1cm)--(270:.6cm);
\draw (162:1cm)--(180:.7cm)--(198:1cm);
\draw (18:1cm)--(0:.8cm)--(342:1cm);
\draw (180:.7cm)--(180:.2cm);
\draw (-.7cm,.09cm)--(-.2cm,.09cm);
\draw (-.7cm,.-.09cm)--(-.2cm,.-.09cm);
\draw (0:.7cm)--(0:.2cm);
\draw (.7cm,.09cm)--(.2cm,.09cm);
\draw (.7cm,.-.09cm)--(.2cm,.-.09cm);
\draw (90:.6cm)--(180:.2cm)--(270:.6cm)--(0:.2cm)--(90:.6cm);
\filldraw[white] (270:.6cm) circle (0.1cm);
\draw (270:.6cm) circle (0.1cm);
\filldraw[white] (90:.6cm) circle (0.1cm);
\draw (90:.6cm) circle (0.1cm);
\filldraw[black] (180:.2cm) circle (0.1cm);
\draw (180:.2cm) circle (0.1cm);
\filldraw[black] (0:.2cm) circle (0.1cm);
\draw (0:.2cm) circle (0.1cm);
\filldraw[white] (180:.7cm) circle (0.1cm);
\draw (180:.7cm) circle (0.1cm);
\filldraw[white] (0:.7cm) circle (0.1cm);
\draw (0:.7cm) circle (0.1cm);
\end{scope}
\begin{scope}[yshift = -10cm]
\node at (-2,0) {$\overset{\text{dual}}{\leftrightarrow}$};
\draw (18:1cm)--(36:.7cm)--(54:1cm);
\draw (18+72:1cm)--(36+72:.7cm)--(54+72:1cm);
\draw (162:1cm)--(198:.7cm)--(234:1cm);
\draw (198:1cm)--(198:.7cm);
\draw (162+108:1cm)--(198+108:.7cm)--(234+108:1cm);
\draw (198+108:1cm)--(198+108:.7cm);
\draw (-.21cm,.665cm)--(-.09cm,.285cm);
\draw (-.14cm,.665cm)--(-.02cm,.285cm);
\draw (-.28cm,.665cm)--(-.16cm,.285cm);
\draw (.41cm,.-.665cm)--(.14cm,.-.285cm);
\draw (.53cm,-.665cm)--(.26cm,-.285cm);
\draw (198:.7cm)--(108:.3cm);
\draw (198:.7cm)--(306:.3cm);
\draw (36:.7cm)--(108:.3cm);
\draw (36:.7cm)--(306:.28cm);
\draw (30:.7cm)--(306:.36cm);
\filldraw[white] (306:.7cm) circle (0.1cm);
\draw (306:.7cm) circle (0.1cm);
\filldraw[white] (198:.7cm) circle (0.1cm);
\draw (198:.7cm) circle (0.1cm);
\filldraw[white] (36:.7cm) circle (0.1cm);
\draw (36:.7cm) circle (0.1cm);
\filldraw[white] (108:.7cm) circle (0.1cm);
\draw (108:.7cm) circle (0.1cm);
\filldraw[black] (108:.3cm) circle (0.1cm);
\draw (108:.3cm) circle (0.1cm);
\filldraw[black] (306:.3cm) circle (0.1cm);
\draw (306:.3cm) circle (0.1cm);
\end{scope}
\begin{scope}[yshift = -7.5cm]
\node at (-2,0) {$\overset{\text{dual}}{\leftrightarrow}$};
\node at (0,1.5) {\hfill};
\draw (18:1cm)--(70:.5cm)--(54:1cm);
\draw (90:1cm)--(70:.5cm)--(126:1cm);
\draw (162:1cm)--(198:.6cm)--(198:1cm);
\draw (234:1cm)--(198:.6cm);
\draw (342:1cm)--(306:.6cm)--(306:1cm);
\draw (270:1cm)--(306:.6cm);
\draw (70:.5cm)--(0,0);
\draw (-.57cm,-.22cm)--(-.09cm,-.07cm);
\draw (-.57cm,-.12cm)--(-.09cm,.03cm);
\draw (.3cm,-.48cm)--(-.05cm,0cm);
\draw (.42cm,-.48cm)--(.09cm,0cm);
\filldraw[white] (70:.5cm) circle (0.1cm);
\draw (70:.5cm) circle (0.1cm);
\filldraw[white] (306:.6cm) circle (0.1cm);
\draw (306:.6cm) circle (0.1cm);
\filldraw[white] (198:.6cm) circle (0.1cm);
\draw (198:.6cm) circle (0.1cm);
\filldraw[black] (0,0) circle (0.1cm);
\draw (0,0) circle (0.1cm);
\end{scope}
\begin{scope}[yshift = -12.5cm]
\node at (-2,0) {$\overset{\text{dual}}{\leftrightarrow}$};
\draw (90:1cm)--(-.4cm,.35cm)--(126:1cm);
\draw (306:1cm)--(-.1cm,0cm)--(342:1cm);
\draw (234:1cm)--(234+18:.8cm)--(270:1cm);
\draw (18:1cm)--(.45,.45cm)--(54:1cm);
\draw (162:1cm)--(-.6cm,-.35cm)--(198:1cm);
\draw (0,.35)--(.45,.45);
\draw (.07,.28)--(.52,.38);
\draw (-.07,.42)--(.38,.52);
\draw (-.247cm,-.76cm)--(-.2,-.35);
\draw (-.177cm,-.76cm)--(-.13,-.35);
\draw (-.317cm,-.76cm)--(-.27,-.35);
\draw (-.45,.35)--(-.65,-.35);
\draw (-.35,.35)--(-.55,-.35);
\draw (-.4,.35)--(0,.35)--(-.1,0)--(-.2,-.35)--(-.6,-.35);
\draw (-.5,.0)--(-.1,0);
\filldraw[white] (.45,.45cm) circle (0.1cm);
\draw (.45,.45cm) circle (0.1cm);
\filldraw[white] (234+18:.8cm) circle (0.1cm);
\draw (234+18:.8cm) circle (0.1cm);
\filldraw[white] (-.4cm,.35cm) circle (0.1cm);
\draw (-.4cm,.35cm) circle (0.1cm);
\filldraw[black] (-.5cm,0cm) circle (0.1cm);
\draw (-.5cm,0cm) circle (0.1cm);
\filldraw[white] (-.6cm,-.35cm) circle (0.1cm);
\draw (-.6cm,-.35cm) circle (0.1cm);
\filldraw[black] (0cm,.35cm) circle (0.1cm);
\draw (0cm,.35cm) circle (0.1cm);
\filldraw[white] (-.1cm,0cm) circle (0.1cm);
\draw (-.1cm,0cm) circle (0.1cm);
\filldraw[black] (-.2cm,-.35cm) circle (0.1cm);
\draw (-.2cm,-.35cm) circle (0.1cm);
\end{scope}
\end{tikzpicture} 
\vspace{-.3cm}
\tabularnewline\hline
\end{tabular}
\end{figure}
\end{center}

%\nocite{*}

\bibliographystyle{abbrvnat}
% use the following instead if you encounter problems 
%\bibliographystyle{alpha}
\bibliography{DimersWebsBib}

\begin{thebibliography}{19}
\providecommand{\natexlab}[1]{#1}
\providecommand{\url}[1]{\texttt{#1}}
\expandafter\ifx\csname urlstyle\endcsname\relax
  \providecommand{\doi}[1]{doi: #1}\else
  \providecommand{\doi}{doi: \begingroup \urlstyle{rm}\Url}\fi

\bibitem[Cautis et~al.((2014))Cautis, Kamnitzer, and Morrison]{SkewHowe}
S.~Cautis, J.~Kamnitzer, and S.~Morrison.
\newblock Webs and quantum skew {H}owe duality.
\newblock \emph{Math. Ann.}, 360:\penalty0 351--390, (2014).

\bibitem[Fomin and Pylyavskyy((2014))]{TensorsII}
S.~Fomin and P.~Pylyavskyy.
\newblock Webs on surfaces, rings of invariants, and clusters.
\newblock \emph{Proc. Natl. Acad. Sci}, 111\penalty0 (27):\penalty0 9680--9687,
  (2014).

\bibitem[Fomin and Pylyavskyy((2016))]{Tensors}
S.~Fomin and P.~Pylyavskyy.
\newblock Tensor diagrams and cluster algebras.
\newblock \emph{Adv. Math.}, 300:\penalty0 717--787, (2016).

\bibitem[Goodearl and Yakimov((2009))]{Goya}
K.~R. Goodearl and M.~Yakimov.
\newblock Poisson structures on affine spaces and flag varieties. ii. general
  case.
\newblock \emph{Trans. Amer. Math. Soc.}, 361\penalty0 (11):\penalty0
  5753--5780, (2009).

\bibitem[Khovanov and Kuperberg((1999))]{KhovanovKuperberg}
M.~Khovanov and G.~Kuperberg.
\newblock Webs bases for {\it sl}(3) are not dual canonical.
\newblock \emph{Pacific J. Math.}, 188\penalty0 (1):\penalty0 129--153, (1999).

\bibitem[Kim((2003))]{KimThesis}
D.~Kim.
\newblock Graphical calculus on representations of quantum {L}ie algebras.
\newblock \emph{Ph.D thesis, UC Davis}, (2003).

\bibitem[Knutson et~al.((2013))Knutson, Lam, and Speyer]{KLS}
A.~Knutson, T.~Lam, and D.~Speyer.
\newblock Positroid varieties: juggling and geometry.
\newblock \emph{Compos. Math}, 149\penalty0 (10):\penalty0 1710--1752, (2013).

\bibitem[Kuo((2004))]{Kuo}
E.~H. Kuo.
\newblock Applications of graphical condensation for enumerating matchings and
  tilings.
\newblock \emph{Theoret. Comput. Sci.}, 319\penalty0 (1--3):\penalty0 29--57,
  (2004).

\bibitem[Kuperberg((1996))]{Kuperberg}
G.~Kuperberg.
\newblock Spiders for rank 2 {L}ie algebras.
\newblock \emph{Commun. Math. Phys.}, 1\penalty0 (180):\penalty0 109--151,
  (1996).

\bibitem[Lam((2015))]{LamDimers}
T.~Lam.
\newblock Dimers, webs, and positroids.
\newblock \emph{J. Lond. Math. Soc.}, 92\penalty0 (3):\penalty0 633--656,
  (2015).

\bibitem[Lam((2016))]{LamNotesII}
T.~Lam.
\newblock Totally nonnegative {G}rassmannian and {G}rassmann polytopes.
\newblock \emph{Current developments in mathematics 2014}, pages 51–--152,
  (2016).

\bibitem[Lusztig((1992))]{Lusztig}
G.~Lusztig.
\newblock Canonical bases in tensor products.
\newblock \emph{Proc. Nat. Acad. Sci. U.S.A}, 89\penalty0 (17):\penalty0
  8177--8179, (1992).

\bibitem[Morrison((2007))]{MorrisonThesis}
S.~Morrison.
\newblock A diagrammatic catgeory for the representation theory of
  ${U}_q(\mathfrak{s}\mathfrak{l}_n)$.
\newblock \emph{Ph.D thesis, UC Berkeley}, (2007).

\bibitem[Oh((2011))]{OhSchubert}
S.~Oh.
\newblock Positroids and {S}chubert matroids.
\newblock \emph{J. Combin. Theory Ser. A}, 118\penalty0 (8):\penalty0
  2426--2435, (2011).

\bibitem[Peterson et~al.((2009))Peterson, Pylyavksyy, and Rhoades]{A2Promotion}
K.~T. Peterson, P.~Pylyavksyy, and B.~Rhoades.
\newblock Promotion and cyclic sieving via webs.
\newblock \emph{J. Algebraic Comb.}, 30\penalty0 (1):\penalty0 19--41, (2009).

\bibitem[Postnikov((2006))]{Postnikov}
A.~Postnikov.
\newblock Total positivity, {G}rassmannians, and {N}etworks.
\newblock \emph{arXiv preprint}, (2006).
\newblock \texttt{arXiv:0609764} [math.CO].

\bibitem[Postnkov et~al.((2009))Postnkov, Williams, and Speyer]{PSW}
A.~Postnkov, L.~Williams, and D.~Speyer.
\newblock Matching polytopes, toric geometry, and the totally non-negative
  {G}rassmannian.
\newblock \emph{J.~Algebraic Combin.}, 30\penalty0 (2):\penalty0 173--191,
  (2009).

\bibitem[Tymoczko((2012))]{Tymoczko}
J.~Tymoczko.
\newblock A simple bijection between standard $3\times n$ tableaux and
  irreducible webs for $\mathfrak{s}\mathfrak{l}_3$.
\newblock \emph{J. Algebraic Combin.}, 35\penalty0 (4):\penalty0 611--632,
  (2012).

\bibitem[Westbury((2012))]{Westbury}
B.~Westbury.
\newblock Web bases for the general linear groups.
\newblock \emph{J. Algebraic Combin.}, 35\penalty0 (1):\penalty0 5753--5780,
  (2012).

\end{thebibliography}
\label{secn:biblio}

\end{document}